\documentclass[10pt,reqno]{amsart} 
\usepackage[10pt,nona4]{optional} 

\usepackage{dsfont,amsmath,amsfonts,amscd,amssymb,graphicx,mathrsfs,eufrak,upgreek}
\usepackage[all,arc,curve,color,frame]{xy}
\usepackage[usenames]{color}
\usepackage{setspace}
\usepackage{array,multirow,booktabs,longtable}

\newcommand{\spacebelowtable}{2ex}

\newcommand{\marginparstretch}{0.6}
\let\oldmarginpar\marginpar
\renewcommand\marginpar[1]{\-\oldmarginpar[\framebox{\setstretch{\marginparstretch}\begin{minipage}{\marginparwidth}{\raggedleft\tiny #1}\end{minipage}}]{\framebox{\setstretch{\marginparstretch}\begin{minipage}{\marginparwidth}{\raggedright\tiny #1}\end{minipage}}}}

\usepackage[colorlinks]{hyperref}
\usepackage{tikz,mathrsfs}
\usetikzlibrary{arrows,shapes.geometric,calc,patterns}

        \tikzset{
        cvertex/.style={circle,draw=black,inner sep=1pt,outer sep=3pt},
        vertex/.style={circle,fill=black,inner sep=1pt,outer sep=3pt},
        DBs/.style={circle,draw=black,circle,fill=black,inner sep=0pt, minimum size=3pt},
        DB/.style={circle,draw=black,circle,fill=black,inner sep=0pt, minimum size=4pt},
         DWs/.style={circle,draw=black,circle,fill=white,inner sep=0pt, minimum size=3pt},
         DWds/.style={circle,draw=black,densely dotted,circle,fill=white,inner sep=0pt, minimum size=3pt},
        DW/.style={circle,draw=black,inner sep=0pt, minimum size=4pt},
        tvertex/.style={inner sep=1pt,font=\scriptsize},
        gap/.style={inner sep=0.5pt,fill=white},
        Ggap/.style={inner sep=0.5pt,fill=green!40!black!20}}

\tikzstyle{mybox} = [draw=black, fill=blue!10, very thick,
    rectangle, rounded corners, inner sep=10pt, inner ysep=20pt]
\tikzstyle{boxtitle} =[fill=blue!50, text=white,rectangle,rounded corners]

\newcommand{\arrow}[2][20]
 {
  \hspace{-5pt}
  \begin{tikzpicture}
   \node (A) at (0,0) {};
   \node (B) at (#1pt,0) {};
   \draw [#2] (A) -- (B);
  \end{tikzpicture}
  \hspace{-5pt}
 }

\newcommand{\birational}[1][20]{\arrow[#1]{->,dashed}}

\newtheorem{thm}{Theorem}[section]
\newtheorem{prop}[thm]{Proposition}
\newtheorem{lemma}[thm]{Lemma}
\newtheorem{defin}[thm]{Definition}
\newtheorem{cor}[thm]{Corollary}

\theoremstyle{definition} 
\newtheorem{warning}[thm]{Warning}
\newtheorem{example}[thm]{Example}

\newtheorem{remark}[thm]{Remark}
\newtheorem{conj}[thm]{Conjecture}

\newtheorem{notation}[thm]{Notation}

\numberwithin{equation}{section}

\newcounter{tempenum}

\newcommand{\m}{\mathfrak{m}}
\newcommand{\n}{\mathfrak{n}}
\newcommand{\p}{\mathfrak{p}}

\renewcommand{\t}[1]{\textnormal{#1}}

\def\op{\mathop{\rm op}\nolimits}
\def\rest{\mathop{\rm res}\nolimits}

\def\CM{\mathop{\rm CM}\nolimits}
\def\uCM{\mathop{\underline{\rm CM}}\nolimits}

\def\mod{\mathop{\rm mod}\nolimits}

\def\coh{\mathop{\rm coh}\nolimits}
\def\Qcoh{\mathop{\rm Qcoh}\nolimits}
\def\Mod{\mathop{\rm Mod}\nolimits}
\def\Rad{\mathop{\rm Rad}\nolimits}
\def\refl{\mathop{\rm ref}\nolimits}
\def\proj{\mathop{\rm proj}\nolimits}

\def\pd{\mathop{\rm pd}\nolimits}
\def\id{\mathop{\rm inj.dim}\nolimits}
\def\uHom{\mathop{\underline{\rm Hom}}\nolimits}
\def\Hom{\mathop{\rm Hom}\nolimits}

\def\RHom{\mathop{\rm {\bf R}Hom}\nolimits}

\def\End{\mathop{\rm End}\nolimits}
\def\Ext{\mathop{\rm Ext}\nolimits}

\def\add{\mathop{\rm add}\nolimits}
\def\Cok{\mathop{\rm Cok}\nolimits}
\def\Ker{\mathop{\rm Ker}\nolimits}

\def\rank{\mathop{\rm rank}\nolimits}
\def\Im{\mathop{\rm Im}\nolimits}

\def\Supp{\mathop{\rm Supp}\nolimits}

\def\Spec{\mathop{\rm Spec}\nolimits}
\def\sup{\mathop{\rm sup}\nolimits}
\def\width{\mathop{\rm wid}\nolimits}
\def\cwidth{\mathop{\rm cwid}\nolimits}
\def\Max{\mathop{\rm Max}\nolimits}

\def\Perf{\mathop{\rm{per}}\nolimits}
\def\pt{\mathop{\rm{pt}}\nolimits}

\def\D{\mathop{\rm{D}^{}}\nolimits}

\def\Db{\mathop{\rm{D}^b}\nolimits}
\def\Kb{\mathop{\rm{K}^b}\nolimits}

\def\flopflop{{\sf{FF}}}
\def\flop{{\sf{F}}}
\def\flopback{{\sf{F'}}}
\def\Id{\mathop{\rm{Id}}\nolimits}

\newcommand{\K}{\mathop{{}_{}\mathbb{C}}\nolimits}

\newcommand{\nef}{\mathrm{nef}}
\newcommand{\deform}{\mathrm{def}}
\newcommand{\DA}{\mathrm{A}_{\deform}}
\newcommand{\con}{\mathrm{con}}
\newcommand{\CA}{\mathrm{A}_{\con}}
\newcommand{\CB}{\mathrm{B}_{\con}}
\newcommand{\CAab}{\mathrm{A}^{\ab}_{\con}}
\newcommand{\AB}{\mathrm{A}}
\newcommand{\BB}{\mathrm{B}}
\newcommand{\FCA}{\mathbb{F}\CA}

\newcommand{\cEU}{\mathcal{E}_U}
\newcommand{\cFU}{\mathcal{F}_U}

\def\RA{\mathop{\rm RA}\nolimits}
\def\LA{\mathop{\rm LA}\nolimits}
\def\ab{\mathop{\rm ab}\nolimits}
\def\redu{\mathop{\rm red}\nolimits}
\def\inc{\mathop{\rm inc}\nolimits}

\def\Rp{{\rm\bf R}p}
\def\Rf{{\rm\bf R}f}

\def\Ri{{\rm\bf R}i}

\def\RGamma{{\rm\bf R}\Gamma}
\def\RHom{{\rm{\bf R}Hom}}
\def\RsHom{{\bf R}\mathcal{H}om}
\newcommand\RDerived[1]{{\rm\bf R}{#1}}
\newcommand\RDerivedi[2]{{\rm\bf R}^{#1}{#2}}
\newcommand\Rfi[1]{{\rm\bf R}^{#1}f}

\newcommand\FM[1]{\operatorname{\sf{FM}}({#1})}

\newcommand\art{\mathsf{Art}}
\newcommand\cart{\mathsf{CArt}}
\newcommand\alg{\mathsf{Alg}}
\newcommand\calg{\mathsf{CAlg}}
\newcommand\Sets{\mathsf{Sets}}
\newcommand\cDef{c\mathcal{D}ef}
\newcommand\Def{\mathcal{D}ef}

\newcommand{\cA}{\mathcal{A}}

\newcommand{\cC}{\mathcal{C}}
\newcommand{\cD}{\mathcal{D}}
\newcommand{\cE}{\mathcal{E}}
\newcommand{\cF}{\mathcal{F}}
\newcommand{\cG}{\mathcal{G}}

\newcommand{\cI}{\mathcal{I}}

\newcommand{\cL}{\mathcal{L}}
\newcommand{\cM}{\mathcal{M}}
\newcommand{\cN}{\mathcal{N}}
\newcommand{\cO}{\mathcal{O}}

\newcommand{\cQ}{\mathcal{Q}}

\newcommand{\cV}{\mathcal{V}}
\newcommand{\cW}{\mathcal{W}}

\newcommand{\Per}{\mathop{{}^{0}\mathrm{Per}}\nolimits}
\newcommand{\mPer}{\mathop{{}^{-1}\mathrm{Per}}\nolimits}

\newcommand{\TwistVsFlop}{\Psi}

\newcommand\PairsCat[2]{\operatorname{Pairs}(#1,#2)}
\newcommand\LargeFamily[2]{\big(#1,\,#2\big)}

\newlength\tempWidth
\newcommand\InSpaceOf[2]{
   \settowidth{\tempWidth}{$#1$}
   \phantom{#1}
   \hspace{-\tempWidth}
   {#2}}

\newcommand{\defColor}{gray!50}
\newcommand{\defBGColor}{gray!10}
\newcommand{\defGridColor}{gray}
  \def\size{5}
  \def\bulge{1/3}
\def\gridShift{-4}
\def\bend{0.1}
\def\len{0.8}
\def\dotsize{0.1}

\newcommand{\defCurveCoords}[3]{
  (#1+\bend,-#3,#2) .. controls (#1-\bend,-#3/3,#2) and (#1-\bend, #3/3,#2) .. (#1+\bend,#3,#2)
}

\newcommand{\defDrawDot}[2]{
    \draw[fill] (#1+\dotsize,\gridShift,#2+\dotsize) -- (#1+\dotsize,\gridShift,#2-\dotsize) -- (#1-\dotsize,\gridShift,#2-\dotsize) -- (#1-\dotsize,\gridShift,#2+\dotsize) -- cycle
    }

\newcommand{\defCurve}[3]{
  \draw[color=#3] \defCurveCoords{#1}{#2}{\len}
}

\newcommand{\defThreeFold}[3]{
  \draw[color=#3] (#1+0.5+\size,#2-0.3) .. controls (#1+0.5+2*\size*\bulge,#2-0.3+\size/2) and (#1+0.5-\size*\bulge,#2-0.3+\size/2) .. (#1+0.5-\size,#2-0.3) .. controls (#1+0.5-5*\size*\bulge,#2-0.3-\size/2) and (#1+0.5+4*\size*\bulge,#2-0.3-\size/2) .. (#1+0.5+\size,#2-0.3)
}

\newcommand{\defGridLines}[4]{
  \def\x{#1} \def\y{#2}
  \def\numx{#3} \def\numy{#4}
  \foreach \i in {0,...,\numy} {
    \draw[color=\defGridColor] (\x,\gridShift,\y+\i) -- +(\numx,0,0);
  };
  \foreach \i in {0,...,\numx} {
    \draw[color=\defGridColor] (\x+\i,\gridShift,\y) -- +(0,0,\numy);
  };
  \draw[dotted] (\x,0.35*\gridShift,\y) -- (\x,0.9*\gridShift,\y)
}

\newcommand{\defUniversal}[4]{
  \def\x{#1} \def\y{#2}
  \def\numx{#3} \def\numy{#4}
  \def\zero{0}
  \draw[color=\defBGColor,fill=\defBGColor] \defCurveCoords{\x}{\y}{\len} -- \defCurveCoords{\x}{\y+\numx}{-\len} -- cycle;
  \draw[color=\defBGColor,fill=\defBGColor] \defCurveCoords{\x}{\y}{\len} -- \defCurveCoords{\x+\numy}{\y}{-\len} -- cycle;
  \draw[color=\defBGColor,fill=\defBGColor] \defCurveCoords{\x+\numy}{\y}{\len} -- \defCurveCoords{\x+\numy}{\y+\numx}{-\len} -- cycle;
  \draw[color=\defBGColor,fill=\defBGColor] \defCurveCoords{\x}{\y+\numx}{\len} -- \defCurveCoords{\x+\numy}{\y+\numx}{-\len} -- cycle;
  \defUniversalCurves{#1}{#2}{#3}{#4}
  \foreach \i in {1,...,\numy} {
    \foreach \j in {1,...,\numx} {
      \defDrawDot{\x+\i}{\y+\j};
    };
  }
}

\newcommand{\defUniversalComm}[4]{
  \def\x{#1} \def\y{#2}
  \def\numx{#3} \def\numy{#4}
  \def\zero{0}
  \draw[color=\defBGColor,fill=\defBGColor] \defCurveCoords{\x}{\y}{\len} -- \defCurveCoords{\x}{\y+\numx}{-\len} -- cycle;
  \draw[color=\defBGColor,fill=\defBGColor] \defCurveCoords{\x}{\y}{\len} -- \defCurveCoords{\x+\numy}{\y}{-\len} -- cycle;
  \defUniversalCurves{#1}{#2}{#3}{#4}
}

\newcommand{\defUniversalCurves}[4]{
  \foreach \i in {0,...,#4} {
    \ifx\i\zero \defCurve{#1+\i}{#2}{black} \else \defCurve{#1+\i}{#2}{\defColor} \fi;
    \defDrawDot{#1+\i}{#2};
  };
  \foreach \j in {1,...,#3} {
    \defCurve{#1}{#2+\j}{\defColor};
    \defDrawDot{#1}{#2+\j};
  }
}

\begin{document}
\title{\textsc{Noncommutative deformations and flops}}
\author{Will Donovan}
\address{Will Donovan, The Maxwell Institute, School of Mathematics, James Clerk Maxwell Building, The King's Buildings, Mayfield Road, Edinburgh, EH9 3JZ, UK.}
\email{Will.Donovan@ed.ac.uk}
\author{Michael Wemyss}
\address{Michael Wemyss, The Maxwell Institute, School of Mathematics, James Clerk Maxwell Building, The King's Buildings, Mayfield Road, Edinburgh, EH9 3JZ, UK.}
\email{wemyss.m@googlemail.com}
\begin{abstract}
We prove that the functor of noncommutative deformations of every flipping or flopping irreducible rational curve in a 3-fold is representable, and hence associate to every such curve a noncommutative deformation algebra $\CA$. This new invariant extends and unifies known invariants for flopping curves in 3-folds, such as the width of Reid \cite{Pagoda}, and the bidegree of the normal bundle. It also applies in the settings of flips and singular schemes. We show that the noncommutative deformation algebra $\CA$ is finite dimensional, and give a new way of obtaining the commutative deformations of the curve, allowing us to make explicit calculations of these deformations for certain $(-3,1)$-curves.

We then show how our new invariant $\CA$ also controls the homological algebra of flops. For any flopping curve in a projective 3-fold with only Gorenstein terminal singularities, we construct an autoequivalence of the derived category of the 3-fold by twisting around a universal family over the noncommutative deformation algebra $\CA$, and prove that this autoequivalence is an inverse of Bridgeland's flop--flop functor. This demonstrates that it is strictly necessary to consider noncommutative deformations of curves in order to understand the derived autoequivalences of a 3-fold, and thus the Bridgeland stability manifold.

\end{abstract}
\thanks{The first author was supported by the EPSRC grant~EP/G007632/1 and by the Erwin Schr\"odinger Institute, Vienna, during the course of this work. The second author was supported by EPSRC grant~EP/K021400/1.}
\maketitle
\parindent 20pt
\parskip 0pt

\tableofcontents

\section{Introduction}

To understand the birational geometry of algebraic varieties via the minimal model program, it is necessary to understand the geometry of certain codimension two modifications known as {\em flips} and {\em flops}. Even for modifications of irreducible rational curves in dimension three this geometry is extraordinarily rich, and has been discussed by many authors, for example  \cite{Pinkham,Pagoda,ClemensKollarMori,KM,KollarMori,Kawa}. A central problem is to classify flips and flops in a satisfying manner, and to construct appropriate invariants.

In the first half of this paper we associate a new invariant to every flipping or flopping curve in a $3$-fold, using noncommutative deformation theory.   This generalises and unifies the classical invariants into one new object, the {\em noncommutative deformation algebra} $\CA$ associated to the curve.  Our new invariant is a finite dimensional algebra, and can be associated to any contractible rational curve in any $3$-fold, regardless of singularities.  It recovers classical invariants in natural ways.  Moreover, unlike these classical invariants, $\CA$ is an algebra, and in the second half of the paper we exploit this to give the first intrinsic description of a derived autoequivalence associated to a general flopping curve.

\subsection{Background on $3$-folds}
\label{sect intro background}

The simplest example of a $3$-fold flop is the well-known Atiyah flop. In this case the flopping curve has normal bundle $\cO(-1)\oplus\cO(-1)$ and is rigid, and the flop may be factored
 as a blow-up of the curve followed by a blow-down.  However, for more general flops the curve is no longer rigid, and the factorization of the birational map becomes more complicated.  For a general flopping irreducible rational curve $C$ in a smooth 3-fold $X$, the following classical invariants are associated to $C$.

\begin{enumerate}
\item {\bf Normal bundle.} Denoting $\cO(a)\oplus \cO(b)$ as simply $(a,b)$, the normal bundle must be $(-1,-1)$, $(-2,0)$ or $(-3,1)$ \cite[Prop.~2]{Pinkham}.
\item {\bf Width \cite[5.3]{Pagoda}.}  In the first two cases above, the width of $C$ is defined to be
\[
\mbox{sup}\big\{ n\mid \exists\mbox{ a scheme } C_n\cong C\times \Spec \mathbb{C}[x]/x^n \mbox{ with } C\subset C_n\subset X \big\}  .
\]
More generally, in the case $(-3,1)$, it is not possible to define the width invariant in this way.
\item {\bf Dynkin type \cite{KM,Kawa}.} By taking a generic hyperplane section, Katz--Morrison and Kawamata  assign to $C$ the data of a Dynkin diagram with a marked vertex. Only some possibilities arise, as shown in Table~\ref{table compare invts} and \eqref{Dynkin marked}.
\item {\bf Length \cite[p95--96]{ClemensKollarMori}.} This is defined to be the multiplicity of the curve $C$ in the fundamental cycle of a generic hyperplane section.
\item {\bf Normal bundle sequence \cite[Thm.~4]{Pinkham}.} The flop $f\colon X \birational X'$ factors into a sequence of blow-ups in centres $C_1,\ldots,C_n$, followed by blow-downs. The normal bundles of these curves form the $\cN$-sequence.
\setcounter{tempenum}{\theenumi}
\end{enumerate}

\opt{10pt}{The following table }\opt{12pt}{Table~\ref{table compare invts} }summarises the relations between the invariants above, and shows that none of these invariants classify all analytic equivalence types of flopping curves.

\newcommand{\seqspacing}{\,}
\begin{table}[!htbp]
\opt{12pt}{\hspace*{-8mm}}
\begin{tabular}{*5c}
\toprule
 {\scriptsize (1)} & {\scriptsize (2)}  & {\scriptsize (3)}& {\scriptsize (4)} & {\scriptsize (5)} \\
 {\bf $\cN_{C|X}$} & {\bf width} & {\bf Dynkin type} & {\bf \: length \:} & {\bf $\cN$-sequence}    \\
\midrule
 $\,(-1,-1)\,$   &  1    & \multirow{2}*{$A_1$}  & \multirow{2}*{1}     & $(-1,-1)$ \\\cmidrule(l){1-2} \cmidrule(l){5-5}
 $(-2,0)$   &   $n>1$   & & & {\scriptsize $\overbrace{(-2,0) \ldots (-2,0)\seqspacing(-1,-1)}^n$ } \\\cmidrule(l){1-5}
\multirow{5}*{$(-3,1)$}  &  \multirow{5}*{\small undefined}  & $D_4$&   2  & {\scriptsize $(-3,1)\seqspacing(-2,-1)\seqspacing(-1,-1)$} \\\cmidrule(l){3-5}
&   & $E_6$ & 3 & \multirow{3}*{\scriptsize $(-3,1)\seqspacing(-3,0)\seqspacing(-2,-1)\seqspacing(-1,-1)$ } \\\cmidrule(l){3-4}
 &   &  $E_7$  & 4 & \\\cmidrule(l){3-4}
 &   &  $E_8(6)$ & 6 &  \\\cmidrule(l){3-5}
 & & $E_8(5)$ & 5 &  {\scriptsize $(-3,1)\seqspacing(-3,0)\seqspacing(-3,0)\seqspacing(-2,-1)\seqspacing(-1,-1)$ }   \\
\bottomrule
\end{tabular}\\[\spacebelowtable]
\caption{Classical invariants of flopping rational curves in smooth 3-folds.}\label{table compare invts}
\end{table}
\opt{10pt,a4}{\vspace{-5mm}}

\subsection{Noncommutative deformations and invariants}
\label{sect intro noncomm def}
Our new invariant of flopping and flipping rational curves $C$ in a 3-fold $X$ is constructed by noncommutatively deforming the associated sheaf $E:=\cO_C(-1)$. Classically, infinitesimal deformations of $C$ are controlled by $\Ext^1_X(E,E)$, the dimension of which is determined by the normal bundle $\cN_{C|X}$.  If $\dim_{\K} \Ext^1_X(E,E)\leq 1$, the curve $C$ deforms over an artinian base $\K[x]/x^{n}$.  However, when $\dim_{\K} \Ext^1_X(E,E)\geq 2$, deformations of $C$ are less well understood, and for example in the setting of flops the width invariant of Reid is no longer sufficient to characterize them.

In this paper, we solve this problem by using noncommutative deformation theory.  We give preliminary definitions here, leaving details to \S\ref{NCdef section}. Recall that the formal commutative deformations of the simple sheaf $E$ associated to $C$ are described by a functor from commutative local artinian $\K$-algebras to sets, given by
\begin{align*}
\cDef_E\colon \cart & \to \Sets \\
R & \mapsto \left. \{ \text{flat $R$-families of coherent sheaves deforming $E$} \} \middle/ \sim \right. .
\end{align*}
In the above we take the point of view that a family is an $R$-module in $\coh X$, as this turns out to generalise.  Following \cite{Laudal,Eriksen}, we may then define  a {\em noncommutative deformation functor} $\Def_E\colon \art_1 \to \Sets$ from $1$-pointed artinian not-necessarily-commutative $\K$-algebras to sets, which extends $\cDef_E$ (for details see \ref{NC deformation definition}).  Our first main theorem is the following.

\begin{thm}[=\ref{chase couples 1}\eqref{chase couples 1 2}]\label{main thm representability} For a flopping or flipping curve $C$ in a 3-fold $X$,  the noncommutative deformation functor $\Def_E$ is representable.
\end{thm}
As a consequence of \ref{main thm representability}, from the noncommutative deformation theory we obtain a universal family $\cE\in \coh X$.  At the same time, $\cDef_E$ is also representable, giving another universal family $\cF\in\coh X$.    The following are our new invariants.

\begin{enumerate}
\setcounter{enumi}{\thetempenum}
\item {\bf Noncommutative deformation algebra.} This is defined  \opt{12pt}{to be }$\DA:=\End_X(\cE)$.
\item {\bf Commutative deformation algebra.} We define this to be $\End_X(\cF)$, and prove in \ref{contraction alg isos} that it is isomorphic to the {\em abelianization} $\DA^{\ab}$, given by quotienting $\DA$ by the ideal generated by commutators.
\setcounter{tempenum}{\theenumi}
\end{enumerate}

The above objects are algebras, and not just vector spaces.  Nevertheless, taking dimensions we obtain the following new numerical invariants. 

\begin{list}{(${\theenumi}$)}{\usecounter{enumi}}
\setcounter{enumi}{\thetempenum}
\item {\bf Noncommutative width.} We take $\width(C) := \dim_{\K} \DA$.
\item {\bf Commutative width.} We set $\cwidth(C) := \dim_{\K} \DA^{\ab}$.
\setcounter{tempenum}{\theenumi}
\end{list}

For flopping and flipping curves these numerical invariants are always finite by \ref{fd intro}, regardless of singularities.  We explain in \ref{can calculate width} that the noncommutative width $\width(C)$ is readily calculated, and is indeed easier to handle than the commutative width $\cwidth(C)$. From the viewpoint of the homological algebra, we will see later in \S\ref{nec of def} that the noncommutative width is the more natural invariant.

For floppable curves of Type $A$, when Reid's width invariant is defined \cite{Pagoda}, we prove in \ref{(-3,1) not commutative} that all three width invariants agree.  For all other contractible curves, including flips and singular flops, our invariants are new.  We will see in \ref{intro example} below that the noncommutative width may be strictly larger than the commutative width.
Remark~\ref{rem lower width bounds} also gives lower bounds for the width according to Dynkin type. This makes it clear that the algebras $\DA$ can be quite complex: for instance in type $E_8$ we find that $\width(C)\geq40$.

\subsection{Contraction algebras}

In addition to the deformation-theoretic viewpoint given above, we explain in \S\ref{flops section} how $\DA$ arises also as a {\em contraction algebra} associated to an algebra obtained by tilting.  It is this alternative description of  $\DA$ that allows us to calculate it, and also control it homologically. 

Briefly, for any contractible curve $C$ in a 3-fold $X$, by passing to an open neighbourhood $U$ of $C$, and then to the formal fibre $\widehat{U}\to\Spec{\widehat{R}}$, it is well known \cite{VdB1d} that $\widehat{U}$ is derived equivalent to an algebra $\AB:=\End_{\widehat{R}}(\widehat{R}\oplus N_1)$.  See \S\ref{complete local geometry summary} for more details.  Let $[\widehat{R}]$ be the two-sided ideal of $\AB$ consisting of homomorphisms $\widehat{R}\oplus N_1\to\widehat{R}\oplus N_1$ that factor through $\add\widehat{R}$.  We write $\CA:=\AB/[\widehat{R}] $ and call $\CA$ the {\em contraction algebra}.

\begin{thm}[=\ref{contraction alg isos}, \ref{cont is fd for flop}]\label{fd intro} For a flopping or flipping curve in a 3-fold, $\DA\cong\CA$, and furthermore this is a finite dimensional algebra.
\end{thm}

We refer the reader to \ref{Pagoda example} for the calculation of $\CA$ in the case of the Pagoda flop, in which case $\CA\cong\K[x]/x^n$, as expected.   More generally, our approach allows us to make calculations beyond Dynkin type $A$.  We illustrate this in the following example.

\begin{example}\label{intro example}The simplest non-type-$A$ flopping contraction is obtained by setting
\[
\widehat{R}:=\frac{\mathbb{C}[[u,v,x,y]]}{u^2+v^2y=x(x^2+y^3)},
\]
and considering the morphism $X\to\Spec \widehat{R}$ obtained by blowing up a generator of the class group.  There is a single floppable $(-3,1)$-curve above the origin, and this is a length two $D_4$ flop, known as the Laufer flop \cite{Pagoda,AM}. In this case, by \cite{AM}, $\AB:=\End_{\widehat{R}}(\widehat{R}\oplus N_1)$ can be presented as the completion of the quiver with relations
\[
\begin{array}{cc}
\begin{array}{c}
\begin{tikzpicture}[>=stealth]
\node (C1) at (0,0) {$\scriptstyle \widehat{R}$};
\node (C1a) at (-0.1,0)  {};
\node (C2) at (1.75,0) {$\scriptstyle N_1$};
\node (C2a) at (1.85,0.05) {};
\node (C2b) at (1.85,-0.05) {};
\draw [->,bend left=20,looseness=1,pos=0.5] (C1) to node[gap]  {$\scriptstyle a$} (C2);
\draw [->,bend left=20,looseness=1,pos=0.5] (C2) to node[gap]  {$\scriptstyle b$} (C1);
\draw[->]  (C2b) edge [in=-80,out=-10,loop,looseness=12,pos=0.5] node[below] {$\scriptstyle y$} (C2b);
\draw[->]  (C2a) edge [in=80,out=10,loop,looseness=12,pos=0.5] node[above] {$\scriptstyle x$} (C2a);
\end{tikzpicture}
\end{array}
&
{\scriptsize{
\begin{array}{l}
ay^{2}=-aba\\
y^{2}b=-bab\\
xy=-yx\\
x^2+yba+bay=y^{3}.
\end{array}}}
\end{array}
\] 
It follows immediately that
\[
\CA\cong \frac{\mathbb{C}\langle\langle x,y\rangle\rangle}{xy=-yx\mbox{, } x^2=y^{3}}=\frac{\mathbb{C}\langle x, y\rangle}{xy=-yx\mbox{, } x^2=y^{3}}.
\] 
We call this algebra a {\em quantum cusp}. From this presentation, it can be seen that $\width(C)=9$ (for details, see \ref{D4 flop general example}).  Furthermore, factoring by commutators it follows that the commutative deformation algebra is given by
\[
\CA^{\ab}\cong \frac{\mathbb{C}[x,y]}{xy\mbox{, } x^2=y^{3}}
\] 
and so $\cwidth(C)=\dim_{\mathbb{C}}\CA^{\ab}=5$.  

The deformations are sketched in Figure~\ref{fig.deformations2}.  In the commutative case on the left, the original curve moves infinitesimally, fibred over $\Spec \CAab$, with the dots in the base representing the vector space basis $\{ 1,x,x^2, y,y^2\}$ of $\CAab$. The noncommutative deformations on the right are harder to draw, but roughly speaking they thicken the universal sheaf $\cF$ to a sheaf $\cE\in\coh X$ over $\CA$, where the nine dots represent a basis of $\CA$.

\opt{a4}{\vspace{-5mm}}
\def\noncommShift{\size*6}
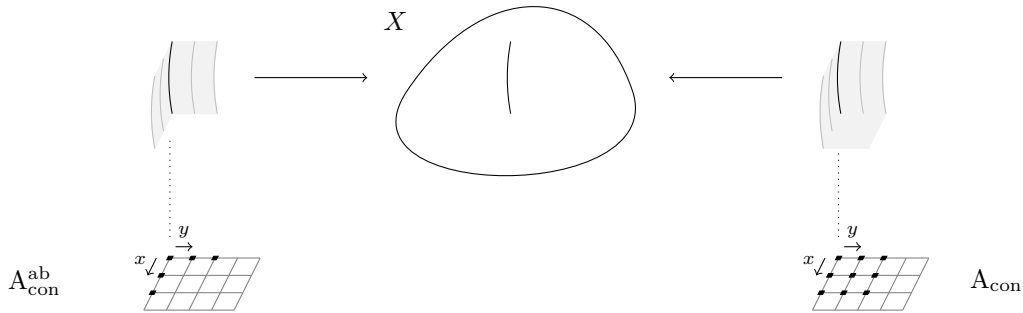
\begin{figure}[h]
\centering
\opt{12pt}{\hspace*{-3mm}}
\begin{tikzpicture}[scale=2,node distance=6]
\node at (0,-0.6) {$
  \begin{tikzpicture}[xscale=0.3,yscale=0.6]
    \defThreeFold{\size*3}{0}{black};
    \defCurve{\size*3}{0}{black};
    \node at (\size*2,\size/4) {$\small X$};
    \defGridLines{0}{0}{4}{3};
    \defUniversalComm{0}{0}{2}{2};
    \draw[->] (0.75*\size,0) -- (\size*1.75,0);
    \draw[->] (-0.6,\gridShift,0) -- node[left,pos=0.1] {$\scriptstyle x$} (-0.6,\gridShift,1);
    \draw[->] (0.25,\gridShift+0.25,0) -- node[above,pos=0.5] {$\scriptstyle y$} (1,\gridShift+0.25,0);
\node at (-6,-4.5) {$\small \CAab$};
\end{tikzpicture}$};

\node at (\size*0.68,-0.6) {$
  \begin{tikzpicture}[xscale=0.3,yscale=0.6]
    \defThreeFold{\noncommShift}{0}{white};
    \defGridLines{\noncommShift}{0}{4}{3};
    \defUniversal{\noncommShift}{0}{2}{2};
    \draw[->] (\noncommShift-0.5*\size,0) -- (\noncommShift-\size*1.5,0);
    \draw[->] (\noncommShift-0.6,\gridShift,0) -- node[left,pos=0.1] {$\scriptstyle x$} (\noncommShift-0.6,\gridShift,1);
    \draw[->] (\noncommShift+0.25,\gridShift+0.25,0) -- node[above,pos=0.5] {$\scriptstyle y$} (\noncommShift+1,\gridShift+0.25,0);
\node at (\noncommShift+7,-4.5) {$\small \CA$};
\end{tikzpicture}$};
\end{tikzpicture} 
\caption{Comparing commutative and noncommutative deformations of the Laufer $D_4$ floppable curve in \ref{intro example}.}
\label{fig.deformations2}
\end{figure}
\end{example}
\opt{a4}{\vspace{-2mm}}

Using this example, and its generalization \ref{D4 flop general example}, we are able to extend the previous known invariants in Table~\ref{table compare invts}.  This is summarised in Table~\ref{table 2}. We feel that the further study of the structure of these newly-arising algebras $\CA$, and of their representations, is of interest in itself.  Furthermore, we conjecture that the contraction algebra distinguishes the analytic type of the flop:

\begin{conj}
Suppose that $Y\to Y_{\con}$ and $Z\to Z_{\con}$ are flopping contractions of an irreducible rational curve in smooth projective $3$-folds, to points $p$ and $q$ respectively.  To these, associate the contraction algebras $\CA$ and $\CB$.  Then the completions of the stalks at $p$ and $q$ are isomorphic if and only if $\CA\cong\CB$.
\end{conj}

In the Type $A$ and $D$ cases for which we know the algebra $\CA$ explicitly, the conjecture is true.  We will return to the general case in future work.

\begin{table}[!h]
\centering
\opt{12pt}{\hspace*{-6mm}}
\begin{tabular}{*6c}
\toprule
& {\scriptsize (3)} & {\scriptsize (6)} & {\scriptsize (7)} & {\scriptsize ($8$)} & {\scriptsize ($9$)} \\
{\bf Example}& {\bf Dynkin type} & $\CA$ & $\CAab$ & $\width(C)$ & $\cwidth(C)$ \\
\midrule
Atiyah & \multirow{2}*{$A_1$} & $\K$ & $\K$ & 1 & 1 \\\cmidrule(l){1-1} \cmidrule(l){3-6}
Pagoda (\ref{Pagoda example}) & & $\frac{\K[x]}{x^{n}}$ & $\frac{\K[x]}{x^{n}} $ & $n$ & $n$ \\\cmidrule(l){1-6}
\multirow{2}*{Laufer (\ref{D4 flop general example})} & \multirow{2}*{$D_4$} & $\frac{\K\langle x,y\rangle}{xy=-yx,\,x^2=y^{2n+1}}$ & $\frac{\K[x,y]}{xy,\,x^2=y^{2n+1}}$ & \multirow{2}*{$3(2n+1)$} & \multirow{2}*{$2n+3$} \\
 & & {\scriptsize quantum cusp} & {\scriptsize cusp} & & \\
\bottomrule
\end{tabular}\\[\spacebelowtable]
\caption{Comparison of new invariants in known examples.}\label{table 2}
\end{table}

\subsection{Homological algebra and derived autoequivalences}\label{sect hom alg der auto}  We now restrict our attention to flopping contractions, and use the contraction algebra $\CA$ to give a unified construction of spherical-type twists related to flops.  This works even in the mildly singular setting. 

Given $X$, a projective $3$-fold with at worst terminal Gorenstein singularities, and a flopping contraction $X\to X_{\con}$ with flop $X'\to X_{\con}$, then there is an  autoequivalence of $\Db(\coh X)$ given as the composition of the flop equivalences
\[
\Db(\coh X)\overset{\flop}{\longrightarrow}
\Db(\coh X')\overset{\flopback}{\longrightarrow}
\Db(\coh X)
\]
where the Fourier--Mukai kernel $\cO_{X\times_{X_{\con}} X'}$ is used for both $\flop$ and $\flopback$ \cite{Bridgeland,Chen}.  We denote this composition by $\flopflop$.  

Without assuming knowledge of the flop, we construct an autoequivalence of $\Db(\coh X)$ which we call the {\em noncommutative twist functor}, using the noncommutative deformation theory from the previous section.  Then, using that $\flopflop$ is an autoequivalence, we show that it is an inverse of the noncommutative twist.  This gives an intrinsic description of  $\flopflop$ in terms of noncommutative deformation theory.

\subsection{Known autoequivalence results}
When $X$ is smooth and the flopping contraction is of Type $A$, the functor $\flopflop$ has previously been described intrinsically as some form of spherical twist, which we now briefly review.

Spherical twists about spherical objects, in the sense of \cite{ST}, are the simplest type of non-trivial derived autoequivalence.  If $C$ is a $(-1,-1)$-curve in a smooth 3-fold, the sheaf $E:=\cO_{\mathbb{P}^1}(-1)\in\coh X$ is a spherical object, and it was very well known to experts that the resulting twist autoequivalence was an inverse of $\flopflop$.  

For all other types of floppable curves, the structure sheaf $E$ is not spherical, so the twist of \cite{ST} is not an autoequivalence. Toda proposed in \cite{Toda} a conceptual way to overcome this.  The commutative deformation functor $\cDef_{E}$ defined above gives us a universal sheaf $\cF\in \coh X$.  Using this sheaf, Toda constructed a Fourier--Mukai functor
\[
T_{\cF}\colon\Db(\coh X)\to \Db(\coh X).
\]
Under the assumption that $X$ is smooth, using explicit knowledge of the deformation base~$S$, Toda proved the following.
\begin{itemize}
\item There is a functorial triangle $\RHom_X(\cF,-)\otimes_S^{\bf L} \cF\to \Id \to T_{\cF} \to $.
\item  For flopping contractions of Type $A$, $T_\cF$ is an autoequivalence.
\item For flopping contractions of Type $A$, $T_{\cF}$ is an inverse of $\flopflop$.
\end{itemize}

The techniques in \cite{Toda} have been very influential, and may be considered in the much more general framework of `spherical functors', see \cite{AL1,AL2}.  However, it has remained an open question as to whether the above results hold for all flopping curves, and whether they extend to mildly singular schemes.  The commutative deformation base is extremely hard to calculate, and so it is difficult to determine whether the techniques in \cite{Toda} can be extended.

\subsection{New autoequivalence results}  We now consider the general setting where $X$ is a projective $3$-fold with at worst terminal Gorenstein singularities, and there is a flopping contraction $X\to X_{\con}$.  By exploiting the universal sheaf $\cE$ provided by the noncommutative deformation theory in \S\ref{sect intro noncomm def}, in \ref{NCTwist finally} we construct a Fourier--Mukai functor
\[
T_\cE\colon\Db(\coh X)\to\Db(\coh X).
\]
Toda's functor $T_{\cF}$ was constructed using $\Spec \K[x]/x^n$ in a crucial way, but in this general setting we cannot take Spec of $\CA$ since it may not be commutative, and so our construction of $T_{\cE}$ is necessarily more abstract.   

The following is the main result of this paper.
\begin{thm}[=\ref{prop twist functorial triangle}, \ref{twist flopflop}]\label{main Db intro}
For any flopping contraction of an irreducible rational curve $C$ in $X$, a projective $3$-fold with only terminal Gorenstein singularities, 
\begin{itemize}
\item There is a functorial triangle 
 $\RHom_X(\cE,-)\otimes_{\CA}^{\bf L} \cE\to \Id \to T_{\cE}\to $.
\item The functor $T_\cE\colon\Db(\coh X)\to\Db(\coh X)$ is an autoequivalence.
\item The functor $T_{\cE}$ is an inverse of $\flopflop$.
\end{itemize}
\end{thm}
Thus in addition to extending the previous work to cover $(-3,1)$-curves, our noncommutative deformation technique also extends to the setting of mildly singular schemes.  
In the course of the proof, we establish the following.
\begin{prop}[=\ref{Ext cor}, \ref{Gsph conditions ok}]\label{prop ext intro} With the assumptions as in \ref{main Db intro}, and writing as above $E:=\cO_{\mathbb{P}^1}(-1)\in\coh X$,
\begin{itemize}
\item The contraction algebra $\CA$ is a self-injective finite dimensional algebra. 
\item The universal object $\cE$ is a perfect complex (although $E$ need not be).
\item We have
\[
\Ext_X^t(\cE,E) =\left\{ \begin{array}{cl} \mathbb{C}&\mbox{if }t=0,3\\
0&\mbox{else.}\\  \end{array} \right.
\]
\end{itemize}
\end{prop}

\begin{remark} We will not discuss here whether a suitable notion of spherical functors could be applied in our noncommutative setting, but the above suggests that it is possible to use this axiomatic approach, potentially giving noncommutative twists in other triangulated and DG categories.  However, we remark that the results in the proposition are not formal consequences of deformation theory, and so proving results in any axiomatic framework will necessarily need to use results obtained in this paper, or find another approach to show \ref{prop ext intro}.
\end{remark}

The noncommutative deformation base $\CA$ is artinian, and thus the noncommutative deformations of flopping curves are infinitesimal and may be analysed on the formal fibre.  Consequently, to establish the Ext vanishing and other properties we pass to this formal fibre.  There, we adapt the techniques of \cite[\S6]{IW4} and prove \ref{prop ext intro} without an explicit presentation of $\CA$, or equivalently a filtration structure on $\cE$.

In the course of proving that $T_{\cE}$ is an autoequivalence, once we have shown that $T_{\cE}$ is fully faithful, the usual Serre functor trick \cite[2.4]{BKR} to establish equivalence does not work in our possibly singular setting. However we are able to bypass this by instead giving an elementary argument based on the `difficult to check' criterion in \cite[1.50]{HuybrechtsFM}.  This may be of independent interest.

\begin{lemma}[=\ref{equiv trick}]
Let $\cC$ be a triangulated category, and $F\colon \cC\to\cC$ an exact fully faithful functor with right adjoint $F^{\RA}$.  Suppose that there exists an object $c\in\cC$ such that $F(c)\cong c[i]$ for some $i$, and further $F(x)\cong x$ for all $x\in c^\perp$.  Then $F$ is an equivalence. 
\end{lemma}

\subsection{Necessity of noncommutative deformations}\label{nec of def}
Restricting to the case when $X$ is smooth, by studying $\CA$ we can also prove that Toda's functor $T_\cF$ defined using commutative deformations is not an autoequivalence for $(-3,1)$-curves.  This shows that our noncommutative deformation approach is strictly necessary.  To achieve this, we investigate the conditions under which $\CA$ is commutative.

\begin{thm}[=\ref{(-3,1) not commutative}]\label{(-3,1) not commutative intro}
Let $X\to X_{\con}$ be the flopping contraction of an irreducible rational curve $C$, where $X$ is smooth.  Then the contraction algebra $\CA$ is commutative if and only if $C$ is a $(-1,-1)$-curve or a $(-2,0)$-curve.
\end{thm}

The main content in the above theorem is the `only if' direction, and we prove it without any explicit case-by-case analysis.  When $\CA$ is not commutative, we prove in \ref{Toda not equiv global} that Toda's functor $T_\cF$ is not an equivalence.  Combining with \ref{(-3,1) not commutative intro}, we obtain the following corollary. 
\begin{cor}[=\ref{Toda not equiv global}]\label{not equiv intro}
With the setup as in \ref{(-3,1) not commutative intro}, suppose that $C$ is a $(-3,1)$-curve.  Then Toda's functor $T_{\cF}$ is not an autoequivalence of $\Db(\coh X)$.
\end{cor}

\subsection{Conventions}  Unqualified uses of the word `module' refer to right modules, and $\mod A$ denotes the category of finitely generated right $A$-modules.  Sometimes left modules will appear, in which case we will either emphasize the fact that they are left modules, or consider them as objects of $\mod A^{\op}$.  If $M\in\mod A$, we let $\add M$ denote all possible summands of finite sums of $M$.  We say that $M$ is a generator if $R\in\add M$.  

We use the functional convention for composing arrows, so $f\cdot g$ means $g$ then $f$.  With this convention, $M$ is a $\End_R(M)^{\op}$-module.  Furthermore, $\Hom_R(M,N)$ is a $\End_R(M)$-module and a $\End_R(N)^{\op}$-module, in fact a bimodule.  Note also that $\Hom_R({}_{S}M_R, {}_{T}M_R)$ is an $S$--$T$ bimodule and $\Hom_{R^{\op}}({}_{R}M_S, {}_{R}M_T)$ is a $T$--$S$ bimodule.

Throughout, we work over the field of complex numbers $\K$, but many of our arguments work without this assumption.

\subsection{Notations}  In Appendix~\ref{list of notations} we list our main notations, along with cross-references to their definitions.

\medskip
\noindent
\textbf{Acknowledgements.}  It is a pleasure to thank Alessio Corti for conversations concerning \cite{Pinkham}, Colin Ingalls for bringing to our attention \cite{Laudal}, Osamu Iyama for explaining aspects of \cite{IR}, Tim Logvinenko for explaining both the technical and philosophical aspects of \cite{AL1} and \cite{AL2},  Ed Segal for answering our deformation questions related to his thesis \cite{Segal},    Miles Reid for discussions regarding \cite{Pagoda}, and the referees for their careful reading and simplifying suggestions.  The first author is grateful for the support of the Erwin Schr\"odinger Institute, Vienna, and for the hospitality of Matthew Ballard, David Favero, and Ludmil Katzarkov. The second author also wishes to thank the Mathematical Sciences Research Institute, Berkeley, for its hospitality whilst this manuscript was being prepared.

\section{Noncommutative Deformation Functors}\label{NCdef section}

Noncommutative deformations of modules over algebras were first introduced and studied by Laudal \cite{Laudal}. Later this was extended to the setting of quasi-coherent sheaves over schemes by Eriksen \cite{Eriksen}, and also studied by Segal \cite{Segal}, Efimov--Lunts--Orlov \cite{ELO1, ELO2, ELO3} and others.   These works all establish prorepresentability of certain functors, whereas in the situation of flips and flops we will require the much stronger property of representability.   The above approaches are not well suited to this problem, so in this section we will adopt a different strategy, based on tilting.

Classically, formal deformation functors are certain covariant functors
$
\cart\to\Sets
$
from local commutative artinian $\K$-algebras to sets.  The idea behind noncommutative deformation theory is that the source category should be enlarged to include noncommutative algebras.  It is well known that contractible irreducible rational curves have only infinitesimal commutative deformations, and that this can be detected by passing to a suitable formal neighbourhood.  In this section we will establish that this is also the case for noncommutative deformations, by relating deformations on the $3$-fold $X$ to deformations over a certain tilting algebra associated to a formal neighbourhood of the curve $C$.  We prove this in a sequence of reduction steps, first to a Zariski local neighbourhood $U$ in \ref{local to global deformations the same}, then to a $\mathbb{C}$-algebra $\Lambda$ in \ref{local to alg deformations the same}, then finally to the formal fibre in \ref{local to very local deformations the same}.  From this, representability will be established in \S\ref{flops section}.

\subsection{Algebras and functors}

A key feature of the theory of noncommutative deformations, following \cite{Laudal,Eriksen,Segal}, is to consider simultaneous deformations of discrete families of $n$ objects, which may be parametrised by the $n$-pointed algebras defined below. For the purposes of this paper, only the case $n=1$ will be used, however we give the general definition as it highlights features that are hidden in the $n=1$ case.

\begin{defin}\phantomsection\label{defin.n-pointed_non-comm_K-alg}
An \emph{$n$-pointed $\K$-algebra} $\Gamma$ is an associative $\K$-algebra, together with $\K$-algebra morphisms $p\colon \Gamma \to \K^n$ and $i\colon \K^n \to \Gamma$ such that $p \circ i = \Id$.  A morphism of $n$-pointed $\K$-algebras $\psi\colon (\Gamma,p,i)\to(\Gamma',p',i')$ is an algebra homomorphism $\psi\colon\Gamma\to\Gamma'$ such that
\[
\begin{tikzpicture}
\node (A) at (0,0) {$\K^n$};
\node (B1) at (1.5,0.75) {$\Gamma$};
\node (B2) at (1.5,-0.75) {$\Gamma'$};
\node (C) at (3,0) {$\K^n$};
\draw[->] (A) -- node[above] {$\scriptstyle i$} (B1);
\draw[->] (A) -- node[below] {$\scriptstyle i'$} (B2);
\draw[->] (B1) -- node[above] {$\scriptstyle p$} (C);
\draw[->] (B2) -- node[below] {$\scriptstyle p'$} (C);
\draw[->] (B1) -- node[right] {$\scriptstyle \psi$}(B2);
\end{tikzpicture}
\]
commutes.  We denote the category of $n$-pointed $\K$-algebras by $\alg_n$.  We denote the full subcategory consisting of those objects that are commutative rings by $\calg_n$.

We write $\art_n$ for the full subcategory of $\alg_n$ consisting of objects $(\Gamma,p,i)$ for which $\dim_{\K}\Gamma<\infty$ and the augmentation ideal $\n:=\Ker(p)$ is nilpotent.  We write $\cart_n$ for the full subcategory of $\art_n$ consisting of those objects that are commutative rings.
\end{defin}

\begin{remark}
In the literature, an $n$-pointed $\mathbb{C}$-algebra is sometimes called an augmented $\mathbb{C}^n$-algebra.
We also remark that when $n=1$ the morphism $i$ is the structure morphism, so can be omitted.  Furthermore, $\cart_1$ is simply the category of commutative artinian local $\K$-algebras, with maximal ideal $\n$.   In our setting below, $\art_1$ will be the source category for our noncommutative deformation functors, since we will restrict to contractions of curves with $1$ irreducible component.  The category $\art_n$ is required for contractions of curves with $n$ irreducible components \cite{DW2}.
\end{remark}

Throughout we will only consider abelian categories $\cA=\Mod\Lambda$ or $\cA=\Qcoh Y$ where $\Lambda$ is a $\mathbb{C}$-algebra and $Y$ is a quasi-projective $\mathbb{C}$-scheme, since considering a general $\mathbb{C}$-linear abelian category requires significantly more technology \cite{LowenVdB, Segal,ELO1,DW2}.  With this restriction, it is possible to pick an object of $\cA$ and study its noncommutative deformation functor in a naive way sufficient for our purposes.

\begin{defin}\label{rem pairs cat}
For $\cA$ an abelian category as above, and an associative $\K$-algebra $\Gamma$, the category $\PairsCat{\cA}{\Gamma}$ has as objects pairs $(b,\phi)$ where $b\in\cA$ and $\phi$ is a $\mathbb{C}$-algebra homomorphism $\phi\colon \Gamma \to \End_{\cA}(b)$.
Morphisms in $\PairsCat{\cA}{\Gamma}$ are defined to intertwine the action maps $\phi$: explicitly, a morphism $f\colon (b,\phi) \to (b',\phi')$ is given by a morphism $f\colon b \to b'$ in $\cA$ such that 
\[
\begin{array}{c}
\begin{tikzpicture}
\node (a1) at (0,0) {$b$};
\node (a2) at (1.25,0) {$b'$};
\node (b1) at (0,-1.25) {$b$};
\node (b2) at (1.25,-1.25) {$b'$};
\draw[->] (a1) to node[above] {$\scriptstyle f$} (a2);
\draw[->] (b1) to node[above] {$\scriptstyle f$} (b2);
\draw[->] (a1) to node[left] {$\scriptstyle \phi(r)$} (b1);
\draw[->] (a2) to node[right] {$\scriptstyle \phi'(r)$} (b2);
\end{tikzpicture}
\end{array}
\]
commutes for all $r\in\Gamma$.
\end{defin}

\begin{defin}\label{NC deformation definition}
Let $\cA$ be the abelian category  $\cA=\Mod\Lambda$ or $\cA=\Qcoh Y$ where $\Lambda$ is a $\mathbb{C}$-algebra and $Y$ is a quasi-projective $\mathbb{C}$-scheme.  Pick $a\in\cA$, then the \emph{noncommutative deformation functor} is defined
\[
\Def^{\cA}_{a}\colon\art_1\to\Sets
\]
by sending
\[
(\Gamma,\n)\mapsto \left. \left \{ ((b,\phi),\delta)
\left|\begin{array}{l}(b,\phi)\in\PairsCat{\cA}{\Gamma}\\ -\otimes_\Gamma b\colon\mod \Gamma\to \cA \t{ is exact}\\ \delta\colon (\Gamma/\n)\otimes_\Gamma b\xrightarrow{\sim}a\end{array}\right. \right\} \middle/ \sim \right.
\]
where $((b,\phi),\delta)\sim ((b',\phi'),\delta')$ if there exists an isomorphism $\tau\colon(b,\phi)\to(b',\phi')$ in $\PairsCat{\cA}{\Gamma}$ such that 
\[
\begin{tikzpicture}
\node (a1) at (0,0) {$(\Gamma/\n)\otimes_\Gamma b$};
\node (a2) at (3,0) {$(\Gamma/\n)\otimes_\Gamma b'$};
\node (b) at (1.5,-1) {$a$};
\draw[->] (a1) -- node[above] {$\scriptstyle 1\otimes \tau$} (a2);
\draw[->] (a1) -- node[pos=0.6,left,inner sep=8,anchor=east] {$\scriptstyle \delta$} (b);
\draw[->] (a2) -- node[pos=0.6,right,inner sep=8,anchor=west] {$\scriptstyle \delta'$} (b);
\end{tikzpicture}
\]
commutes.  The \emph{commutative deformation functor} is defined to be the restriction of the above functor to the category $\cart_1$, and is denoted by
\[
\cDef^{\cA}_a\colon \cart_1\to \Sets.
\]
\end{defin}

\begin{remark}
When $\cA=\Mod\Lambda$, the functor $-\otimes_\Gamma b$ in \ref{NC deformation definition} is the standard tensor functor of modules.  When $\cA=\Qcoh Y$, the functor $-\otimes_\Gamma b$ is defined in exactly the same way as in the case of tilting equivalences, and can be thought of as tensoring over the noncommutative sheaf of algebras $\cO_Y\otimes_{\underline{\mathbb{C}}}\underline{\Gamma}$, where $\underline{\mathbb{C}}$ and $\underline{\Gamma}$ are the constant sheaves on $Y$ associated to $\mathbb{C}$ and $\Gamma$ (see e.g.\ \S\ref{geom noncomm twist}).
\end{remark}

The fact that $\Def^{\cA}_a$ is a functor is routine, and we refer the reader to \cite{Laudal, Eriksen}.

\subsection{Contractions and deformation functors}\label{setup section 2.1}

Our setup is the contraction of an irreducible rational curve $C$ ($C^{\redu}\cong\mathbb{P}^1$) in a quasi-projective normal $3$-fold $X$ with only Cohen--Macaulay canonical singularities.  We remark that throughout the paper, there will be no extra  assumptions on the singularities of $X$ unless stated.  By a contraction, we mean a projective birational morphism $f\colon X\to X_{\con}$ contracting $C$ to a point $p$, satisfying $\Rf_*\cO_X=\cO_{X_{\con}}$, such that $X\backslash C\cong X_{\con}\backslash p$.  This incorporates both flips and flops.  

We are interested in deformations of the sheaf $\cO_{\mathbb{P}^1}(-1)$ viewed as a sheaf on $X$, so we choose an affine open neighbourhood $U_{\con} \cong \Spec R$ around $p$. Putting $U:=f^{-1}(U_{\con})$, we have a commutative diagram
\begin{eqnarray}
\begin{array}{c}
\begin{tikzpicture}
\node (C) at (-0.6,0) {$C^{}$}; 
\node (U) at (1,0) {$U$};
\node (X) at (3,0) {$X$};
\draw[right hook->] (C) to node[pos=0.6,above] {$\scriptstyle e$} (U);
\draw[right hook->] (U) to node[above] {$\scriptstyle i$} (X);

\node (x) at (-0.6,-1.5) {$\phantom{R}p\phantom{R}$}; 
\node (Uc) at (1,-1.5) {$U_{\con}$};
\node (Xc) at (3,-1.5) {$X_{\con}$};
\node at (0.1,-1.5) {$\in$}; 
\draw[right hook->] (Uc) to (Xc);

\node (m) at (-0.6,-2.5) {$\m$}; 
\node (R) at (1,-2.5) {$\Spec R$};
\node at (0.1,-2.5) {$\in$}; 

\draw[->] (X) --  node[right] {$\scriptstyle f$} (Xc);
\draw[->] (U) --  node[right] {$\scriptstyle f|_U$}  (Uc);
\draw[|->] (C) --  (x);
\node [rotate=-90] at (1, -2) {$\cong$};
\node [rotate=-90] at (-0.6, -2) {$=$};
\end{tikzpicture}
\end{array}\label{key comm local global diagram}
\end{eqnarray}
where $C^{\redu}\cong \mathbb{P}^1$,  $e$ is a closed embedding, and $i$ is an open embedding.  Since $X_{\con}$ is separated and $U_{\con}$ is affine, the morphism $U_{\con}\hookrightarrow X_{\con}$ is affine, and since affine morphisms are preserved under pullback it follows that  $i$ is an affine morphism. This, together with the fact that $i$ is an open embedding, implies that there is a fully faithful embedding of derived categories
\begin{eqnarray}
\Ri_*=i_*\colon\D(\Qcoh U)\hookrightarrow \D(\Qcoh X)\label{inclusion line}
\end{eqnarray}
with adjoints $i^!$ and $i^*$. 

 Throughout we write $E:=e_*\cO_{\mathbb{P}^1}(-1)\in\coh U$, so that $i_*E\in\coh X$ is  just the sheaf $\cO_{\mathbb{P}^1}(-1)$ viewed as a sheaf on $X$.  From \ref{NC deformation definition}, we thus have two deformation functors $\Def^{\Qcoh U}_{E}$ and $\Def^{\Qcoh X}_{i_*E}$, both of which take $\art_1\to\Sets$.  For brevity, we denote them by $\Def^U_E$ and $\Def^X_{i_*E}$ respectively.   The isomorphism of these deformation functors below is not surprising: it says that noncommutative deformations can be detected Zariski locally.  Although elementary, we give the proof in full, since we will need variations of it later.

\begin{prop}\label{local to global deformations the same}
There is a natural isomorphism $\Def^{ U}_{E}\xrightarrow{\sim}\Def^{ X}_{i_*E}$ induced by $i_*$, with inverse induced by $i^*$. By restriction, there is also a natural isomorphism $\cDef^{U}_{E}\xrightarrow{\sim}\cDef^{X}_{i_*E}$, with a corresponding inverse. 
\end{prop}
\begin{proof}
Given $((\cF,\phi),\delta)\in\Def^U_E(\Gamma)$, composing $\Gamma\xrightarrow{\phi}\End_U(\cF)\xrightarrow{i_*}\End_X(i_*\cF)$ gives a pair $(i_*\cF,i_*\phi)\in\PairsCat{\Qcoh X}{\Gamma}$.  We check that this belongs to $\Def^{ X}_{i_*E}(\Gamma)$. Firstly, by the projection formula and \eqref{inclusion line}, the functor $-\otimes_\Gamma i_*\cF\colon \mod\Gamma\to\Qcoh X$ factorizes as
\begin{eqnarray}
-\otimes_\Gamma i_*\cF\colon \mod\Gamma\xrightarrow{-\otimes_\Gamma\cF} \Qcoh U\xrightarrow{i_*}\Qcoh X.\label{comp is exact}
\end{eqnarray}
The first functor is exact by definition, and the second functor is exact by \eqref{inclusion line}, so the composition $-\otimes_\Gamma i_*\cF$ is exact.  Next, we have that
\[
\Gamma/\n\otimes_\Gamma i_*\cF \overset{\sim}{\longrightarrow} i_*(\Gamma/\n\otimes_\Gamma \cF)\xrightarrow{i_*(\delta)}i_*E 
\]
is an isomorphism, so denoting the first isomorphism by $\alpha$, it follows that 
\[
\LargeFamily{(i_*\cF,i_*\phi)}{i_*(\delta)\cdot\alpha}\in\Def^X_{i_*E}(\Gamma).
\]
It is easy to check that this preserves the equivalence relation $\sim$ and is functorial, so gives a natural transformation
\[
i_*\colon \Def^U_E\to\Def^X_{i_*E}.
\]  

Conversely, given  $((\cG,\psi),\gamma)\in\Def^X_{i_*E}(\Gamma)$, composing $\Gamma\xrightarrow{\psi}\End_X(\cG)\xrightarrow{i^*}\End_U(i^*\cG)$ gives a pair $(i^*\cG,i^*\psi)\in\PairsCat{\Qcoh U}{\Gamma}$.   The functor 
\opt{10pt}{$-\otimes_\Gamma i^*\cG\colon \mod\Gamma\to\Qcoh U$}
\opt{12pt}{\[-\otimes_\Gamma i^*\cG\colon \mod\Gamma\to\Qcoh U\]}
factorizes as
\begin{eqnarray}
-\otimes_\Gamma i^*\cG\colon \mod\Gamma\xrightarrow{-\otimes_\Gamma\cG} \Qcoh X\xrightarrow{i^*}\Qcoh U,\label{comp is exact 22}
\end{eqnarray}
so as a composition of two exact functors it is exact.  Further we have that
\[
\Gamma/\n\otimes_\Gamma i^*\cG \overset{\sim}{\longrightarrow} i^*(\Gamma/\n\otimes_\Gamma \cG)\xrightarrow{i^*(\gamma)}i^*i_*E \overset{\sim}{\longrightarrow} E 
\]
is an isomorphism, so denoting the first isomorphism by $\beta$, it follows that 
\[
\LargeFamily{(i_*\cF,i_*\phi)}{\varepsilon_{E}\cdot i^*(\gamma)\cdot\beta}\in\Def^X_{i_*E}(\Gamma)
\]
where $\varepsilon_{E}$ is the counit morphism.  Again it is easy to check that this preserves the equivalence relation $\sim$ and is functorial, so gives a natural transformation
\[
i^*\colon \Def^X_{i_*E}\to\Def^U_{E}.
\]  

To show that these natural transformations are isomorphisms, we need to show that
\begin{align*}
((\cF,\phi),\delta)&\sim \LargeFamily{(i^*i_*\cF,i^*i_*\phi)}{\varepsilon_{E}\cdot i^*(i_*(\delta)\cdot\alpha)\cdot \beta}\\
 ((\cG,\psi),\gamma)&\sim \LargeFamily{(i_*i^*\cG,i_*i^*\psi)}{i_*(\varepsilon_{E}\cdot i^*(\gamma)\cdot\beta)\cdot\alpha}.
\end{align*}
We explain the second, the first being slightly easier.  Since $\cG$ is filtered by the sheaf $i_*E$, its support is contained in $U$ and so the unit map $\eta_{\cG}\colon \cG\to i_*i^*\cG$ is an isomorphism.  By functoriality of the adjunction it is clear that 
\[
\begin{array}{c}
\begin{tikzpicture}[yscale=1.25]
\node (G) at (-1,0) {$\cG$}; 
\node (iiG) at (1,0) {$i_*i^*\cG$};
\node (G2) at (-1,-1) {$\cG$}; 
\node (iiG2) at (1,-1) {$i_*i^*\cG$};
\draw[->] (G) to node[above] {$\scriptstyle \eta_{\cG}$} (iiG);
\draw[->] (G2) to node[above] {$\scriptstyle \eta_{\cG}$} (iiG2);
\draw[->] (G) --  node[left] {$\scriptstyle \psi(r)$}  (G2);
\draw[->] (iiG) --  node[right] {$\scriptstyle i_*i^*\psi(r)$}  (iiG2);
\end{tikzpicture}
\end{array}
\]
commutes for all $r\in\Gamma$, and so $(\cG,\psi)\cong(i_*i^*\cG,i_*i^*\psi)$ in $\PairsCat{\Qcoh X}{\Gamma}$.  Finally, we claim that
\[
\begin{tikzpicture}[scale=1.25]
\node (a1) at (0,0) {$(\Gamma/\n)\otimes_\Gamma \cG$};
\node (a2) at (3,0) {$(\Gamma/\n)\otimes_\Gamma i_*i^*\cG$};
\node (b) at (1.5,-1) {$i_*E$};
\draw[->] (a1) -- node[above] {$\scriptstyle 1\otimes \eta_{\cG}$} (a2);
\draw[->] (a1) -- node[pos=0.6,left,inner sep=8,anchor=east] {$\scriptstyle \gamma$} (b);
\draw[->] (a2) -- node[pos=0.6,right,inner sep=8,anchor=west] {$\scriptstyle (i_*\varepsilon_{E})\cdot (i_*i^*(\gamma))\cdot (i_*\beta)\cdot\alpha$} (b);
\end{tikzpicture}
\]
commutes.  We rewrite this diagram, adding in a unit morphism as the dotted arrow:
\[
\begin{tikzpicture}[xscale=1.5,yscale=1.25]
\node (a1) at (0,0) {$(\Gamma/\n)\otimes_\Gamma \cG$};
\node (a2) at (3,0) {$(\Gamma/\n)\otimes_\Gamma i_*i^*\cG$};
\node (b) at (0,-3) {$i_*E$};
\node (b1) at (3,-1) {$i_*((\Gamma/\n)\otimes_\Gamma i^*\cG)$};
\node (b2) at (3,-2) {$i_*i^*((\Gamma/\n)\otimes_\Gamma \cG)$};
\node (b3) at (3,-3) {$i_*i^*i_*E$};
\draw[->] (a1) -- node[above] {$\scriptstyle 1\otimes \eta_{\cG}$} (a2);
\draw[->] (a1) -- node[left] {$\scriptstyle \gamma$} (b);
\draw[->] (a2) -- node[right] {$\scriptstyle \alpha$} (b1);
\draw[->] (b1) -- node[right] {$\scriptstyle i_*\beta$} (b2);
\draw[->] (b2) -- node[right] {$\scriptstyle i_*i^*(\gamma)$} (b3);
\draw[->] (b3) -- node[above] {$\scriptstyle i_*\varepsilon_{E}$} (b);
\draw[densely dotted,->] (a1) -- node[gap] {$\scriptstyle \eta_{(\Gamma/\n\otimes\cG)}$} (b2);
\end{tikzpicture}
\]
The bottom half commutes by the functoriality of the unit, since $i_*\varepsilon_E=(\eta_{i_*E})^{-1}$ by the triangular identity.  It remains to show that the top half commutes. This just follows by inspection, since locally the maps send
\[
\begin{tikzpicture}[xscale=1.5,yscale=1.25]
\node (a1) at (0,0) {$1\otimes a$};
\node (a2) at (3,0) {$1\otimes i_*i^*a$};
\node (b1) at (3,-1) {$i_*(1\otimes i^*a)$};
\node (b2) at (3,-2) {$\phantom{.}i_*i^*(1\otimes a).$};
\draw[|->] (a1) -- (a2);
\draw[|->] (a2) -- (b1);
\draw[|->] (b1) --  (b2);
\draw[densely dotted,|->] (a1) -- (b2);
\end{tikzpicture}\qedhere
\]
\end{proof}

\subsection{From geometry to algebra}
\label{sect geom to alg}
In \ref{local to global deformations the same} we reduced the problem of deforming $i_*E$ in $X$ to deforming $E$ in $U$. Now we further reduce the problem to deforming a simple module over a certain noncommutative algebra, using a derived equivalence of Van den Bergh \cite{VdB1d}. With respect to our applications later, this step is crucial, since it yields a different description of the universal object, one which is homologically much easier to control.  

We keep the setup as in \eqref{key comm local global diagram}, so we have a projective birational map $f\colon U\to\Spec R$ of 3-folds, with at most one-dimensional fibres, such that $\Rf_*\cO_U=\cO_R$, where $R$ is a Cohen--Macaulay $\mathbb{C}$-algebra.  Since the fibre is at most one-dimensional, in this setup it is known \cite[3.2.8]{VdB1d} that there is a bundle $ \cV:=\cO_U\oplus\cN$  inducing a derived equivalence
\begin{eqnarray}
\begin{array}{c}
\begin{tikzpicture}
\node (a1) at (0,0) {$\Db(\coh U)$};
\node (a2) at (5,0) {$\Db(\mod \End_U(\cV))$};
\node (b1) at (0,-1) {$\Per U$};
\node (b2) at (5,-1) {$\mod\End_U(\cV) $};
\draw[->] (a1) -- node[above] {$\scriptstyle\RHom_U(\cV,-)$} node [below] {$\scriptstyle\sim$} (a2);
\draw[->] (b1) --  node [below] {$\scriptstyle\sim$} (b2);
\draw[right hook->] (b1) to (a1);
\draw[right hook->] (b2) to (a2);
\end{tikzpicture}
\end{array}\label{derived equivalence}
\end{eqnarray}
Here $\Per U$ denotes the category of perverse sheaves, defined to be
\[
\Per U:=\left\{ a\in\Db(\coh U)\left| \begin{array}{c}H^i(a)=0\mbox{ if }i\neq 0,-1\\
f_*H^{-1}(a)=0\mbox{, }\Rfi{1}_* H^0(a)=0\\ \Hom(c,H^{-1}(a))=0\mbox{ for all }c\in\cC \end{array}\right. \right\}
\]
where
\[
\cC:=\{ c\in\coh U\mid \Rf_*c=0\}.
\]
By birationality \cite[4.2.1]{VdB1d}, $\End_{U}(\cV)\cong \End_R(f_*\cV)$, hence we write $N:=f_*\cN$ and throughout this section we set
\[
\Lambda:=\End_U(\cV) = \End_U(\cO_U\oplus\cN)\cong\End_R(R\oplus N).
\]

Since it is clear that $E=e_*\cO_{\mathbb{P}^1}(-1)$, viewed as a complex in degree zero, belongs to $\Per U$, sending it across the derived equivalence \eqref{derived equivalence} yields a module in degree zero.
\begin{defin}\label{T definition}
We write $T:=\Hom_U(\cV,E)\cong\RHom_U(\cV,E)$.
\end{defin}

By \ref{NC deformation definition}, there is a deformation functor $\Def^{\Mod\Lambda}_T\colon\art_1\to \Sets$, which to ease notation we denote by $\Def^{\Lambda}_T$.  The following is an analogue of \ref{local to global deformations the same}, and is proved in a very similar way.

\begin{prop}\label{local to alg deformations the same}
There is a natural isomorphism $\Def^{ U}_{E}\cong\Def^{\Lambda}_{T}$ induced by $\Hom_U(\cV,-)$ and $-\otimes_\Lambda\cV$, and consequently by restriction a natural isomorphism $\cDef^{U}_{E}\cong\cDef^{\Lambda}_{T}$. 
\end{prop}
\begin{proof}
The functors $\Hom_U(\cV,-)\colon\Qcoh U\to \Mod\Lambda$ and $-\otimes_\Lambda\cV\colon\Mod\Lambda\to\Qcoh U$ are adjoint.  Since $T$ is a module corresponding under the derived equivalence to a sheaf $E$, we know that $T\cong \Hom_U(\cV,T\otimes_{\Lambda}\cV)$ and $E\cong \Hom_U(\cV,E)\otimes_{\Lambda}\cV$ via the unit and counit morphisms. 

The proof is now identical to \ref{local to global deformations the same}, where in the analogue of \eqref{comp is exact} and \eqref{comp is exact 22}, the second functors are no longer exact, but they are exact out of the image of the first functor (using $\RHom_U(\cV,E)=\Hom_U(\cV,E)$ and  $T\otimes_\Lambda^{\bf L}\cV=T\otimes_\Lambda\cV$ respectively), which is sufficient for the remainder of the proof.
\end{proof}

\begin{defin}\label{contraction def lambda}
We define $[R]$ to be the two-sided ideal of $\Lambda=\End_R(R\oplus N)$ consisting of those morphisms $R\oplus N\to R\oplus N$ which factor 
\[
\begin{tikzpicture}
\node (A) at (0,0) {$R\oplus N$};
\node (B2) at (1.5,-0.75) {$P$};
\node (C) at (3,0) {$R\oplus N$};
\draw[densely dotted,->] (A) -- (B2);
\draw[densely dotted,->] (B2) -- (C);
\draw[->] (A) -- (C);
\end{tikzpicture}
\]
through an object $P\in\add R$. We set $I_{\con}:=[R]$ and define the {\em contraction algebra associated to $\Lambda$} to be $\Lambda_{\con}:=\Lambda/I_{\con}$.
\end{defin}

\begin{remark}\label{morita is an issue}
The contraction algebra is the fundamental object in our paper, but we remark that the algebra $\Lambda_{\con}$ defined above in \ref{contraction def lambda} depends on $\Lambda$, which in turn depends on the choice of derived equivalence in \eqref{derived equivalence}; different equivalences may be obtained due to the lack of Krull--Schmidt decompositions of tilting sheaves, see \cite[3.2.7]{VdB1d}. In \ref{def basic algebra} we will define a contraction algebra associated to $C$ which is intrinsic to the geometry of $X$, and does not involve choices, as shown in \ref{contraction alg isos}. It will turn out in \ref{morita lemma} that $\Lambda_{\con}$ will be morita equivalent, though not necessarily isomorphic, to the contraction algebra associated to $C$.
\end{remark}

\subsection{Complete local geometric setting}\label{complete local geometry summary}
Having in \ref{local to global deformations the same} and \ref{local to alg deformations the same} reduced the problem of deforming $i_*E$ in $X$ to the purely algebraic problem of deforming $T$ in $\mod \Lambda$, we next search for a candidate representing object.  This requires passing through a morita equivalence; indeed we would like to say that $\Lambda_{\con}$ belongs to $\art_1$, since then it is a natural candidate for the representing object.  However, the algebras in $\art_1$ have only one simple module, and it is one-dimensional: this is not in general true for $\Lambda_{\con}$ for the reasons in \ref{morita is an issue}.  Thus to obtain a candidate representing object in $\art_1$, we need to pass to the basic algebra associated to $\Lambda_{\con}$. The purpose of this subsection is to obtain this from  the formal fibre, and construct the contraction algebra $\CA$. We will show that $\CA$ is indeed the representing object in \S\ref{flops section}.

For $X\in\Mod \Lambda$ and $\n\in\Max R$, we write $X_\n:=X\otimes_RR_\n$, and for $Y\in\Mod\Lambda$ we often write $\widehat{Y}:=Y_\n\otimes_{R_\n}\widehat{R_\n}$.  We require a better description of $\widehat{\Lambda}$.  Completing the base with respect to $\m$ in \eqref{key comm local global diagram}, since endomorphism rings of finitely generated modules behave well under completion \cite[2.10]{EisenbudBook}, it is standard that
\[
\widehat{\Lambda}:=\Lambda\otimes_R\widehat{R}=\End_R(R\oplus N)\otimes_R\widehat{R}\cong \End_{\widehat{R}}(\widehat{R}\oplus \widehat{N}).
\]
On the other hand, considering $\Spec \widehat{R}$ and the formal fibre $f\colon \widehat{U}\to\Spec\widehat{R}$, it is also standard that the tilting bundle $\cV=\cO_U\oplus \cN$ on $U$ restricts to a tilting bundle $\cO_{\widehat{U}}\oplus\widehat{\cN}$ on $\widehat{U}$.  Then by birationality \cite[3.2.10]{VdB1d} 
\[
\End_{\widehat{U}}(\cO_{\widehat{U}}\oplus \widehat{\cN})\cong \End_{\widehat{R}}(f_*(\cO_{\widehat{U}}\oplus \widehat{\cN})),
\]
and this is clearly just $\End_{\widehat{R}}(\widehat{R}\oplus \widehat{N})$.  Hence passing to the formal fibre induces a derived equivalence
\[
\begin{array}{c}
\begin{tikzpicture}
\node (a1) at (0,0) {$\Db(\coh\widehat{U})$};
\node (a2) at (5,0) {$ \Db(\mod\widehat{\Lambda}).$};
\draw[->] (a1) -- node[above] {$\scriptstyle\RHom_{\widehat{U}}(\cO_{\widehat{U}}\oplus\widehat{\cN},-)$} node [below] {$\scriptstyle\sim$} (a2);
\end{tikzpicture}
\end{array}
\]

There is a more explicit tiling bundle on $\widehat{U}$, constructed in \cite{VdB1d}, whose associated tilting algebra, defined in \ref{def basic algebra} below, is the basic algebra morita equivalent to $\widehat{\Lambda}$.  For this, let $C=\pi^{-1}(\m)$ where $\m$ is the unique closed point of $\Spec \widehat{R}$, then giving $C$ the reduced scheme structure, we have $C^{\redu}\cong \mathbb{P}^1$.  Let $\cL$ denote the line bundle on $\widehat{U}$ such that $\cL\cdot C=1$.  If the multiplicity of $C$ is equal to one, set $\cM:=\cL$ and $\cN_1:=\cM^*$ \cite[3.5.4]{VdB1d}, else define $\cM$ to be given by the maximal extension
\[
0\to\cO_{\widehat{U}}^{\oplus(r-1)}\to\cM\to\cL\to 0
\]
associated to a minimal set of $r-1$ generators of $H^1(\widehat{U},\cL^{*})$ and set $\cN_1:=\cM^*$.  Then $\cO_{\widehat{U}}\oplus \cN_1$ is a tilting bundle on $\widehat{U}$ \cite[3.5.5]{VdB1d}.  Again by birationality 
\[
\End_{\widehat{U}}(\cO_{\widehat{U}}\oplus \cN_1)\cong \End_{\widehat{R}}(f_*(\cO_{\widehat{U}}\oplus \cN_1)),
\]
and we have $f_*(\cO_{\widehat{U}}\oplus \cN_1)=\widehat{R}\oplus N_1$ where $N_1:=f_*\cN_1$.   We remark that $\rank_{\widehat{R}}N_1$ is equal to the scheme-theoretic multiplicity of the curve $C$ \cite[3.5.4]{VdB1d}.  

\begin{lemma}\label{decomp into irred 1}
We can write $\widehat{R}\oplus\widehat{N}\cong \widehat{R}^{\oplus a_0}\oplus N_1^{\oplus a_1}$ for some $a_0,a_1\in\mathbb{N}$.
\end{lemma}
\begin{proof}
Since $\cO_{{U}}\oplus{\cN}\in\Per{U}$ is a progenerator \cite[3.2.7]{VdB1d}, $\cO_{\widehat{U}}\oplus\widehat{\cN}\in\Per\widehat{U}$ is a progenerator, so by \cite[3.5.5]{VdB1d}, there exist positive integers such that
\[
\cO_{\widehat{U}}\oplus\widehat{\cN}\cong\cO_{\widehat{U}}^{\oplus a_0}\oplus \cN_1^{\oplus a_1}.
\]
Pushing down, it follows that $\widehat{R}\oplus\widehat{N}\cong \widehat{R}^{\oplus a_0}\oplus N_1^{\oplus a_1}$.
\end{proof}

To ease notation with hats, we put $\mathfrak{R}:=\widehat{R}$.

\begin{defin}\label{def basic algebra}
We define $\AB:=\End_{\mathfrak{R}}(\mathfrak{R}\oplus N_1)$, the basic algebra morita equivalent to $\widehat{\Lambda}$, and define the {\em contraction algebra associated to $C$} to be $\CA:=\End_{\mathfrak{R}}(\mathfrak{R}\oplus N_1)/[\mathfrak{R}]$, where $[\mathfrak{R}]$ is defined similarly as in \ref{contraction def lambda}.
\end{defin}

We show later in \ref{contraction alg isos} that the contraction algebra $\CA$ is intrinsic to the geometry of $X$, and in particular does not depend on the choice of derived equivalence in \eqref{derived equivalence}.  It is not hard to see, and we will prove it later in \ref{morita lemma}, that $\CA$ is morita equivalent to $\widehat{\Lambda}_{\con}$.    However, the morita equivalence between $\AB$ and $\widehat{\Lambda}$ is easy to describe, and we do this now. We write
\[
Y:=\mathfrak{R}\oplus N_1, \qquad Z:=\mathfrak{R}^{\oplus a_0}\oplus N_1^{\oplus a_1},
\]
so that $\AB=\End_{\mathfrak{R}}(Y)$ and $\widehat{\Lambda}=\End_{\mathfrak{R}}(Z)$. Then we put
\[
P:=\Hom_{\mathfrak{R}}(Y,Z), \qquad Q:=\Hom_{\mathfrak{R}}(Z,Y).
\]
These have the structure of bimodules, namely ${}_{\widehat{\Lambda}}P_{\InSpaceOf{\widehat{\Lambda}}{\AB}}$ and ${}_{\InSpaceOf{\widehat{\Lambda}}{\AB}}Q_{\widehat{\Lambda}}$. It is clear that $P$ is a progenerator, and we have a morita context
\[
(\AB,\widehat{\Lambda},P,Q).
\]
In particular this implies that
\[
\begin{array}{c}
\begin{tikzpicture}[xscale=1]
\node (d1) at (3,0) {$\mod \AB$};
\node (e1) at (7.5,0) {${}_{}\mod\widehat{\Lambda}$};
\draw[->,transform canvas={yshift=+0.4ex}] (d1) to  node[above] {$\scriptstyle \mathbb{F}:=\Hom_{\AB}(P,-)=-\otimes_{\AB}Q $} (e1);
\draw[<-,transform canvas={yshift=-0.4ex}] (d1) to node [below]  {$\scriptstyle \Hom_{\widehat{\Lambda}}(Q,-)=-\otimes_{\widehat{\Lambda}}P $} (e1);
\end{tikzpicture}
\end{array}
\]
are equivalences.

We now show that, in our situation of flips and flops, the algebras $\Lambda_{\con}$ and $\CA$ are finite dimensional. This will be used later to obtain numerical invariants associated to the curve $C$.

\begin{prop}\label{cont is fd for flop}
With the setup as in \S\ref{setup section 2.1}, and with notation as above,
\begin{enumerate}
\item\label{cont is fd for flop 1} $\Lambda_{\con}$ and $\CA$ are finite dimensional algebras,  which as $R$-modules are supported only at $\m$.
\item\label{cont is fd for flop 2} As an $R$-module, $\FCA$ is supported only at $\m$.
\item\label{cont is fd for flop 3} $T$ is a simple $\Lambda$-module, which as an $R$-module is supported only at $\m$.
\end{enumerate}
\end{prop}
\begin{proof}
(1) The contraction $U\to\Spec R$ is an isomorphism away from a single point $\m\in\Max R$.  By base change, it is clear that $\add N_\p=\add R_\p$ for all $\p\neq\m$, so ${\Lambda_{\con}}_\p=0$ for all $\p\neq \m$.  It follows that $\Lambda_{\con}$ is supported only on $\m$, and so in particular is finite dimensional.  The proof for $\CA$ is identical, using the fact that $\widehat{U}\to\Spec\widehat{R}$ is an isomorphism away from the closed point.\\
(2) follows immediately from \eqref{cont is fd for flop 1}.\\
(3) The fact that $E$ is simple in $\Per\widehat{U}$ is \cite[3.5.8]{VdB1d}.  Thus $\widehat{T}$ is a simple $\widehat{\Lambda}$-module, since it corresponds to a simple module via a morita equivalence.    Now let $e$ denote the idempotent in $\Lambda$ corresponding to $R$, then
\[
\begin{array}{c}
\begin{tikzpicture}
\node (roof) at (0,0) {$\Db(\coh U)$};
\node (X) at (5,0) {$\Db(\mod\Lambda)$};
\node (X') at (0,-1.5) {$\Db(\coh \Spec R)$};
\node (base) at (5,-1.5) {$\Db(\mod R)$};
\draw[->] (roof) -- node[above] {$\scriptstyle \RHom_U(\cV,-)$} node[below] {$\scriptstyle \sim$} (X);
\draw[->] (roof) --  node[right] {$\scriptstyle \Rf_*$}  (X');
\draw[->] (X) --  node[right] {$\scriptstyle (-)e$} (base);
\draw[-,transform canvas={yshift=+0.15ex}] (X') -- (base);
\draw[-,transform canvas={yshift=-0.15ex}] (X') -- (base);
\end{tikzpicture}
\end{array}
\]
commutes. Since $\Rf_*E=0$, it follows that $Te=0$, and so $T$ is a finitely generated $\Lambda_{\con}$-module.  In particular, by \eqref{cont is fd for flop 1} $T$ is supported only at $\m$, so $T\cong\widehat{T}$.     It follows that $T$ is a simple $\Lambda$-module.
\end{proof}

\begin{defin}
We denote by $S$ the simple $\AB$-module corresponding across the morita equivalence to $\widehat{T}$.
\end{defin}

Note that since $\AB$ is basic, $\dim_{\mathbb{C}}S=1$.  It is also clear that $S$ can be viewed as a simple $\CA$-module, and it is the unique simple $\CA$-module.  Under the running flips and flops setting, $\CA$ is further finite dimensional by \ref{cont is fd for flop}\eqref{cont is fd for flop 1}.  Thus $\CA\in\art_1$, making it the candidate for the representing object.

\subsection{Reduction to complete local setting}\label{complete local DB section}
The proof that $\CA$ represents the deformation functors in \ref{local to global deformations the same} and \ref{local to alg deformations the same} requires one more reduction step, namely we must relate them to a similar deformation functor on the formal fibre.  This step is largely routine, although in this subsection we do introduce notation and known results that will be used later.

The following is well known, and will be used throughout.
\begin{lemma}\label{useful finite length stuff} Let $X\in\Db(\mod\Lambda)$, $Y\in\D(\Mod\Lambda)$, then 
\begin{enumerate}
\item\label{useful finite length stuff 1} $\RHom_{\Lambda}(X,Y)\otimes_RR_{\m}\cong \RHom_{\Lambda_{\m}}(X_\m,Y_\m)$.
\item\label{useful finite length stuff 2} $\RHom_{\Lambda}(X,Y)_{\m}\otimes_{R_\m}\widehat{R}_\m\cong \RHom_{\widehat{\Lambda}}(\widehat{X},\widehat{Y})$.
\item\label{useful finite length stuff 3} If $X\in\mod\Lambda$ with $\dim_{\mathbb{C}}X<\infty$, then $X\cong \bigoplus_{\m\in\Supp_R X}X_{\m}\cong \bigoplus_{\m\in\Supp_R X}\widehat{X}$ as $\Lambda$-modules.
\end{enumerate}
\end{lemma}
\begin{proof}
(1) and (2) are \cite[p1100]{IR}. \\
(3) The natural map  $\psi\colon X\to \bigoplus_{\m\in\Supp_R X}X_{\m}$ is clearly a $\Lambda$-module homomorphism.  Viewing $\psi$ as an $R$-module homomorphism, since $X$ has finite length, $\psi$ is bijective by \cite[2.13b]{EisenbudBook}.  It follows that $\psi$ is a $\Lambda$-module isomorphism.  The last isomorphism follows since if $X_\m$ has finite length, then $X_\m\cong\widehat{X}$ as $\Lambda_\m$-modules.
\end{proof}

Now $\Lambda=\End_R(R\oplus N)$ and $\Lambda_{\con}=\Lambda/[R]$.   Completing with respect to $\m$,  $
\widehat{\Lambda}\cong\End_{\widehat{R}}(\widehat{R}\oplus \widehat{N})$ and we already know from \S\ref{complete local geometry summary} that $\widehat{N}$ may decompose into many more summands than $N$ does.  However, the following states that the localization (resp.\ completion) of the contraction algebra associated to $\Lambda$ is the contraction algebra associated to the localization $\Lambda_\m$ (resp.\ completion $\widehat{\Lambda}$).  This will allow us to reduce many of our problems to the formal fibre.

\begin{lemma}\label{loc and comp of contraction} With notation as above,
\begin{enumerate}
\item\label{loc and comp of contraction 1} $[R]\otimes_RR_\m\cong[R_\m]$.  We denote this by ${I_{\m}}_{\con}$.
\item\label{loc and comp of contraction 2} $[R_\m]\otimes_R\widehat{R}\cong [\widehat{R}]$.  We denote this by $\widehat{I}_{\con}$.
\item\label{loc and comp of contraction 3} ${\Lambda_{\con}}\otimes_RR_{\m}\cong \Lambda_\m/[R_\m]$ and $\Lambda_{\con}\otimes_R\widehat{R}\cong \widehat{\Lambda}/[\widehat{R}]$.
\end{enumerate}
\end{lemma}
\begin{proof}
Since localization and completion are exact, \eqref{loc and comp of contraction 3} follows from the first two.  We prove \eqref{loc and comp of contraction 2}, with \eqref{loc and comp of contraction 1} being similar.  To ease notation, we temporarily write $M:=R\oplus N$.

It is clear that  $[R_\m]\otimes_R\widehat{R}\subseteq [\widehat{R}]$, so let $\psi\in [\widehat{R}]$.  Then $\psi\colon\widehat{M}\to\widehat{M}$ factors through $\add\widehat{R}$, so by definition $\psi$ factors through a summand of $\widehat{R}^{a}$ for some $a\in\mathbb{N}$.  This implies that $\psi$ factors through $\widehat{R}^{a}$, so there is a commutative diagram
\[
\begin{tikzpicture}
\node (A) at (0,0) {$\widehat{M}$};
\node (B2) at (1.5,-0.75) {$\widehat{R}^a$};
\node (C) at (3,0) {$\widehat{M}$};
\draw[->] (A) -- node[below] {$\scriptstyle \alpha$} (B2);
\draw[->] (B2) -- node[below] {$\scriptstyle \beta$} (C);
\draw[->] (A) -- node[above] {$\scriptstyle \psi$}(C);
\end{tikzpicture}
\]
As in \ref{useful finite length stuff}\eqref{useful finite length stuff 2} $\Hom_{\Lambda_\m}(M_\m,R_\m^a)\otimes_{R_\m}\widehat{R}\cong\Hom_{\widehat{\Lambda}}(\widehat{M},\widehat{R}^a)$ etc, so we can find $\Lambda_\m$-module homomorphisms
\begin{eqnarray}
\begin{array}{c}
\begin{tikzpicture}
\node (A) at (0,0) {$M_\m$};
\node (B2) at (1.5,-0.75) {${R}_\m^a$};
\node (C) at (3,0) {$M_\m$};
\draw[->] (A) -- node[below] {$\scriptstyle g$} (B2);
\draw[->] (B2) -- node[below] {$\scriptstyle h$} (C);
\draw[->] (A) -- node[above] {$\scriptstyle f$}(C);
\end{tikzpicture}
\end{array}\label{comm for local}
\end{eqnarray}
such that $\widehat{f}=\psi$, $\widehat{g}=\alpha$ and $\widehat{h}=\beta$. It remains to show that \eqref{comm for local} commutes, since then $f\in[R_\m]$.  But since the completion is exact
\[
\Im(f-hg)\otimes \widehat{R}\cong\Im\widehat{f-hg}=0,
\]
so the completion of $\Im(f-hg)$ is zero, which implies that $\Im(f-hg)=0$, i.e.\ $f=hg$. 
\end{proof}

Below, we write ${\Lambda_{\m}}_{\con}$ for the contraction algebra associated to $\Lambda_\m$, and  $\widehat{\Lambda}_{\con}$ for the contraction algebra associated to $\widehat{\Lambda}$.  By the above, ${\Lambda_{\m}}_{\con}\cong(\Lambda_{\con})_\m$ and $\widehat{\Lambda}_{\con}\cong\widehat{(\Lambda_{\con})}$, so there is no ambiguity.

\begin{lemma}\label{contraction splits}
$\Lambda_{\con}\cong\widehat{\Lambda}_{\con}$ both as $\Lambda$-modules and as algebras.  \end{lemma}
\begin{proof}
Since $\Lambda_{\con}$ is supported only at $\m$ by \ref{cont is fd for flop}, it follows from \ref{loc and comp of contraction} and \ref{useful finite length stuff}\eqref{useful finite length stuff 3} that the natural map $\Lambda_{\con}\to\widehat{\Lambda}_{\con}$ is bijective.  Since it is both a $\Lambda$-module homomorphism and an algebra homomorphism, we have $\Lambda_{\con}\cong\widehat{\Lambda}_{\con}$ both as $\Lambda$-modules, and as algebras. 
\end{proof}

\begin{warning}
Often when studying the birational geometry of global 3-folds,  we are forced to flip or flop together several irreducible curves that don't intersect.  When this occurs, $\Lambda_{\con}\ncong\widehat{\Lambda}_{\con}$, instead $\Lambda_{\con}\cong\bigoplus_{\m\in\Supp\Lambda_{\con}}\widehat{\Lambda}_{\con}$. 
\end{warning}

Hence we have the following functors
\[
\begin{array}{c}
\begin{tikzpicture}[xscale=1]
\node (d1) at (3,0.8) {$\Db(\Mod \Lambda)$};
\node (e1) at (6,0.8) {$\Db(\Mod \Lambda_\m)$};
\node (f1) at (9,0.8) {$\Db(\Mod \widehat{\Lambda})$};
\node (g1) at (12,0.8) {$\Db(\Mod \AB)$};
\draw[->] (d1.5) to  node[above] {$\scriptstyle \otimes_RR_\m$} (e1.175);
\draw[<-] (d1.-5) to node [below]  {$\scriptstyle\rest$} (e1.-175);
\draw[->] (e1.5) to  node[above] {$\scriptstyle \otimes_{R_\m}\widehat{R}$} (f1.175);
\draw[<-] (e1.-5) to node [below]  {$\scriptstyle\rest$} (f1.-175);
\draw[<->]  (f1) -- node[above] {$\scriptstyle \mathbb{F}$} (g1);
\node (d2) at (3,0) {$T$};
\draw[|->] (d1.east |- 3.5,0) -- (e1.west |- 5.5,0);
\node (e2) at (6,0) {$T_\m$};
\draw[|->] (e1.east |- 6.5,0) -- (f1.west |- 8.5,0);
\node (f2) at (9,0.08) {$\widehat{T}$};
\draw[<->] (f1.east |- 9.5,0) -- (g1.west |- 11.25,0);
\node (g2) at (12,0) {$S$};
\node (d3) at (3,-0.5) {$\Lambda_{\con}$};
\draw[|->] (d1.east |- 3.5,-0.5) -- (e1.west |- 5.25,-0.5);
\node (e3) at (6.1,-0.5) {$(\Lambda_\m)_{\con}$};
\draw[|->] (e1.east |- 6.85,-0.5) -- (f1.west |- 8.5,-0.5);
\node (f3) at (9.1,-0.5) {$\widehat{\Lambda}_{\con}$};
\node (d4) at (3,-1) {$\FCA$};
\draw[<-|] (d1.east |- 3.5,-1) -- (e1.west |- 5.25,-1);
\node (e4) at (6.1,-1) {$\FCA$};
\draw[<-|] (e1.east |- 6.75,-1) -- (f1.west |- 8.4,-1);
\node (f4) at (9.1,-1) {$\FCA$};
\draw[<->] (f1.east |- 9.75,-1) -- (g1.west |- 11.25,-1);
\node (g4) at (12,-1) {$\CA$};
\end{tikzpicture}
\end{array}
\]
where $\rest$ denotes restriction of scalars along the natural ring homomorphisms. The line for $\Lambda_{\con}$ follows by \ref{loc and comp of contraction}. Note that the first two functors are not equivalences,  however, since $T$, $\Lambda_{\con}$ and $\CA$ are supported only at $\m$ we have $T\cong\rest(T_\m)$ etc, so the arrows in the bottom three lines in the above diagram can also be drawn in the opposite direction.

Given this, the following is not surprising, and is a standard application of the proofs and techniques developed so far.
\begin{prop}\label{local to very local deformations the same}
There are natural isomorphisms 
\[
\Def^{\Lambda}_{T}\cong\Def^{\Lambda_\m}_{T_\m}\cong\Def^{\widehat{\Lambda}}_{\widehat{T}}\cong\Def^{\AB}_{S},
\] 
and consequently by restriction natural isomorphisms
\[
\cDef^{\Lambda}_{T}\cong\cDef^{\Lambda_\m}_{T_\m}\cong\cDef^{\widehat{\Lambda}}_{\widehat{T}}\cong\cDef^{\AB}_{S}.
\] 
\end{prop}
\begin{proof}
The functors $-\otimes_RR_\m\colon\Mod \Lambda\to \Mod\Lambda_\m$ and $\rest\colon\Mod\Lambda_\m\to\Mod\Lambda$ are adjoint, and are both exact.  Since, by \ref{cont is fd for flop}, $T$ is a finite length module supported only at $\m$, $T\cong \rest(T\otimes_RR_\m)$ and ${T_{\m}}\cong\rest({T_{\m}})\otimes_RR_\m$  via the unit and counit morphisms. 

The proof is now identical to \ref{local to global deformations the same}, where in the analogue of \eqref{comp is exact} and \eqref{comp is exact 22}, the second functors are replaced by $-\otimes_RR_\m$ and $\rest$, which are still exact. This establishes the first isomorphism.

For the second, again restriction and extension of scalars give an exact adjoint pair,  and since $T_\m$ is a finite length module, the necessary unit and counit maps are isomorphisms, since restriction and extension of scalars are an equivalence on finite length modules.  The proof is then identical to the above, establishing the second isomorphism. The third isomorphism also follows in an identical manner.
\end{proof}

Thus, in conclusion,  combining  \ref{local to global deformations the same}, \ref{local to alg deformations the same} and \ref{local to very local deformations the same} gives a natural isomorphism
\begin{eqnarray}
\Def^{X}_{i_*E}\cong\Def^{\AB}_{S}\colon \art_1\to \Sets,\label{all def functors iso 1}
\end{eqnarray}
and so to prove representability, we have reduced the problem to showing that $\Def^{\AB}_{S}$ is representable.  Since $\AB$ is basic (unlike $\Lambda$), by \ref{cont is fd for flop} $\CA\in\art_1$ and so this is the natural candidate for the representing object.

\section{Representability and Contraction Algebras}\label{flops section}

We keep the setup in \S\ref{setup section 2.1}.  In the previous section we established natural isomorphisms of functors
\begin{eqnarray}
\Def^{X}_{i_*E}\cong\Def^{U}_{E}\cong\Def^{\AB}_{S}\colon \art_1\to \Sets\label{all def functors iso}
\end{eqnarray}
and showed that $\CA\in\art_1$. In this section we will show that $\CA$ is a representing object for these functors, and describe the associated universal families.  This then allows us to define numerical invariants associated to the curve $C$, and show that Toda's commutative deformation functor is not an equivalence in general.

\subsection{Representability on the basic algebra $\AB$}

In this subsection we show that $\Def^{\AB}_{S}$ and $\cDef^{\AB}_S$ are representable.  This statement turns out to be an elementary consequence of the definitions, and is implicit in \cite{Eriksen2, Eriksen,Segal}, albeit in a slightly modified setting.
\begin{prop}\label{complete is representable} $
\Def^{\AB}_{S}\cong \Hom_{\art_1}(\CA,-)\colon\art_1\to\Sets$.
\end{prop}
\begin{proof}
Since $S=\mathbb{C}$ is an $\AB$-module, this induces a homomorphism $\AB\to\mathbb{C}$ which we denote by $q$.   Since objects in $\PairsCat{\Mod\AB}{\Gamma}$ are bimodules ${}_{\Gamma}M_{\AB}$,
\begin{align*}
\Def^{\AB}_{S}(\Gamma)
&=\left. \left \{ ({}_{\Gamma}M_{\AB},\delta)
\left|\begin{array}{l}  -\otimes_\Gamma M\colon\mod \Gamma\to \mod\AB \t{ is exact}\\ \delta\colon (\Gamma/\n)\otimes_\Gamma M\xrightarrow{\sim}S\mbox{ as $\AB$-modules}\end{array}\right. \right\} \middle/ \sim \right. \\
&=\left. \left \{ \begin{array}{cl} \bullet & \mbox{a right $\AB$-module structure on $\Gamma\otimes_{\mathbb{C}}S$ such that}\\
&\mbox{$\Gamma\otimes_{\mathbb{C}}S$ becomes a $\Gamma$-$\AB$ bimodule.}\\
\bullet & \delta\colon (\Gamma/\n)\otimes_\Gamma (\Gamma\otimes_{\mathbb{C}}S)\xrightarrow{\sim}S\mbox{ as $\AB$-modules}
\end{array} \right\} \middle/ \sim \right. \\
&=\left. \left \{ \begin{array}{cl} \bullet & \mbox{a $\mathbb{C}$-algebra homomorphism $\AB\to\Gamma$ such that}\\
&\mbox{the composition $\AB\to\Gamma\to\Gamma/\n=\mathbb{C}$ is $q$}\end{array} \right\} \middle/  \sim  \right.
\end{align*}
which is just $\Hom_{\alg_1}(\CA,\Gamma)$ since $q$ factors through $\CA$. The result then follows since $\CA\in\art_1$ which is by definition a full subcategory of $\alg_1$.
\end{proof}

\begin{cor}\label{complete is representable comm}
$
\cDef^{\AB}_{S}\cong \Hom_{\cart_1}(\CA^{\ab},-)\colon\cart_1\to\Sets$.
\end{cor}
\begin{proof}
Since any homomorphism from $\CA$ to a commutative ring $\Gamma$ must kill the two-sided ideal in $\CA$ generated by the commutators, the result follows since 
\[
\cDef^{\AB}_S(\Gamma)\stackrel{{\scriptsize\mbox{\ref{complete is representable}}}}{=}\Hom_{\alg_1}(\CA,\Gamma)\cong\Hom_{\alg_1}(\CAab,\Gamma)=\Hom_{\cart_1}(\CA^{\ab},\Gamma).\qedhere
\]
\end{proof}

\subsection{Chasing through the representing couple}
\label{sect chasing rep couple}

It follows from \ref{complete is representable} that all the functors in \eqref{all def functors iso} are representable.  In this subsection we chase the representing couples through the relevant isomorphisms, as we need the universal sheaves in geometric terms for our applications later.
 
If $F\colon\art_1\to\Sets$ is a deformation functor, recall that $(\Gamma,\xi)$ is a {\em couple} for $F$ if $\Gamma\in\art_1$ and $\xi\in F(\Gamma)$.  A couple $(\Gamma,\xi)$ induces a natural transformation
\[
\alpha_\xi\colon\Hom_{\art_1}(\Gamma,-)\to F
\]
which when applied to $\Gamma_2\in\art_1$ simply takes $\phi\in\Hom_{\art_1}(\Gamma,\Gamma_2)$ to the element $F(\phi)(\xi)$ of $F(\Gamma_2)$.  We say that a couple $(\Gamma,\xi)$ {\em represents} $F$ if $\alpha_\xi$ is a natural isomorphism of functors on $\art_1$.  If a couple $(\Gamma,\xi)$ represents $F$, then it is unique up to unique isomorphism of couples.

Now given the natural isomorphism $\beta\colon
\Hom_{\art_1}(\CA,-)\stackrel{\sim}{\to}\Def^{\AB}_{S} $ in \ref{complete is representable}, the couple $(\CA,\beta(\Id_{\CA}))=(\CA,{}_{\CA}({\CA})_{\AB})$ represents $\Def^{\AB}_{S}$.  Simply tracking the element ${}_{\CA}({\CA})_{\AB}\in\Def^{\AB}_S(\CA)$ through the isomorphisms
\[
\Def^{\AB}_{S}\xrightarrow{\mathbb{F}}\Def^{\widehat{\Lambda}}_{\widehat{T}}
\xrightarrow{\rest}\Def^{\Lambda}_{T}
\xrightarrow{-\otimes_\Lambda\cV}\Def^{U}_{E}
\xrightarrow{i_*}\Def^{X}_{i_*E}
\]
we obtain the following:
\begin{cor}\label{chase couples 1} In the setup of \S\ref{setup section 2.1},
\begin{enumerate}
\item\label{chase couples 1 1} The couple $(\CA,\FCA\otimes_{\Lambda}\cV)$ represents $\Def^{U}_{E}$.
\item\label{chase couples 1 2} The couple $(\CA,i_*(\FCA\otimes_{\Lambda}\cV))$ represents $\Def^{X}_{i_*E}$.
\end{enumerate}
\end{cor}

We now fix notation.

\begin{defin}\label{cE defin}
We define $\cEU:=\FCA\otimes_{\Lambda}\cV\in\coh U$, and $\cE:=i_*(\cEU)$.
\end{defin}

Thus $(\CA,\cE)$ is the representing couple for $\Def^{X}_{i_*E}$.   Similarly to the above, we can deduce from \ref{complete is representable comm} that all the commutative deformation functors have representing couples, with the left-hand term $\CAab$.  Chasing through, it is obvious that the representing couple for $\cDef^U_{E}$ is given by a coherent sheaf. We again fix notation.

\begin{defin}
We write $\cFU\in\coh U$ for the coherent sheaf such that $(\CAab,\cFU)$ represents $\cDef^U_{E}$, and write $\cF:=i_*\cFU$ so that $(\CAab,\cF)$ represents $\cDef^X_{i_*E}$.
\end{defin}

Since $i_*$ need not preserve coherence, we must prove that $\cE$ and $\cF$ are coherent.

\begin{lemma}\label{highers vanish 1} With assumptions and notation as above,
\begin{enumerate}
\item\label{highers vanish 1 1} $\cFU$ and $\cEU$ are filtered by the sheaf $E$.
\item\label{highers vanish 1 2} $\Ri_* \cEU=i_*\cEU=\cE$ and  $\Ri_* \cFU=i_*\cFU=\cF$.
\item\label{highers vanish 1 3} $\cF$ and $\cE$ are filtered by the sheaf $i_*E$, hence are coherent.
\end{enumerate}
\end{lemma}
\begin{proof}
Denote the radical of $\CA$ by $J$.\\
(1) Since $\cEU\in\Def^U_{E}(\CA)$, the functor $-\otimes_{\CA}\cEU\colon\mod\CA\to\Qcoh U$ is exact, and $(\CA/J)\otimes_{\CA}\cEU\cong E$.  Since $\CA$ is a finite dimensional algebra, filtering $\CA$ by its simple module $(\CA/J)$ and applying  $-\otimes_{\CA}\cEU$ yields the result.  The case of $\cFU$ is identical.\\
(2) follows from \eqref{inclusion line}, and (3) follows by combining \eqref{highers vanish 1 1} and \eqref{highers vanish 1 2}.
\end{proof}

In a similar vein, using the fact $\Hom_U(\cV,E)=\RHom_U(\cV,E)$ in \ref{T definition} together with a filtration argument gives the following, which will be used later.

\begin{lemma}\label{highers vanish 2} With assumptions and notation as above,
\begin{enumerate}
\item\label{highers vanish 2 1} $\RHom_U(\cV,\cEU)\cong \FCA$ and $\FCA\otimes_{\Lambda}^{\bf L}\cV\cong\cEU$.
\item\label{highers vanish 2 2} $\RHom_U(\cV,\cFU)\cong \FCA^{\ab}$ and $\FCA^{\ab}\otimes_{\Lambda}^{\bf L}\cV\cong\cFU$.
\end{enumerate}
\end{lemma}

\subsection{The contraction algebra of $C$ and $\width(C)$}

\begin{defin}\label{contraction def}
With the setup and assumptions as in \S\ref{setup section 2.1}, suppose that $X\to X_{\con}$ is the contraction of an irreducible rational curve $C$.    We define the {\em noncommutative deformation algebra} of $C$ to be $\End_X(\cE)$.
\end{defin}

The noncommutative deformation algebra is the fundamental object in our paper.  The above definition is intrinsic to the geometry of $X$, since $\cE$ arises from the representing couple of the noncommutative deformations of the sheaf $\cO_{\mathbb{P}^1}(-1)$ in $X$.  However, it is hard to calculate in the above form. It is the following description of $\End_X(\cE)$ as the contraction algebra $\CA=\End_{\widehat{R}}(\widehat{R}\oplus N_1)/[\widehat{R}]$ that both makes it possible to calculate $\End_X(\cE)$, and also to control it homologically.

\begin{lemma}\label{contraction alg isos} There are isomorphisms
\[
\End_X(\cE)\cong\End_U(\cEU)\cong\End_\Lambda(\FCA)\cong\End_{\widehat{\Lambda}}(\FCA)\cong \End_{\AB}(\CA)\cong \CA.
\]
Furthermore, $\End_X(\cF)\cong\End_X(\cE)^{\ab}$.
\end{lemma}
\begin{proof}
The first isomorphism uses the fully faithful embedding of derived categories $\Ri_*$, together with the fact that $\Ri_*\cEU\cong \cE$ by \eqref{inclusion line}.  The second isomorphism follows from the equivalence \eqref{derived equivalence} together with \ref{highers vanish 2}.  The third isomorphism follows since $\FCA$ is supported only at $\m$ (by \ref{cont is fd for flop}), and the fact that restriction and extension of scalars are an equivalence on these finite length subcategories.  The fourth isomorphism is just chasing across a morita equivalence, and the last is simply
\[
\End_{\AB}(\CA)=\End_{\AB}(\AB/\AB e\AB)\cong \AB/\AB e\AB=\CA
\] 
where $e$ is the idempotent in $\AB$ corresponding to $\widehat{R}$.

For the final statement that $\End_X(\cF)\cong\End_X(\cE)^{\ab}$, observe that
\opt{10pt}{\[
\CA^{\ab}\cong\Hom_{\AB}(\AB,\CA^{\ab})\cong\Hom_{\AB}(\CA,\CA^{\ab})\cong\End_{\AB}(\CA^{\ab})
\stackrel{{\scriptsize\mbox{\ref{highers vanish 2}}}}{\cong} \End_U(\cFU)
\stackrel{{\scriptsize\mbox{\ref{highers vanish 1}}}}{\cong} \End_X(\cF)
\]}
\opt{12pt}{\begin{align*}
\CA^{\ab}\cong\Hom_{\AB}(\AB,\CA^{\ab})\cong\Hom_{\AB}(\CA,\CA^{\ab})&\cong\End_{\AB}(\CA^{\ab})\\
&\stackrel{{\scriptsize\mbox{\ref{highers vanish 2}}}}{\cong} \End_U(\cFU)\\
&\stackrel{{\scriptsize\mbox{\ref{highers vanish 1}}}}{\cong} \End_X(\cF)
\end{align*}}
and chasing through this establishes that $\CA^{\ab}\cong\End_X(\cF)$ as rings. 
\end{proof}

Thus throughout the remainder of the paper, we use the terms `contraction algebra' and `noncommutative deformation algebra' interchangeably.

\begin{example}\label{Pagoda example} (Pagoda \cite{Pagoda})  Consider 
\[
\widehat{R}:=\frac{\mathbb{C}[[u,v,x,y]]}{uv=(x-y^n)(x+y^n)}
\]
for some $n\geq 1$.  It is well known that $N_1\cong(u,x+y^n)$.

We can present $\AB:=\End_{\widehat{R}}(\widehat{R}\oplus N_1)$ as the completion of the quiver with relations
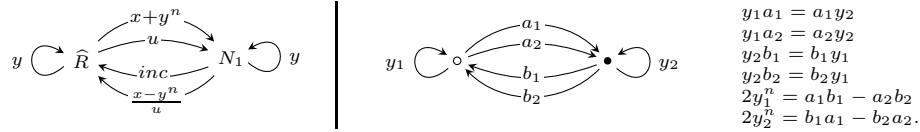
\[
\begin{array}{c|c}
\begin{array}{c}
\begin{tikzpicture} [bend angle=45, looseness=1,>=stealth]
\node (C1) at (0,0)  {$\scriptstyle \widehat{R}$};
\node (C1a) at (-0.1,0)  {};
\node (C2) at (2,0)  {$\scriptstyle N_{1}$};
\node (C2a) at (2.1,0.05) {};
\draw [->,bend left] (C1) to node[gap]  {$\scriptstyle x+y^n$} (C2);
\draw [->,bend left=20,looseness=1] (C1) to node[gap]{$\scriptstyle u$} (C2);
\draw [->,bend left] (C2) to node[gap]{$\scriptstyle \frac{x-y^n}{u}$} (C1);
\draw [->,bend left=20,looseness=1] (C2) to node[gap,below=-5pt]{$\scriptstyle inc$} (C1);
\draw[->]  (C1a) edge [in=150,out=220,loop,looseness=10,pos=0.5] node[left] {$\scriptstyle y$} (C1a);
\draw[->]  (C2a) edge [in=30,out=-40,loop,looseness=10,pos=0.5] node[right] {$\scriptstyle y$} (C2a);
\end{tikzpicture}
\end{array} 
&
\begin{array}{cc}
\begin{array}{c}
\begin{tikzpicture} [bend angle=45, looseness=1,>=stealth]
\node (C1) at (0,0) [cvertex] {};
\node (C2) at (2,0)  [vertex] {};
\node (C1a) at (-0.1,0.05)  {};
\node (C1b) at (-0.1,-0.05)  {};
\node (C2a) at (2.1,0.05) {};
\node (C2b) at (2.1,-0.05) {};
\draw [->,bend left] (C1) to node[gap]  {$\scriptstyle a_{1}$} (C2);
\draw [->,bend left=20,looseness=1] (C1) to node[gap]  {$\scriptstyle a_{2}$} (C2);
\draw [->,bend left] (C2) to node[gap]  {$\scriptstyle b_{2}$} (C1);
\draw [->,bend left=20,looseness=1] (C2) to node[gap,below=-5pt]  {$\scriptstyle b_{1}$} (C1);
\draw[->]  (C1) edge [in=150,out=220,loop,looseness=10,pos=0.5] node[left] {$\scriptstyle y_1$} (C1);
\draw[->]  (C2) edge [in=30,out=-40,loop,looseness=10,pos=0.5] node[right] {$\scriptstyle y_2$} (C2);
\end{tikzpicture}
\end{array}
&
{\scriptsize{
\begin{array}{l}
y_1a_1=a_1y_2 \\
y_1a_2=a_2y_2\\
y_2b_1=b_1y_1\\
y_2b_2=b_2y_1\\
2y_1^n=a_1b_1-a_2b_2\\
2y_2^n=b_1a_1-b_2a_2.
\end{array}}}
\end{array} 
\end{array}
\]
To obtain $\CA=\End_{\widehat{R}}(\widehat{R}\oplus N_1)/[\widehat{R}]$ we factor by all arrows that factor through the vertex $\widehat{R}$, and so $\CA\cong \mathbb{C}[[y_2]]/2y_2^n\cong \mathbb{C}[y]/y^n$.
\end{example}

\begin{defin}\label{width def}
Let $X\to X_{\con}$ be a contraction of an irreducible rational curve $C$ and let $\cE\in\coh X$ be as defined in \ref{cE defin}.   We define
\begin{enumerate}
\item the {\em width} of $C$ to be $\width(C):=\dim_{\mathbb{C}}\End_X(\cE)$.
\item the {\em commutative width} of $C$ to be $\cwidth(C):=\dim_{\mathbb{C}}\End_X(\cE)^{\ab}$.
\end{enumerate}
\end{defin}
We remark that the width of a $(-3,1)$-curve is not defined \cite{Pagoda}, and so in this case (as well as in the singular setting, and the flipping setting) the invariants in \ref{width def} are new.  In contrast to determining an explicit presentation for $\End_X(\cE)$, which sometimes can be hard, to calculate $\width(C)$ is much simpler, and can be achieved by using just commutative algebra on the base singularity $\widehat{R}$.  We refer the reader to \ref{can calculate width} later.

\begin{example} (Pagoda)
In \ref{Pagoda example}, $\CA\cong\K[y]/y^n$ which is commutative, so
\[
\width(C)=\cwidth(C)=n.
\] 
\end{example}

\begin{example}
(Francia flip) Consider the ring $R:=\K[[a,b,c_1,c_2,d]]/I$ where $I$ is generated by the $2\times 2$ minors of the following matrix:
\[
\begin{pmatrix}
a&b&c_1\\
b&c_2&d
\end{pmatrix}
\]
There is a flip
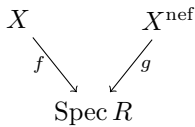
\[
\begin{array}{c}
\begin{tikzpicture}[yscale=1.25]
\node (X) at (-1,0) {$X$};
\node (X') at (1,0) {$X^{\nef}$};
\node (base) at (0,-1) {$\Spec R$};
\draw[->] (X) --  node[left] {$\scriptstyle f$} (base);
\draw[->] (X') -- node[right] {$\scriptstyle g$} (base);
\end{tikzpicture}
\end{array}
\]
contracting a curve $C$ in $X$ and producing a nef curve $C^{\nef}$ in $X^{\nef}$.  The scheme $X$ is obtained by blowing up the ideal $S_1:=(c_2,d^2)$ whereas $X^{\nef}$ is obtained by blowing up the ideal $S_2:=(c_1,d)$.  In fact $X$ is derived equivalent to $\AB:=\End_R(R\oplus S_1)$, which can be presented as 
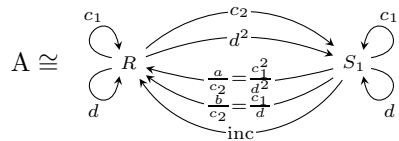
\[
\begin{array}{c}
\opt{10pt}{\begin{tikzpicture} [looseness=1,>=stealth]}
\opt{12pt}{\begin{tikzpicture} [scale=1.15,looseness=1,>=stealth]}
\node at (-1.25,0) {$\AB\cong$};
\node (C1) at (0,0)  {$\scriptstyle R$};
\node (C1a) at (-0.1,0.05)  {};
\node (C1b) at (-0.1,-0.05)  {};
\node (C2) at (3,0)  {$\scriptstyle S_{1}$};
\node (C2a) at (3.1,0.05) {};
\node (C2b) at (3.1,-0.05) {};
\draw [->,bend left=55,looseness=1.1] (C2) to node[gap]{$\scriptstyle \inc$} (C1);
\draw [->,bend left=35] (C2) to node[gap]{$\scriptstyle \scriptstyle \frac{b}{c_2}=\frac{c_1}{d}$} (C1);
\draw [->,bend left=15] (C2) to node[gap,above=-5pt]{$\scriptstyle \frac{a}{c_2}=\frac{c_1^2}{d^2}$} (C1);
\draw [->,bend left=40] (C1) to node[gap]  {$\scriptstyle c_2$} (C2);
\draw [->,bend left=20] (C1) to node[gap]{$\scriptstyle d^2$} (C2);
 \draw[<-]  (C1b) edge [in=-170,out=-100,loop,looseness=12] node[below] {$\scriptstyle d$} (C1b);
\draw[<-]  (C1a) edge [in=170,out=100,loop,looseness=12] node[above] {$\scriptstyle c_1$} (C1a);
\draw[->]  (C2b) edge [in=-80,out=-10,loop,looseness=12] node[below] {$\scriptstyle d$} (C2b);
\draw[->]  (C2a) edge [in=80,out=10,loop,looseness=12] node[above] {$\scriptstyle c_1$} (C2a);
\end{tikzpicture}
\end{array} 
\]
At the vertex $S_1$, the loops $c_1$ and $d$ commute since they belong to $\End_R(S_1)\cong R$, which is commutative.  Also, at the vertex $S_1$, the path $c_1^2$ factors through the vertex $R$ as the arrow $\frac{c_1^2}{d^2}$ followed by the arrow $d^2$.  Similarly $c_1d$ and $d^2$ factor through the vertex $R$, so
\[
\CA\cong \frac{\K[[c_1,d]]}{c_1^2\mbox{, }c_1d\mbox{, }d^2}\cong \frac{\K[x,y]}{x^2\mbox{, }xy\mbox{, }y^2}.
\]
In particular $\CA$ is commutative, thus
\[
\width(C)=\cwidth(C)=3.
\]
On the other hand, $X^{\nef}$ is derived equivalent to $\BB:=\End_R(R\oplus S_2)$ which can be presented as 
\[
\begin{array}{c}
\opt{10pt}{\begin{tikzpicture} [looseness=1,>=stealth]}
\opt{12pt}{\begin{tikzpicture} [scale=1.15,looseness=1,>=stealth]}
\node at (-1,0) {$\BB\cong$};
\node (C1) at (0,0)  {$\scriptstyle R$};
\node (C1a) at (-0.1,0.05)  {};
\node (C1b) at (-0.1,-0.05)  {};
\node (C2) at (3,0)  {$\scriptstyle S_{2}$};
\node (C2a) at (3.1,0.05) {};
\node (C2b) at (3.1,-0.05) {};
\draw [->,bend left=55,looseness=1.1] (C2) to node[gap]{$\scriptstyle \inc$} (C1);
\draw [->,bend left=35] (C2) to node[gap]{$\scriptstyle \scriptstyle \frac{b}{c_1}=\frac{c_2}{d}$} (C1);
\draw [->,bend left=15] (C2) to node[gap,above=-5pt]{$\scriptstyle \frac{a}{c_1}=\frac{b}{d}$} (C1);
\draw [->,bend left=40] (C1) to node[gap]  {$\scriptstyle c_1$} (C2);
\draw [->,bend left=20] (C1) to node[gap]{$\scriptstyle d$} (C2);
\end{tikzpicture}
\end{array} 
\]
Thus we see $\CB=\K$, which is commutative, and $\width(C^{\nef})=\cwidth(C^{\nef})=1$.
\end{example}

\begin{example}\label{D4 flop general example} (The Laufer $D_4$ flop) 
We consider the more general case than in the introduction, namely the family of $D_4$ flops given by 
\[
R_n:=\frac{\mathbb{C}[[u,v,x,y]]}{u^2+v^2y=x(x^2+y^{2n+1})}
\] 
where $n\geq 1$.  In this case $\AB:=\End_R(R\oplus N_1)$ where $N_1$ is a rank 2 Cohen--Macaulay module, and by \cite{AM}, $\AB$ can be written abstractly as the completion of the quiver with relations
\[
\begin{array}{cc}
\begin{array}{c}
\begin{tikzpicture}[>=stealth]
\node (C1) at (0,0) {$\scriptstyle R$};
\node (C1a) at (-0.1,0)  {};
\node (C2) at (1.75,0) {$\scriptstyle N_1$};
\node (C2a) at (1.85,0.05) {};
\node (C2b) at (1.85,-0.05) {};
\draw [->,bend left=20,looseness=1,pos=0.5] (C1) to node[gap]  {$\scriptstyle a$} (C2);
\draw [->,bend left=20,looseness=1,pos=0.5] (C2) to node[gap]  {$\scriptstyle b$} (C1);
\draw[->]  (C1a) edge [in=150,out=220,loop,looseness=10,pos=0.5] node[left] {$\scriptstyle x$} (C1a);
\draw[->]  (C2b) edge [in=-80,out=-10,loop,looseness=12,pos=0.5] node[below] {$\scriptstyle y$} (C2b);
\draw[->]  (C2a) edge [in=80,out=10,loop,looseness=12,pos=0.5] node[above] {$\scriptstyle x$} (C2a);
\end{tikzpicture}
\end{array}
&
{\scriptsize{
\begin{array}{l}
ay^{2}=-xa\\
y^{2}b=-bx\\
ab=-x^n\\
xy=-yx\\
x^2+yba+bay=(-1)^{n+1}y^{2n+1}.
\end{array}}}
\end{array}
\] 
Thus we see immediately that
\[
\CA\cong \frac{\mathbb{C}\langle\langle x,y\rangle\rangle}{xy=-yx\mbox{, }x^2=(-1)^{n+1}y^{2n+1}}\cong \frac{\mathbb{C}\langle\langle x,y\rangle\rangle}{xy=-yx\mbox{, }x^2=y^{2n+1}}.
\] 
The contraction algebra $\CA$ is finite dimensional by \ref{cont is fd for flop}.  To see this explicitly, post- and premultiplying the second equation by $x$ we obtain $x^3=y^{2n+1}x=xy^{2n+1}$.  But using the first equation repeatedly gives $xy^{2n+1}=-y^{2n+1}x$, thus $x^3=0$.  Consequently the algebra $\CA$ is spanned by the monomials
\[
\begin{array}{ccccc}
1&y&y^2&\hdots& y^{2n}\\
x&xy&xy^2&\hdots& xy^{2n}\\
x^2&x^2y&x^2y^2&\hdots& x^2y^{2n}
\end{array}
\]
In fact this is a basis (which for example can be checked using the Diamond Lemma) and so $\width(C)=\dim_{\mathbb{C}}\CA=3(2n+1)$.  Furthermore
\[
\CA^{\ab}\cong \frac{\mathbb{C}[[ x,y]]}{xy=0\mbox{, } x^2=y^{2n+1}}\cong\frac{\mathbb{C}[x,y]}{xy=0\mbox{, } x^2=y^{2n+1}}
\] 
and so $\cwidth(C)=\dim_{\mathbb{C}}\CA^{\ab}=2n+3$.
\end{example}

\subsection{Commutativity of $\CA$ and flops}
When $C$ is an irreducible rational flopping curve and $U$ is smooth, necessarily $R$ is Gorenstein.  Below in \ref{(-3,1) not commutative} we show that in this flopping setup, $\width(C)=\cwidth(C)$ if and only if $C$ has normal bundle $(-1,-1)$ or $(-2,0)$, and furthermore in this case this invariant recovers Reid's notion of width from \cite{Pagoda}.  

\begin{thm}\label{(-3,1) not commutative}
Let $U\to\Spec R$ be the flopping contraction of an irreducible rational curve $C$, where $U$ is smooth.  Then $\End_U(\cEU)$ is commutative if and only if $C$ is a $(-1,-1)$ or $(-2,0)$-curve.
\end{thm}
\begin{proof}
Since $\End_U(\cEU)\cong \CA$ by \ref{contraction alg isos}, we assume throughout that $R$ is complete local. We check commutativity on $\CA$.\\
($\Leftarrow$) Suppose that $C$ is not a $(-3,1)$-curve, then it is well known that 
\[
R\cong\frac{\mathbb{C}[[u,v,x,y]]}{uv=(x-y^n)(x+y^n)}
\]
\cite{Pagoda} for some $n\geq 1$, and that $\AB$ is isomorphic to the completion of the quiver with relations in \ref{Pagoda example}.  As shown in \ref{Pagoda example}, $\CA\cong \mathbb{C}[y]/y^n$, which is evidently commutative and has dimension $n$, which is the same as Reid's width.\\
($\Rightarrow$) By contraposition. Suppose that $C$ is a $(-3,1)$-curve, then since $U$ is smooth, 
\[
\dim_{\K}\Ext_{U}^1(E,E)=\dim_{\K}\Ext_{U}^2(E,E)=2
\]
by CY duality.  Hence by the standard deformation theory argument \cite[1.1]{Laudal} 
\[
\CA\cong\frac{\mathbb{C}\langle\langle x,y\rangle\rangle}{(f_1,f_2)}
\]
where $(f_1,f_2)$ is the closure of the ideal generated by $f_1$ and $f_2$, and further when we write both $f_1$ and $f_2$ as a sum of words, each word has degree two or higher.  We know $\CA$ is finite dimensional by \ref{cont is fd for flop}, hence it follows using the argument of \cite[1.2]{Agata} that $\CA\cong\End_U(\cEU)$ cannot be commutative.
\end{proof}

The above is somewhat remarkable, since it does not use the classification of flops, and in fact it turns out later (\ref{Toda not equiv global}) that the above is the key step in showing that Toda's commutative deformation twist functor is not an equivalence.  However, below we give a second proof of ($\Rightarrow$) in \ref{(-3,1) not commutative} that does use the classification of flops, since this proof gives us extra information regarding lower bounds of the possible widths of curves.  We do not use this second proof (and thus the classification) anywhere else in this paper.

\begin{proof}
As in the previous proof, we can assume that $R$ is complete local.  Since $R$ is a compound Du Val singularity and $C$ is a $(-3,1)$-curve, a generic hyperplane section $g\in R$ gives  a Du Val surface singularity $R/gR$ of type $D_4$, $E_6$, $E_7$ or $E_8$.  In fact, taking the pullback
\[
\begin{array}{c}
\begin{tikzpicture}
\node (a1) at (0,0) {$X$};
\node (a2) at (2.5,0) {$U$};
\node (b1) at (0,-1) {$\Spec (R/gR)$};
\node (b2) at (2.5,-1) {$\Spec R$};
\draw[->] (a1) to (a2);
\draw[->] (b1) to (b2);
\draw[->] (a1) to (b1);
\draw[->] (a2) to (b2);
\end{tikzpicture}
\end{array}
\]
then $X$ is a partial resolution of the Du Val singularity where there is only one curve above the origin, and that curve must correspond to a marked curve in one of the following diagrams
\begin{eqnarray}
\begin{array}{c}
\begin{array}{ccc}
\begin{array}{c}
\begin{tikzpicture}[xscale=0.6,yscale=0.6]
 \node (0) at (0,0) [DB] {};
 \node (1) at (1,0) [DW] {};
 \node (1b) at (1,1) [DB] {};
 \node (2) at (2,0) [DB] {};
 \draw [-] (0) -- (1);
\draw [-] (1) -- (2);
\draw [-] (1) -- (1b);
 \node at (1,-0.5) {$\scriptstyle D_4$};
\end{tikzpicture}
\end{array}
&
\begin{array}{c}
\begin{tikzpicture}[xscale=0.6,yscale=0.6]
\node (0) at (0,0) [DB] {};
\node (1) at (1,0) [DB] {};
\node (2) at (2,0) [DW] {};
\node (2b) at (2,1) [DB] {};
\node (3) at (3,0) [DB] {};
\node (4) at (4,0) [DB] {};
\draw  (0) -- (1);
\draw  (1) -- (2);
\draw  (2) -- (3);
\draw  (3) -- (4);
\draw  (2) -- (2b);
 \node at (2,-0.5) {$\scriptstyle E_6$};
\end{tikzpicture}
\end{array}
&
\begin{array}{c}
\begin{tikzpicture}[xscale=0.6,yscale=0.6]
\node (0) at (0,0) [DB] {};
\node (1) at (1,0) [DB] {};
\node (2) at (2,0) [DW] {};
\node (2b) at (2,1) [DB] {};
\node (3) at (3,0) [DB] {};
\node (4) at (4,0) [DB] {};
\node (5) at (5,0) [DB] {};
\draw  (0) -- (1);
\draw  (1) -- (2);
\draw  (2) -- (3);
\draw  (3) -- (4);
\draw  (4) -- (5);
\draw  (2) -- (2b);
 \node at (2,-0.5) {$\scriptstyle E_7$};
\end{tikzpicture}
\end{array}
\end{array}\\
\\
\begin{array}{cc}
\begin{array}{c}
\begin{tikzpicture}[xscale=0.6,yscale=0.6]
\node (0) at (0,0) [DB] {};
\node (1) at (1,0) [DB] {};
\node (2) at (2,0) [DB] {};
\node (2b) at (2,1) [DB] {};
\node (3) at (3,0) [DW] {};
\node (4) at (4,0) [DB] {};
\node (5) at (5,0) [DB] {};
\node (6) at (6,0) [DB] {};
\draw  (0) -- (1);
\draw  (1) -- (2);
\draw  (2) -- (3);
\draw  (3) -- (4);
\draw  (5) -- (6);
\draw  (4) -- (5);
\draw  (2) -- (2b);
 \node at (3,-0.5) {$\scriptstyle E_8(5)$};
\end{tikzpicture}
\end{array}
&
\begin{array}{c}
\begin{tikzpicture}[xscale=0.6,yscale=0.6]
\node (0) at (0,0) [DB] {};
\node (1) at (1,0) [DB] {};
\node (2) at (2,0) [DW] {};
\node (2b) at (2,1) [DB] {};
\node (3) at (3,0) [DB] {};
\node (4) at (4,0) [DB] {};
\node (5) at (5,0) [DB] {};
\node (6) at (6,0) [DB] {};
\draw  (0) -- (1);
\draw  (1) -- (2);
\draw  (2) -- (3);
\draw  (3) -- (4);
\draw  (4) -- (5);
\draw  (5) -- (6);
\draw  (2) -- (2b);
 \node at (2,-0.5) {$\scriptstyle E_8(6)$};
\end{tikzpicture}
\end{array}
\end{array}
\end{array}\label{Dynkin marked}
\end{eqnarray}
\cite{KM,Kawa}.  Now set $Y:=R\oplus N_1$ as in \S\ref{complete local geometry summary}, so $\AB=\End_R(Y)$, and let $e$ be the idempotent corresponding to $R$, so $\CA=\AB/\AB e\AB$.  We denote the category of maximal Cohen--Macaulay $R$-modules by $\CM R$.  Since flopping contractions are crepant, $\AB\in\CM R$ \cite[3.2.10]{VdB1d}, so since $R$ is an isolated singularity, by the depth lemma $\Ext^1_R(Y,Y)=0$.  Also, since $Y$ is a generator, necessarily $Y\in \CM R$.  

Thus applying $\Hom_R(Y,-)$ to
\[
0\to Y\xrightarrow{g}Y\to Y/gY\to 0
\]
gives
\[
0\to \Hom_R(Y,Y)\xrightarrow{g\cdot}\Hom_R(Y,Y)\to\Hom_R(Y,Y/gY)\to 0,  
\]
which implies that 
\[
\AB/g\AB\cong \Hom_R(Y,Y/gY)\cong  \End_{R/gR}(Y/gY)
\]
where the last isomorphism arises from the extension--restriction of scalars adjunction for the ring homomorphism $R\to R/gR$. Since $ \End_{R/gR}(Y/gY)\in \CM R/gR$ with $\rank_R N_1=\rank_{R/gR}(N_1/gN_1)$, in the McKay correspondence $N_1/gN_1$ is the CM module corresponding to one of the marked curves in (\ref{Dynkin marked}).

Now let $e^\prime$ be the idempotent in $\Delta:=\End_{R/gR}(Y/gY)$ corresponding to $R/gR$, so $\Delta_{\con}=\Delta/\Delta e^\prime\Delta$.  We claim that $\Delta_{\con}$ is a factor of $\CA$.  This follows since setting $I=\AB e\AB$ we have
\begin{eqnarray}
\frac{\CA}{g\CA}\cong \frac{\AB/I}{(g\AB+I)/I}\cong \frac{\AB}{g\AB+I}\cong \frac{\AB/g\AB}{(g\AB+I)/g\AB}\label{factor sequence}
\end{eqnarray}
and it is easy to show that as a $\Delta$-ideal, $(g\AB+I)/g\AB$ is equal to $\Delta e^\prime \Delta$.  Hence (\ref{factor sequence}) is isomorphic to $\Delta_{\con}$, and so indeed   $\Delta_{\con}$ is a factor of $\CA$.  

It follows that if we can show $\Delta_{\con}$ is not commutative, then $\CA$ is not commutative.  But $\Delta_{\con}\cong\uHom_{R/gR}(N_1/gN_1,N_1/gN_1)$, and it is well known that we can calculate this using the AR quiver of $\uCM{R/gR}$ (see e.g.\ \cite[Thm 4.5, Example 4.6]{IW}).  Below in \ref{reduce to surfaces} we show by a case-by-case analysis of the five possible $N_1$ appearing in (\ref{Dynkin marked}) that each $\Delta_{\con}$ is not commutative, and so each $\CA$ is not commutative.
\end{proof}

\begin{lemma}\label{reduce to surfaces}
With notation as in the proof of \ref{(-3,1) not commutative}, the $\Delta_{\con}$ corresponding to partial resolutions of the Du Val surfaces involving only the marked curves in (\ref{Dynkin marked}) are never commutative.
\end{lemma}
\begin{proof}
We give details for the $E_7$ flop, the remaining cases being very similar.   
We wish to calculate $\uHom_{R/gR}(N_1/gN_1,N_1/gN_1)$, which is graded by path length in the AR quiver of $\uCM R/gR$.   To calculate the dimension of each graded piece, we begin by placing a $1$ (corresponding to the identity) in the place of $N_1/gN_1$, and proceed by knitting:
\[
\begin{array}{c}
\begin{tikzpicture}[xscale=0.85,yscale=0.85]
\node (A0) at (0,0) {0};
\node (A1) at (1,0) {1};
\node (A2) at (2,0) {1};
\node (A3) at (3,0) {1};
\node (A4) at (4,0) {2};
\node (A5) at (5,0) {1};
\node (A6) at (6,0) {1};
\node (A7) at (7,0) {1};
\node (A8) at (8,0) {0};
\node (B0) at (0.5,-0.5) {1};
\node (B1) at (1.5,-0.5) {2};
\node (B2) at (2.5,-0.5) {2};
\node (B3) at (3.5,-0.5) {3};
\node (B4) at (4.5,-0.5) {3};
\node (B5) at (5.5,-0.5) {2};
\node (B6) at (6.5,-0.5) {2};
\node (B7) at (7.5,-0.5) {1};
\node (B8) at (8.5,-0.5) {0};
\node (C0) at (0,-1) {1};
\node (C1) at (0.5,-1) {1};
\node (C2) at (1,-1) {2};
\node (C3) at (1.5,-1) {1};
\node (C4) at (2,-1) {3};
\node (C5) at (2.5,-1) {2};
\node (C6) at (3,-1) {4};
\node (C7) at (3.5,-1) {2};
\node (C8) at (4,-1) {4};
\node (C0) at (4.5,-1) {2};
\node (C1) at (5,-1) {4};
\node (C2) at (5.5,-1) {2};
\node (C3) at (6,-1) {3};
\node (C4) at (6.5,-1) {1};
\node (C5) at (7,-1) {2};
\node (C6) at (7.5,-1) {1};
\node (C7) at (8,-1) {1};
\node (C8) at (8.5,-1) {0};
\node (D0) at (0.5,-1.5) {1};
\node (D1) at (1.5,-1.5) {2};
\node (D2) at (2.5,-1.5) {3};
\node (D3) at (3.5,-1.5) {3};
\node (D4) at (4.5,-1.5) {3};
\node (D5) at (5.5,-1.5) {3};
\node (D6) at (6.5,-1.5) {2};
\node (D7) at (7.5,-1.5) {1};
\node (D8) at (8.5,-1.5) {0};
\node (E0) at (0,-2) {0};
\node (E1) at (1,-2) {1};
\node (E2) at (2,-2) {2};
\node (E3) at (3,-2) {2};
\node (E4) at (4,-2) {2};
\node (E5) at (5,-2) {2};
\node (E6) at (6,-2) {2};
\node (E7) at (7,-2) {1};
\node (E8) at (8,-2) {0};
\node (F0) at (0.5,-2.5) {0};
\node (F1) at (1.5,-2.5) {1};
\node (F2) at (2.5,-2.5) {1};
\node (F3) at (3.5,-2.5) {1};
\node (F4) at (4.5,-2.5) {1};
\node (F5) at (5.5,-2.5) {1};
\node (F6) at (6.5,-2.5) {1};
\node (F7) at (7.5,-2.5) {0};
\node (F8) at (8.5,-2.5) {0};
\draw (0,-1) circle (7.5pt);
\draw (1,-1) circle (7.5pt);
\draw (2,-1) circle (7.5pt);
\draw (3,-1) circle (7.5pt);
\draw (4,-1) circle (7.5pt);
\draw (5,-1) circle (7.5pt);
\draw (6,-1) circle (7.5pt);
\draw (7,-1) circle (7.5pt);
\draw (8,-1) circle (7.5pt);
\end{tikzpicture}
\end{array}
\]
For details, see for example \cite[\S4]{IW}.  The above shows that the degree zero morphism set is one-dimensional (spanned by the identity), the degree two morphism set is two-dimensional, the degree four morphism set is three-dimensional, etc.  Summing up, we see that $\dim_{\mathbb{C}}\Delta_{\con}=1+2+3+4+4+4+3+2+1=24$.  Further, if we set
\[
\begin{array}{c}
\begin{tikzpicture}[xscale=0.85,yscale=0.85]
\node (A0) at (0,0) {0};
\node (A1) at (2,0) {1};
\node (B0) at (1,-0.5) {1};
\node (B1) at (3,-0.5) {2};
\node (C0) at (0,-1) {1};
\node (C1) at (1,-1) {1};
\node (C2) at (2,-1) {2};
\node (C3) at (3,-1) {1};
\node (D0) at (1,-1.5) {1};
\node (D1) at (3,-1.5) {2};
\node (E0) at (0,-2) {0};
\node (E1) at (2,-2) {1};
\node (F0) at (1,-2.5) {0};
\node (F1) at (3,-2.5) {1};
\draw (0,-1) circle (7.5pt);
\draw (2,-1) circle (7.5pt);
\draw[->] (C0) -- node[below, pos=0.7] {$\scriptstyle y$} (C2);
\draw[->] (C0) -- (B0) -- node[above, pos=0.5] {$\scriptstyle x$} (C2);
\end{tikzpicture}
\end{array}
\]
then both $x$ and $y$ have degree two, and by the mesh relations $y^2=0$.  Since the degree four morphism set is three-dimensional, necessarily $xy\neq yx$.
\end{proof}

\begin{remark}\label{rem lower width bounds}
Although we do not know explicitly all the contraction algebras   $\End_X(\cE)\cong\CA$ for all flopping contractions (except for some of type $A$ and $D_4$),  a more detailed analysis of the proof of \ref{reduce to surfaces} gives a precise value for $\dim_{\mathbb{C}}(\Delta_{\con})$, and thus gives a lower bound for the possible $\width(C)$:
\[
\begin{tabular}{*2c}
\toprule
{\bf Dynkin type}&$\width(C)$\\
\midrule
$A$ & $\geq1$\\
$D_4$ & $\geq4$\\
$E_6$ & $\geq12$\\
$E_7$ & $\geq24$\\
$E_8 (5)$ & $\geq40$\\
$E_8 (6)$ &$\geq 60$\\
\bottomrule
\end{tabular}
\]
\end{remark}

\subsection{On Toda's commutative deformation functor}

In this subsection we show that Toda's functor defined in \cite{Toda} using commutative deformations is never an equivalence for floppable $(-3,1)$-curves.  In fact, we work much more generally, without any assumptions on the singularities of $X$, then specialize down to establish the result.  We first establish the results locally on $U$, then lift to $X$.

\begin{thm}\label{not S spherical}
Let $U\to\Spec R$ be a contraction of an irreducible rational curve $C$.  Then
\begin{enumerate}
\item\label{not S spherical 1} If $\End_U(\cEU)$ is not commutative, then $\Ext^1_U(\cFU,E)\neq 0$.
\item\label{not S spherical 2} If $U$ is smooth and $C$ is a flopping $(-3,1)$-curve, then Toda's functor does not give an autoequivalence of $\Db(\coh U)$.
\end{enumerate}
\end{thm}
\begin{proof}
(1)  Since the Ext group is supported only on $\m$,
\[
\Ext^1_U(\cFU,E)\cong  \Ext_{\AB}^1(\CAab,S).
\]
We can present $\AB=\End_{\widehat{R}}(\widehat{R}\oplus N_1)$ as the (completion of) a quiver with relations, where there are two vertices $0$ and $1$, corresponding to projectives as follows:
\[
P_0:=\Hom_{\widehat{R}}(\widehat{R}\oplus N_1,\widehat{R}),
\qquad
P_1:=\Hom_{\widehat{R}}(\widehat{R}\oplus N_1,N_1).
\]
We remark that $P_1$ is the projective cover of $S$, and that $S$ is the vertex simple at $1$. 

We know that $\End_U(\cEU)\cong\CA$ by \ref{contraction alg isos}.  Since $\CA$ is not commutative, $\CA\neq \CAab$. Hence in the minimal projective resolution of $\CAab$ the kernel $K$ of the natural map 
$P_1\to \CAab$
contains some non-zero element $x$ which is a composition of cycles at vertex $1$, such that $x$ is zero in $\CAab$ but not in $\CA$.  Since this element $x$ does not factor through other vertices, there cannot be a surjective map $P_0^{a}\to K$, as we need some $P_1\to K$ in order to surject onto the element $x$. Consequently the projective cover of $K$ must contain $P_1$, and hence $\Ext^1_{\AB}(\CAab, S)\neq 0$.  \\
(2) By \ref{(-3,1) not commutative}, $\End_U(\cEU)$ is not commutative, so this follows by \eqref{not S spherical 1} \cite[Theorem~3.1]{AL1}.
\end{proof}

\begin{cor}\label{Toda not equiv global}
Suppose that $X\to X_{\con}$ is a contraction of an irreducible rational curve $C$.
\begin{enumerate}
\item\label{Toda not equiv global 1} If $\End_X(\cE)$ is not commutative, then $\Ext_X^1(\cF,i_*E) \neq 0$.
\item\label{Toda not equiv global 2} If $X$ is smooth and $C$ is a flopping $(-3,1)$-curve, Toda's functor does not give an autoequivalence of $\Db(\coh X)$.
\end{enumerate}
\end{cor}
\begin{proof}
(1) Since $\End_X(\cE)\cong\End_U(\cEU)$ is not commutative, we know that $\Ext^1_{U}(\cFU,E)\neq 0$ by \ref{not S spherical}.  The result then follows since
\[
0\neq \Ext_U^1(\cFU,E)
\stackrel{\t{\scriptsize{\eqref{inclusion line}}}}{\cong}
\Ext_X^1(\cF,i_*E).
\]
(2) By \ref{(-3,1) not commutative}, $\End_X(\cE)\cong\End_U(\cEU)$ is not commutative, so this follows by \eqref{Toda not equiv global 1}.
\end{proof}

\section{Strategy for Noncommutative Twists}\label{strategy} 

In the remainder of the paper we restrict to the setting of a flopping contraction $X\to X_{\con}$ of an irreducible rational  curve $C$, where $X$ is projective and has only Gorenstein terminal singularities.  Our aim is to overcome \ref{Toda not equiv global} and use the universal sheaf $\cE$ to produce a noncommutative twist functor, then prove that it is an autoequivalence which furthermore gives an intrinsic description of the  flop--flop functor of Bridgeland and Chen.

\subsection{Notation}
As before, and so as to fix notation for the remainder of the paper, our setup is the contraction of a irreducible rational floppable curve $C$ in a projective normal $3$-fold $X$ with at worst Gorenstein terminal singularities (e.g. $X$ is smooth).  We denote the contraction map $f\colon X\to X_{\con}$ and remark that necessarily $\Rf_*\cO_X=\cO_{X_{\con}}$ and $X_{\con}$ has only Gorenstein terminal singularities.  As before, there is a commutative diagram
\begin{eqnarray}
\begin{array}{c}
\begin{tikzpicture}
\node (C) at (-0.6,0) {$C^{}$}; 
\node (U) at (1,0) {$U$};
\node (X) at (3,0) {$X$};
\draw[right hook->] (C) to node[pos=0.6,above] {$\scriptstyle e$} (U);
\draw[right hook->] (U) to node[above] {$\scriptstyle i$} (X);

\node (x) at (-0.6,-1.5) {$\phantom{R}p\phantom{R}$}; 
\node (Uc) at (1,-1.5) {$U_{\con}$};
\node (Xc) at (3,-1.5) {$X_{\con}$};
\node at (0.1,-1.5) {$\in$}; 
\draw[right hook->] (Uc) to (Xc);

\node (m) at (-0.6,-2.5) {$\m$}; 
\node (R) at (1,-2.5) {$\Spec R$};
\node at (0.1,-2.5) {$\in$}; 

\draw[->] (X) --  node[right] {$\scriptstyle f$} (Xc);
\draw[->] (U) --  node[right] {$\scriptstyle f|_U$}  (Uc);
\draw[|->] (C) --  (x);
\node [rotate=-90] at (1, -2) {$\cong$};
\node [rotate=-90] at (-0.6, -2) {$=$};
\end{tikzpicture}
\end{array}\label{assumptions diagram}
\end{eqnarray}
where $C^{\redu}\cong\mathbb{P}^1$,  $e$ is a closed embedding and $i$ is an open embedding.    As at the beginning of \S\ref{flops section},  $U$ is derived equivalent to an algebra $\Lambda\cong\End_R(R\oplus N)$, and we set $\Lambda_{\con}:=\Lambda/I_{\con}$ where $I_{\con}=[R]$.  Since $X_{\con}$ has only terminal Gorenstein singularities, $R$ has only isolated hypersurface singularities.  Furthermore, since a single irreducible curve has been contracted to a point $\m$, as an $R$-module $\Lambda_{\con}$ is supported only on $\m$ and $\Lambda_{\con}\cong\Lambda_{\con}\otimes_R\widehat{R}\cong\widehat{\Lambda}_{\con}$.  To ease notation, as before we set $\mathfrak{R}:=\widehat{R}$.

\subsection{Mutation}\label{strategy mutation}
Our strategy is to first work on the formal fibre $\mathfrak{R}$, and here our main new tool is that of mutation, an algebraic operation developed in \cite[\S6]{IW4} that extended other theories of mutation to the setting of quivers with loops, 2-cycles, and no superpotential, all of which occur for general flops.

One of the key properties of mutation exploited below is that it always gives rise to derived equivalences as follows (see \S\ref{sect on the mutation functor} for more details). As input, we consider a \emph{modifying} $\mathfrak{R}$-module $Y$, chosen as in \S\ref{sect geom to alg} so that $\End_\mathfrak{R}(Y)$ is derived equivalent to the formal fibre. The output is another $\mathfrak{R}$-module, the \emph{mutation} $\upnu Y$, together with a tilting bimodule inducing an equivalence
\[
\Phi\colon \Db(\mod \End_\mathfrak{R}(Y)) \to \Db(\mod \End_\mathfrak{R}(\upnu Y)).
\]
It turns out that $\Phi$ is isomorphic to the inverse of the flop functor on the formal fibre \cite[4.2]{HomMMP}, but we will not need this fact. Instead, we compose with a similar functor in the opposite direction to obtain a \emph{mutation--mutation} autoequivalence of $\Db(\mod \End_\mathfrak{R}(Y))$, denoted by $\Phi \circ \Phi$.  This will be easier to study, and forms our algebraic model for the inverse of the (analytic) flop--flop functor.  Writing $\AB=\End_\mathfrak{R}(Y)$, 
\begin{enumerate}
\item \label{enum mutn mutn twist}
For the $\AB$-bimodule $I_{\AB}$ defined by the natural exact sequence
\begin{equation}\label{crucial SES}
I_{\AB} \to \AB \to \CA,
\end{equation}
we will show in \ref{mut mut functor general}\eqref{mut mut functor general 1} that
\begin{equation}\label{mutn mutn twist}
\Phi\circ\Phi = \RHom_{\AB}(I_{\AB},-).
\end{equation}
\setcounter{tempenum}{\theenumi}
\end{enumerate}

Although we do not give the details here, the functor $\RHom_{\AB}(I_{\AB},-)$ may be interpreted as an (inverse) twist around the noncommutative deformation family of the simple $\AB$-module $S$, and so \eqref{mutn mutn twist} gives a twist description of $\Phi \circ \Phi$.  In the process of proving \eqref{enum mutn mutn twist}, we further obtain:
\begin{list}{(${\theenumi}$)}{\usecounter{enumi}}
\setcounter{enumi}{\thetempenum}
\item \label{enum min proj res} A minimal projective resolution of $\CA$ by $\AB$-modules in \ref{proj res thm}\eqref{proj res thm 2}.
\item \label{enum action on CA} A description of the action of $\Phi \circ \Phi$ on $\CA$ in \ref{track Lambda J 1}\eqref{track Lambda J 1 2}.
\setcounter{tempenum}{\theenumi}
\end{list}
Point \eqref{enum action on CA} is needed to produce a global spanning class later, and point \eqref{enum min proj res} will establish that $\cE$ is a perfect complex such that
\begin{eqnarray*}
\Ext_X^t(\cE,i_*E) =\left\{ \begin{array}{cl} \mathbb{C}&\mbox{if }t=0,3\\
0&\mbox{else},\\  \end{array} \right.
\end{eqnarray*}
generalising \cite{Toda}.

\subsection{Strategy}\label{section twist strategy}
The remainder of the paper involves lifting the above formal fibre results to algebraic flops.  From the complete local algebraic model in \S\ref{strategy mutation}, a twist autoequivalence around the universal sheaf $\cE$ on $X$ is produced using the following strategy:

\begin{description}
\item[Tilting algebra $\Lambda$]  Using an algebraic idempotent trick, together with a morita equivalence, we are able to lift the tilting bimodule $I_{\AB}$ giving rise to the complete local equivalence $\Phi\circ\Phi$ to a $\Lambda$-bimodule $I_{\con}$.  We prove in \ref{I tilts not complete local} that this bimodule gives rise to a derived autoequivalence of $\Lambda$, and we thus obtain an algebraic local model $\RHom_{\Lambda}(I_{\con},-)$ for our twist functor. 

\item[Local geometry $U$] Since $U$ is derived equivalent to $\Lambda$, the above algebraic $\Lambda$-bimodule $I_{\con}$ induces a Fourier--Mukai kernel on $U\times U$. Since $I_{\con}$ is a tilting module, this kernel gives a twist autoequivalence on $U$, in \S\ref{geom noncomm twist}.

\item[Global geometry $X$] A gluing construction then produces a twist functor $T_{\cE}$ on $X$ in \S\ref{global section}. It remains to show that $T_{\cE}$ is an autoequivalence.  We are able to conclude that $\cE$ is perfect from the minimal projective resolution of $\CA$ in point~\eqref{enum min proj res} above, then, by inferring the action of $T_{\cE}$ on the universal sheaf $\cE$ from the knowledge of $\Phi \circ \Phi$ in \eqref{enum action on CA}, we are able to construct (in \S\ref{global proof subsection}) a spanning class on which the action of $T_{\cE}$ is known.  A spanning class argument establishes that $T_{\cE}$ is fully faithful, and from there a straightforward lemma, \ref{equiv trick}, allows us to conclude that $T_{\cE}$ is an autoequivalence in \S\ref{newsubsection}.
\end{description}

Furthermore, in \ref{prop twist functorial triangle}\eqref{prop twist functorial triangle 2} we show that $T_\cE$ fits into a functorial triangle
\[
\RHom_X(\cE,-)\otimes_{\CA}^{\bf L} \cE\to \Id \to T_{\cE}\to \]
as a consequence of the short exact sequence \eqref{crucial SES} above. We finally prove in \ref{twist flopflop} that $T_{\cE}$ is an inverse of the flop--flop functor $\flopflop$, giving the global geometric version of \eqref{mutn mutn twist}.

\section{Complete Local Mutation and Ext Vanishing}\label{prelim}

We keep the notation and assumptions as in \S\ref{strategy}.  There, $U\to\Spec R$ is a crepant morphism, which in fact is equivalent to the condition $\Lambda=\End_R(R\oplus N)\in \CM R$ \cite{IW5}, where recall $\CM R$ denotes the category of maximal Cohen--Macaulay $R$-modules.   Since $N\cong\Hom_R(R,N)$ is a summand of $\Lambda$, this in turn forces $N\in\CM R$. Writing $\refl R$ for the category of reflexive $R$-modules, modules $Y\in\refl R$ satisfying $\End_R(Y)\in \CM R$ are thus of interest, and are called \emph{modifying} $R$-modules \cite{IW4}.

\subsection{Mutation}\label{mut subsection}

Now $\Lambda=\End_R(R\oplus N)$ and completing with respect to the maximal ideal $\m$ in \eqref{assumptions diagram}
\[
\widehat{\Lambda}\cong\End_{\widehat{R}}(\widehat{R}\oplus\widehat{N}).
\]
As in \ref{decomp into irred 1}, we decompose 
\[
\widehat{R}\oplus\widehat{N}\cong \widehat{R}^{\oplus a_0}\oplus N_1^{\oplus a_1}.
\]
Throughout the section, as in \S\ref{complete local geometry summary}, to ease notation with hats we write $\mathfrak{R}:=\widehat{R}$, and
\[
Y:=\mathfrak{R}\oplus N_1, \qquad Z:=\mathfrak{R}^{\oplus a_0}\oplus N_1^{\oplus a_1},
\]
so that $\AB=\End_\mathfrak{R}(Y)$ and $\widehat{\Lambda}=\End_\mathfrak{R}(Z)$.  We write $(-)^{*}$ for the duality functor \[\Hom_{\mathfrak{R}}(-,\mathfrak{R})\colon\refl \mathfrak{R}\to \refl \mathfrak{R}.\]

\begin{defin}\label{setup2} With the setup as above,
\begin{enumerate}
\item We take a minimal right $(\add \mathfrak{R})$-approximation  
\[
\mathfrak{R}_0\xrightarrow{a}N_1\label{approx 1}
\]
of $N_1$, which by definition means that
\begin{enumerate}
\item $\mathfrak{R}_0\in\add \mathfrak{R}$ and $(a\cdot)\colon\Hom_\mathfrak{R}(\mathfrak{R},\mathfrak{R}_0)\to\Hom_\mathfrak{R}(\mathfrak{R},N_1)$ is surjective,
\item if $g\in\End_\mathfrak{R}(\mathfrak{R}_0)$ satisfies $a=ag$, then $g$ is an automorphism.
\end{enumerate}
In other words, $a$ is a projective cover of $N_1$.  Necessarily $a$ is surjective.  Since $\mathfrak{R}$ is complete, such an $a$ exists and is unique up to isomorphism.  We put $K_0:=\Ker a$, so we have an exact sequence
\begin{eqnarray}
0\to K_0\xrightarrow{c} \mathfrak{R}_0\xrightarrow{a}N_1\to 0.\label{K0}
\end{eqnarray}
Trivially
\[
0\to \Hom_\mathfrak{R}(\mathfrak{R},K_0)\xrightarrow{c\cdot} \Hom_\mathfrak{R}(\mathfrak{R},\mathfrak{R}_0)\xrightarrow{a\cdot}\Hom_\mathfrak{R}(\mathfrak{R},N_1)\to 0\label{K0a}
\]
is exact, thus since by definition $\CA=\End_\mathfrak{R}(Y)/[\mathfrak{R}]=\End_\mathfrak{R}(N_1)/[\mathfrak{R}]$, applying $\Hom_\mathfrak{R}(Y,-)$ to \eqref{K0} yields an exact sequence
\begin{eqnarray}
0\to \Hom_\mathfrak{R}(Y,K_0)\xrightarrow{c\cdot} \Hom_\mathfrak{R}(Y,\mathfrak{R}_0)\xrightarrow{a\cdot}\Hom_\mathfrak{R}(Y,N_1)\to\CA\to 0\label{begin lambdacon}
\end{eqnarray}
of $\AB$-modules.

\item We define the \emph{right mutation} of $Y$ as
\[
\upmu Y:=\mathfrak{R}\oplus K_0,
\]
that is we remove the summand $N_1$ and replace it with $K_0$.
\item Dually, we consider a minimal right $(\add \mathfrak{R}^*)$-approximation
\[
\mathfrak{R}_1^*\xrightarrow{b}N_1^*
\]
of $N_1^*$, and we put $K_1:=\Ker b$.  Thus again we have an exact sequence 
\begin{eqnarray}
0\to K_1\xrightarrow{d} \mathfrak{R}_1^*\xrightarrow{b}N_1^*\to 0.\label{K1}
\end{eqnarray}
\item We define the \emph{left mutation} of $Y$ as
\[
\upnu Y:=\mathfrak{R}\oplus K_1^*.
\]
\end{enumerate}
\end{defin}
In fact $\upnu Y=(\upmu Y^*)^*$, so really we only need to define right mutation.

\begin{remark}\label{can calculate width}
Applying $\Hom_\mathfrak{R}(Y,-)$ to \eqref{K0} and observing the resulting long exact sequence in Ext gives
\[
0\to \Hom_\mathfrak{R}(Y,K_0)\xrightarrow{c\cdot} \Hom_\mathfrak{R}(Y,\mathfrak{R}_0)\xrightarrow{a\cdot}\Hom_\mathfrak{R}(Y,N_1)\to\Ext^1_{\mathfrak{R}}(Y,K_0)\to 0.
\]
Combining this with \eqref{begin lambdacon} shows that, as $\mathfrak{R}$-modules
\[
\CA\cong \Ext^1_{\mathfrak{R}}(Y,K_0)\cong \Ext^1_{\mathfrak{R}}(N_1,\Omega N_1)\cong D\Ext_\mathfrak{R}^1(\Omega N_1,N_1)\cong D\Ext^2_\mathfrak{R}(N_1,N_1)
\]
where the third isomorphism is just AR duality, since $\mathfrak{R}$ is an isolated singularity. Taking dimensions of both sides shows that
\[
\width(C)=\dim_{\mathbb{C}}\Ext^2_{\widehat{R}}(N_1,N_1).
\]
This may be calculated by using computer algebra on the base $\widehat{R}$, and gives a very easy way to calculate $\width(C)$.  In contrast, computing the algebra structure on $\CA$ is always a little harder, and this is needed to calculate $\CAab$ and hence $\cwidth(C)$.  We remark that there does not seem to be any easy description of $\cwidth(C)$ in terms of Ext groups.  
\end{remark}

\begin{remark}
Mutation is defined in \cite[\S6]{IW4} for any modifying module, so in particular also for $Z=\mathfrak{R}^{\oplus a_0}\oplus N_1^{\oplus a_1}$. Summing the  exact sequence \eqref{K0} gives an exact sequence
\begin{eqnarray}
0\to K_0^{\oplus a_1}\xrightarrow{c^{\oplus a_1}} \mathfrak{R}_0^{\oplus a_1}\xrightarrow{a^{\oplus a_1}}N_1^{\oplus a_1}\to 0\label{K0NL}
\end{eqnarray}
where $a^{\oplus a_1}$ is a minimal $(\add \mathfrak{R})$-approximation of $N_1^{\oplus a_1}$.   Thus 
\[
\upmu Z=\mathfrak{R}_0^{\oplus a_0}\oplus K_0^{\oplus a_1}.
\]
Similarly, summing \eqref{K1} gives a minimal $(\add \mathfrak{R}^*)$-approximation of $(N_1^{\oplus a_1})^*$ and so
\[
\upnu Z=\mathfrak{R}_0^{\oplus a_0}\oplus (K_1^*)^{\oplus a_1}.
\]
From this, it is clear that $\End_\mathfrak{R}(\upnu Y)$ is morita equivalent to $\End_\mathfrak{R}(\upnu Z)$.
\end{remark}

Recall that $\mathbb{F}=\Hom_{\AB}(P,-)\colon\mod\AB\to\mod\widehat{\Lambda}$ is the morita equivalence from \S\ref{complete local geometry summary}, where $P=\Hom_\mathfrak{R}(Y,Z)$.

\begin{lemma}\label{morita lemma}
$\widehat{\Lambda}_{\con}\cong(\FCA)^{\oplus a_1}$ as $\widehat{\Lambda}$-modules. In particular 
\[
\Lambda_{\con}\cong\widehat{\Lambda}_{\con}\cong M_{a_1}(\CA),
\]
where $M_{a_1}(\CA)$ denotes the ring of $a_1\times a_1$ matrices over $\CA$.  Consequently $\Lambda_{\con}\cong\widehat{\Lambda}_{\con}$ and $\CA$ are morita equivalent. 
\end{lemma}
\begin{proof}
Applying $\Hom_\mathfrak{R}(Z,-)$ to \eqref{K0NL} yields the exact sequence
\[
\Hom_\mathfrak{R}(Z,\mathfrak{R}_0^{\oplus a_1})\xrightarrow{(a\cdot)^{\oplus a_1}}\Hom_\mathfrak{R}(Z,N_1^{\oplus a_1})\to \widehat{\Lambda}_{\con}\to 0
\]
of $\widehat{\Lambda}$-modules. On the other hand applying $\mathbb{F}$ to \eqref{begin lambdacon} shows that the top sequence in the following commutative diagram is exact
\[
\opt{10pt}{\def\seqspace{6}}
\opt{12pt}{\def\seqspace{7}\hspace*{-5mm}}
\begin{tikzpicture}[scale=1,node distance=1]
\node (A1) at (0,0) {$\Hom_{\AB}(\Hom_\mathfrak{R}(Y,Z),\Hom_\mathfrak{R}(Y,\mathfrak{R}_0))$};
\node (A2) at (\seqspace,0) {$\Hom_{\AB}(\Hom_\mathfrak{R}(Y,Z),\Hom_\mathfrak{R}(Y,N_1))$};
\opt{12pt}{\node (A3') at (4+\seqspace,0) {$\ldots$};}
\opt{10pt}{\node (A3) at (9.8,0) {$\mathbb{F}(\CA)$};
\node (A4) at (11,0) {$0$};}
\opt{12pt}{\node (A2') at (4,-2.5) {$\ldots$};
\node (A3) at (6,-2.5) {$\mathbb{F}(\CA)$};
\node (A4) at (8,-2.5) {$0$};
\draw[->] (A2) -- (A3');
\draw[->] (A2') -- (A3);}
\node (B1) at (0,-1.5) {$\Hom_{\mathfrak{R}}(Z,\mathfrak{R}_0)$};
\node (B2) at (\seqspace,-1.5) {$\Hom_{\mathfrak{R}}(Z,N_1)$};
\draw[->] (A1) -- (A2);
\opt{10pt}{\draw[->] (A2) -- (A3);}
\draw[->] (A3) -- (A4);
\draw[->] (B1) -- node[above] {$\scriptstyle a\cdot$} (B2);
\draw[->] (B1) -- node[sloped,left,anchor=south] {$\scriptstyle \sim$} (A1);
\draw[->] (B2) -- node[sloped,left,anchor=south] {$\scriptstyle \sim$} (A2);
\end{tikzpicture}
\]
where the vertical isomorphisms are just reflexive equivalence (see e.g. \cite[2.5]{IW4}).  It follows that $\FCA\cong\Cok(a\cdot)$ and thus
\[
\FCA^{\oplus a_1}\cong\Cok(a^{\oplus a_1}\cdot)\cong\widehat{\Lambda}_{\con}.
\]
The last statement follows by \ref{contraction splits} since
\opt{10pt}{\[
\widehat{\Lambda}_{\con}\cong \End_{\widehat{\Lambda}}(\widehat{\Lambda}_{\con})\cong \End_{\widehat{\Lambda}}(\FCA^{\oplus a_1})\cong \End_{\AB}(\CA^{\oplus a_1})\cong  M_{a_1}(\End_{\AB}(\CA))\cong M_{a_1}(\CA).
\]}
\opt{12pt}{\begin{align*}
\widehat{\Lambda}_{\con}\cong \End_{\widehat{\Lambda}}(\widehat{\Lambda}_{\con})\cong \End_{\widehat{\Lambda}}(\FCA^{\oplus a_1})&\cong \End_{\AB}(\CA^{\oplus a_1})\\
&\cong  M_{a_1}(\End_{\AB}(\CA))\\
&\cong M_{a_1}(\CA).\qedhere
\end{align*}}
\end{proof}

The following is important, and is a consequence of the fact that $\mathfrak{R}$ is a hypersurface singularity.

\begin{prop}\label{mut twice is identity} With assumptions and notation as above,
\begin{enumerate}
\item\label{mut twice is identity 1} $\upnu(\upnu Y)\cong Y$.
\item\label{mut twice is identity 2} $K_1^*\cong K_0$, so that $\upnu Y\cong\upmu Y$.
\end{enumerate}
\end{prop}
\begin{proof}
(1) By definition of mutation, $\upnu=\Omega^{-1}$, where $\Omega^{-1}$ denotes the cosyzygy functor on $\CM \mathfrak{R}$.  But $N_1\in\CM \mathfrak{R}$ since $A\in\CM \mathfrak{R}$, and it is well known that on hypersurfaces $\Omega^{-2}(Z)\cong Z$ for all $Z\in\CM \mathfrak{R}$ \cite{Eisenbud}.\\
(2)   To calculate $\upnu(\upnu Y)$, we take a minimal right ($\add(\frac{\upnu Y}{K_1^*})^*$)-approximation, equivalently a minimal right ($\add \mathfrak{R}^*$)-approximation, of  $(K_1^*)^*\cong K_1$ and by doing so obtain an exact sequence
\[
0\to K_3\to \mathfrak{R}_2^*\to K_1\to 0\label{second approx}
\]  
with $\mathfrak{R}^*_2\in \add \mathfrak{R}^*$.  Since we have $\upnu(\upnu Y):=\mathfrak{R}\oplus K_3^*$,  by \eqref{mut twice is identity 1} and Krull--Schmidt it follows that $K_3^*\cong N_1$.  Now dualizing the above sequence gives  an exact sequence
\[
0\to K_1^*\to \mathfrak{R}_2\to N_1\to 0
\]
where the last map is a minimal right ($\add \mathfrak{R}$)-approximation.  By uniqueness of minimal approximations, this implies that $\mathfrak{R}_2\cong \mathfrak{R}_0$, $K_1^*\cong K_0$ and further $\upnu Y=\mathfrak{R}\oplus K_1^*\cong \upmu Y$.
\end{proof}

Since all modules $\mathfrak{R}, N_1, K_0, K_1\in\CM \mathfrak{R}$ and $(-)^*$ is an exact duality on $\CM \mathfrak{R}$, we can dualize \eqref{K1} and obtain an exact sequence
\begin{align}
&0\to N_1\stackrel{b^{*}}\to \mathfrak{R}_1\stackrel{d^{*}}\to K_{1}^{*}\to 0. \label{K1D}
\end{align}

\subsection{$\Ext$ vanishing}  With regards to our applications, the point of this subsection is to prove \ref{Ext cor}.  The following is an application of \cite[\S 6]{IW4}.

\begin{prop}\label{proj res thm} With assumptions and notation as above,
\begin{enumerate}
\item\label{proj res thm 1} Applying $\Hom_\mathfrak{R}(Y,-)$ to the sequence (\ref{K1D}) gives an exact sequence
\begin{eqnarray}
0\to\Hom_\mathfrak{R}(Y,N_1)\xrightarrow{b^*\cdot}\Hom_\mathfrak{R}(Y,\mathfrak{R}_1)\xrightarrow{d^*\cdot} \Hom_\mathfrak{R}(Y,K_1^*)\to 0.\label{tilting key sequence}
\end{eqnarray}
\item\label{proj res thm 2} The minimal projective resolution of $\CA$ as an $\AB$-module has the form
\[
0\to P\to Q_1\to Q_0\to P\to \CA\to 0
\]
where $P:=\Hom_\mathfrak{R}(Y,N_1)$, and $Q_i\in\add Q$ for $Q:=\Hom_\mathfrak{R}(Y,\mathfrak{R})$.\smallskip
\item\label{proj res thm 3} $\pd_{\widehat{\Lambda}}\widehat{\Lambda}_{\con}=3$ and $\pd_{\widehat{\Lambda}}\widehat{I}_{\con}=2$.
\end{enumerate}
\end{prop}
\begin{proof}
(1) Since $\CA$ is finite dimensional by \ref{cont is fd for flop}, this follows from the argument exactly as in \cite[(6.Q)]{IW4}.\\
(2) Since $K_1^*\cong K_0$ by \ref{mut twice is identity}, splicing \eqref{tilting key sequence} and \eqref{begin lambdacon} gives the minimal projective resolution
\opt{10pt}{\[
0\to \Hom_\mathfrak{R}(Y,N_1)\to \Hom_\mathfrak{R}(Y,\mathfrak{R}_1)\to \Hom_\mathfrak{R}(Y,\mathfrak{R}_0)\to \Hom_\mathfrak{R}(Y,N_1)\to \CA\to 0.
\]}
\opt{12pt}{\begin{align*}
0\to \Hom_\mathfrak{R}(Y,N_1)\to \Hom_\mathfrak{R}(Y,\mathfrak{R}_1)\to \Hom_\mathfrak{R}(Y,\mathfrak{R}_0)&\to \Hom_\mathfrak{R}(Y,N_1)\\&\to \widehat{\Lambda}_{\con}\to 0.
\end{align*}}
(3) We have $\pd_{\AB}\CA=3$ by \eqref{proj res thm 2}, and so chasing across the morita equivalence using \ref{morita lemma}, $\pd_{\widehat{\Lambda}}\widehat{\Lambda}_{\con}=3$.  The final statement follows.
\end{proof}

\begin{cor}\label{Ext cor}  There are isomorphisms
\[
\Ext_{\Lambda}^t(\FCA,T) 
\cong\Ext_{\widehat{\Lambda}}^t(\FCA,\widehat{T}) 
\cong\Ext_{\AB}^t(\CA,S) 
\cong\left\{ \begin{array}{cl} \mathbb{C}&\mbox{if }t=0,3\\
0&\mbox{else,}\\  \end{array} \right.
\]
and further $\CA$ is a self-injective algebra.
\end{cor}
\begin{proof}
The first two isomorphisms are consequences of the fact that the Ext groups are supported only at $\m$.  The third isomorphism is an immediate consequence of the minimal projective resolution in \ref{proj res thm}, since $\Hom_{\AB}(Q_i,S)=0$.  

Now if the outer two terms in a short exact sequence are annihilated by an idempotent, so is the middle term.  Hence $\mod\CA$ is extension-closed in $\mod\AB$, and it follows that
\begin{eqnarray}
\Ext^1_{\CA}(S,\CA)=\Ext^1_{\AB}(S,\CA)\cong D\Ext^2_{\AB}(\CA,S)\label{selfinj sequence}
\end{eqnarray}
where the last isomorphism holds since $\AB$ is 3-sCY \cite[2.22(2)]{IW4}, $\pd_{\AB}\CA<\infty$ and $S$ has finite length.   Thus \eqref{selfinj sequence} shows that $\Ext^1_{\CA}(S,\CA)=0$.  Since $\CA$ is finite dimensional with unique simple $S$, it follows that $\CA$ is self-injective.
\end{proof}

\subsection{On the mutation functor $\Phi$}\label{sect on the mutation functor}  We retain the setup and notation from \S\ref{mut subsection}, namely $\AB=\End_\mathfrak{R}(Y)$ and $\widehat{\Lambda}=\End_\mathfrak{R}(Z)$.  For the case of left mutation $\upnu Y$, there is a derived equivalence between $\End_\mathfrak{R}(Y)$ and $\End_\mathfrak{R}(\upnu Y)$ \cite[\S 6]{IW4} given by a tilting $\End_\mathfrak{R}(Y)$-module $V$ constructed as follows.  We consider the sequence \eqref{K1D}
\[
0\to N_1\xrightarrow{b^*} \mathfrak{R}_1\to K_1^*\to 0
\]
obtained by dualizing (\ref{K1}).  It is trivial that $b^*$ is a minimal left $(\add \mathfrak{R})$-approximation of $N_1$.  Applying $\Hom_\mathfrak{R}(Y,-)$, gives an induced map \opt{10pt}{$(b^*\cdot )\colon\Hom_\mathfrak{R}(Y,N_1)\to\Hom_\mathfrak{R}(Y,\mathfrak{R}_1),$}\opt{12pt}{\[(b^*\cdot )\colon\Hom_\mathfrak{R}(Y,N_1)\to\Hom_\mathfrak{R}(Y,\mathfrak{R}_1),\]} and our tilting $\AB$-module is by definition given by
\[
V:=\Hom_\mathfrak{R}(Y,\mathfrak{R})\oplus\Cok(b^*\cdot).
\]

By \eqref{tilting key sequence} we already know that $\Cok(b^*\cdot)\cong \Hom_\mathfrak{R}(Y,K_1^*)$, so in fact  
\[
V\cong \Hom_\mathfrak{R}(Y,\mathfrak{R})\oplus \Hom_\mathfrak{R}(Y,K_1^*)=\Hom_\mathfrak{R}(Y,\upnu Y).
\]
This gives rise to an equivalence \cite[6.8]{IW4}
\[
\Phi:=\RHom_{\AB}(V,-)\colon\Db(\mod\AB)\to \Db(\mod\End_{\mathfrak{R}}(\upnu Y)),
\]
which we call the {\em mutation functor}.   By \ref{mut twice is identity} we can mutate $\End_{\mathfrak{R}}(\upnu Y)$ back to obtain $\AB$, and in an identical way 
\[
W:=\Hom_\mathfrak{R}(\upnu Y,Y)
\]
is a tilting module giving rise to an equivalence which by abuse of notation we also denote
\[
\Phi:=\RHom(W,-)\colon\Db(\mod\End_{\mathfrak{R}}(\upnu Y))\to \Db(\mod\AB).
\]

Similarly 
\[
V':=\Hom_\mathfrak{R}(Z,\mathfrak{R}^{\oplus a_0})\oplus\Hom_\mathfrak{R}(Z,(K_1^*)^{\oplus a_1})=\Hom_\mathfrak{R}(Z,\upnu Z)
\]
is a tiling $\widehat{\Lambda}$-module, giving rise to an equivalence
\[
\Phi':=\RHom_{\widehat{\Lambda}}(V',-)\colon\Db(\mod\widehat{\Lambda})\to \Db(\mod\End_{\mathfrak{R}}(\upnu Z))
\]
and we mutate back via the tilting module $W':=\Hom_\mathfrak{R}(\upnu Z,Z)$.

The following is easily seen, and is an elementary application of morita theory.

\begin{lemma}\label{mut com diagram}
The following diagram commutes:
\[
\begin{array}{c}
\begin{tikzpicture}
\node (A1) at (0,0) {$\Db(\mod\AB)$};
\node (A2) at (0,-1.5) {$\Db(\mod\End_\mathfrak{R}(\upnu Y))$};
\node (A3) at (0,-3) {$\Db(\mod\AB)$};
\node (B1) at (5,0) {$\Db(\mod\widehat{\Lambda})$};
\node (B2) at (5,-1.5) {$\Db(\mod\End_\mathfrak{R}(\upnu Z))$};
\node (B3) at (5,-3) {$\Db(\mod\widehat{\Lambda})$};
\draw[->] (A1) -- node[above] {$\scriptstyle \mathbb{F}$} (B1);
\draw[->] (A2) -- node[above] {$\scriptstyle morita$} (B2);
\draw[->] (A3) -- node[above] {$\scriptstyle \mathbb{F}$} (B3);
\draw[->] (A1) -- node[left] {$\scriptstyle \Phi$} (A2);
\draw[->] (A2) -- node[left] {$\scriptstyle \Phi$} (A3);
\draw[->] (B1) -- node[right] {$\scriptstyle \Phi'$} (B2);
\draw[->] (B2) -- node[right] {$\scriptstyle \Phi'$} (B3);
\end{tikzpicture}
\end{array}
\]
\end{lemma}

For our purposes later, we need to be able to track the object $\CA$ through the derived equivalence $\Phi$. This is taken care of in the following lemma.

\begin{lemma}\label{track Lambda J 1}
Write $\BB:=\End_{\mathfrak{R}}(\upnu Y)$ and $\CB:=\BB/[\mathfrak{R}]$, then
\begin{enumerate}
\item\label{track Lambda J 1 1} $\Phi(\CA)\cong \CB[-1]$.
\item\label{track Lambda J 1 2} $\Phi\circ\Phi(\CA)\cong\CA[-2]$.
\item\label{track Lambda J 1 3} $\Phi'\circ\Phi'(\FCA)\cong \FCA[-2]$.
\end{enumerate}
\end{lemma} 
\begin{proof}
Note first that in fact $\CA\cong\CB$ as algebras \cite[6.20]{IW4}, so $\dim_{\mathbb{C}}\CB<\infty$.  \\
(1) We already know that $V=\Hom_\mathfrak{R}(Y,\upnu Y)$.  Now since $K_1^*\cong K_0$ by \ref{mut twice is identity}, splicing \eqref{K0} and \eqref{K1D} gives us an exact sequence of $R$-modules
\[
0\to K_1^*\stackrel{c}{\to} \mathfrak{R}_0\xrightarrow{b^{*}\cdot a} \mathfrak{R}_1\stackrel{d^{*}}\to K_{1}^{*} \to 0 
\]
to which applying $\Hom_\mathfrak{R}(Y,-)$ gives a complex of $\AB$-modules
\opt{10pt}{\begin{eqnarray}
0\to \Hom_\mathfrak{R}(Y,K_1^*)\xrightarrow{c\cdot} \Hom_\mathfrak{R}(Y,\mathfrak{R}_0)\xrightarrow{(b^{*}\cdot a)\cdot} \Hom_\mathfrak{R}(Y,\mathfrak{R}_1)\xrightarrow{d^{*}\cdot} \Hom_\mathfrak{R}(Y,K_{1}^{*})\to 0 \label{base Lam sequence}
\end{eqnarray}}
\opt{12pt}{\begin{equation}\begin{split}
0\to \Hom_\mathfrak{R}(Y,K_1^*)\xrightarrow{c\cdot} \Hom_\mathfrak{R}(Y,\mathfrak{R}_0)\xrightarrow{(b^{*}\cdot a)\cdot} \Hom_\mathfrak{R}(Y,\mathfrak{R}_1)\xrightarrow{d^{*}\cdot} \\\xrightarrow{d^{*}\cdot}\Hom_\mathfrak{R}(Y,K_{1}^{*})\to 0 \end{split}\label{base Lam sequence}\end{equation}}
whereas applying $\Hom_\mathfrak{R}(\upnu Y,-)$ gives a complex of $\BB$-modules
\opt{10pt}{\begin{eqnarray}
0\to \Hom_\mathfrak{R}(\upnu Y,K_1^*)\xrightarrow{c\cdot} \Hom_\mathfrak{R}(\upnu Y,\mathfrak{R}_0)\xrightarrow{(b^{*}\cdot a)\cdot} \Hom_\mathfrak{R}(\upnu Y,\mathfrak{R}_1)\xrightarrow{d^{*}\cdot} \Hom_\mathfrak{R}(\upnu Y,K_{1}^{*})\to 0 \label{base Del sequence}
\end{eqnarray}}
\opt{12pt}{\begin{equation}\begin{split}
0\to \Hom_\mathfrak{R}(\upnu Y,K_1^*)\xrightarrow{c\cdot} \Hom_\mathfrak{R}(\upnu Y,\mathfrak{R}_0)\xrightarrow{(b^{*}\cdot a)\cdot} \Hom_\mathfrak{R}(\upnu Y,\mathfrak{R}_1)\xrightarrow{d^{*}\cdot}\\\xrightarrow{d^{*}\cdot} \Hom_\mathfrak{R}(\upnu Y,K_{1}^{*})\to 0 \end{split}\label{base Del sequence}
\end{equation}}
In both \eqref{base Lam sequence} and \eqref{base Del sequence}, we consider the last Hom term to be in degree zero.  Since every term in \eqref{base Lam sequence} is a summand of $V$, by reflexive equivalence it follows that under the derived equivalence $\Phi$, \eqref{base Lam sequence} gets sent to \eqref{base Del sequence}.  Hence the statement in the lemma follows if we can show that \eqref{base Lam sequence} is quasi-isomorphic to $\CA[1]$ and that \eqref{base Del sequence} is quasi-isomorphic to $\CB[0]$.

Consider first \eqref{base Lam sequence}.  Combining the exact sequences \eqref{tilting key sequence} and \eqref{begin lambdacon}
\begin{gather*}
0\to\Hom_\mathfrak{R}(Y,N_1)\xrightarrow{b^*\cdot}\Hom_\mathfrak{R}(Y,\mathfrak{R}_1)\xrightarrow{d^*\cdot} \Hom_\mathfrak{R}(Y,K_1^*)\to 0\\
0\to \Hom_\mathfrak{R}(Y, K_1^*)\xrightarrow{c\cdot} \Hom_\mathfrak{R}(Y,\mathfrak{R}_0)\xrightarrow{a\cdot} \Hom_\mathfrak{R}(Y, N_1)\to\CA\to 0
\end{gather*}
shows that \eqref{base Lam sequence} is quasi-isomorphic to $\CA[1]$.

Next, consider \eqref{base Del sequence}.  Applying \ref{proj res thm}\eqref{proj res thm 1} with the module $\upnu Y$ (instead of $Y$), it follows that
\[
0\to \Hom_\mathfrak{R}(\upnu Y,K_1^*)\xrightarrow{c\cdot} \Hom_\mathfrak{R}(\upnu Y,\mathfrak{R}_0)\xrightarrow{a\cdot}\Hom_\mathfrak{R}(\upnu Y,N_1)\to 0
\] 
is exact.  Further, since trivially $d^*$ is a minimal ($\add \mathfrak{R}$)-approximation, applying the functor $\Hom_\mathfrak{R}(\upnu Y,-)$ to \eqref{K1D} gives an exact sequence
\[
0\to \Hom_\mathfrak{R}(\upnu Y,N_1)\xrightarrow{b^*\cdot} \Hom_\mathfrak{R}(\upnu Y,\mathfrak{R}_1)\xrightarrow{d^*\cdot}\Hom_\mathfrak{R}(\upnu Y,K_1^*)\to \CB\to 0.
\]
Splicing the last two exact sequences shows that \eqref{base Del sequence} is quasi-isomorphic to $\CB[0]$.\\
(2) Since as above $\dim_{\mathbb{C}}\CB<\infty$, this follows since we can apply \eqref{track Lambda J 1 1} twice.\\
(3) This follows from \eqref{track Lambda J 1 2} using the commutativity of the diagram in \ref{mut com diagram}.
\end{proof}

We write $I_{\AB}$ for the ideal $[\mathfrak{R}]$ of $\AB=\End_\mathfrak{R}(\mathfrak{R}\oplus N_1)$.  The following gives a description of the mutation--mutation functors in terms of the bimodules $I_{\AB}$ and $\widehat{I}_{\con}$.

\begin{thm}\label{mut mut functor general}  There are functorial isomorphisms
\begin{enumerate}
\item\label{mut mut functor general 1} $\Phi\circ\Phi\cong\RHom_{\AB}(I_{\AB},-)$.
\item\label{mut mut functor general 2} $\Phi'\circ\Phi'\cong\RHom_{\widehat{\Lambda}}(\widehat{I}_{\con},-)$.
\end{enumerate}
Thus $\RHom_{\widehat{\Lambda}}(\widehat{I}_{\con},-)$ is an autoequivalence of $\Db(\mod\widehat{\Lambda})$, with inverse $-\otimes^{\bf L}_{\widehat{\Lambda}} \widehat{I}_{\con}$, and so  $\widehat{I}_{\con}$ is a tilting $\widehat{\Lambda}$-module.
\end{thm}
\begin{proof}
(1) To avoid confusion we will write $\BB=\End_\mathfrak{R}(\upnu Y)$, and also denote
\[
\Db(\mod\AB)\xrightarrow{\Phi}\Db(\mod\BB) \mbox{ by } \Phi^1=\RHom_{\AB}(V,-),
\]
and 
\[
\Db(\mod\BB)\xrightarrow{\Phi}\Db(\mod\AB) \mbox{ by } \Phi^2=\RHom_{\BB}(W,-).
\]
The composition $\Phi^2\circ\Phi^1\cong\RHom_{\AB}(W\otimes^{\bf L}_{\BB}V,-)$, so we show that $W\otimes^{\bf L}_{\BB}V\cong I_{\AB}$.

As above, $W:=\Hom_\mathfrak{R}(\upnu Y, Y)$.  Hence $W$ is quasi-isomorphic to its projective resolution
\begin{equation}\label{W proj res}
\hdots\to 0\to \Hom_\mathfrak{R}(\upnu Y,K_1^*)\xrightarrow{{ c\cdot \choose 0}} \Hom_\mathfrak{R}(\upnu Y,\mathfrak{R}_0)\oplus \Hom_\mathfrak{R}(\upnu Y,\mathfrak{R})\to 0\to \hdots
\end{equation}
But $\Phi^1$ is an equivalence that sends $V$ to $\BB$.  In fact, $\Phi^1$ sends
\begin{align*} \Hom_\mathfrak{R}(Y,\mathfrak{R}) & \mapsto \Hom_\mathfrak{R}(\upnu Y,\mathfrak{R})  \\
\Hom_\mathfrak{R}(Y,K_1^*) & \mapsto \Hom_\mathfrak{R}(\upnu Y,K_1^*). \end{align*} Hence applying the inverse of $\Phi^1$, namely $-\otimes^{\bf L}_{\BB} V$, to \eqref{W proj res} gives the complex
\begin{eqnarray}
\hdots\to 0\to \Hom_\mathfrak{R}(Y,K_1^*)\xrightarrow{{c\cdot \choose 0}} \Hom_\mathfrak{R}(Y,\mathfrak{R}_0)\oplus \Hom_\mathfrak{R}(Y,\mathfrak{R})\to 0\to \hdots\label{key obs}
\end{eqnarray}
Then using $K_1^*\cong K_0$ and the fact that \eqref{begin lambdacon}
\[
0\to \Hom_\mathfrak{R}(Y, K_1^*)\xrightarrow{c\cdot} \Hom_\mathfrak{R}(Y,\mathfrak{R}_0)\to \Hom_\mathfrak{R}(Y, N_1)\to \CA\to 0
\]
is exact, the complex (\ref{key obs}) is quasi-isomorphic to
\[
\hdots\to 0\to \AB=\Hom_\mathfrak{R}(Y,\mathfrak{R})\oplus\Hom_\mathfrak{R}(Y,N_1) \to \CA\to \hdots
\]
which is clearly quasi-isomorphic to $I_{\AB}$.\\
(2) is similar.  The final statements follow since the mutation functors $\Phi'$ are always derived equivalences \cite[6.8]{IW4}.
\end{proof}

\begin{cor}\label{track simple S}
$\Phi\circ\Phi(S)\cong S[-2]$.
\end{cor}
\begin{proof}
Let $J:=\Rad(\CA)$ denote the Jacobson radical, then $\CA/J\cong S$ and so
\[
S\cong 
(\CA/J)\otimes_{\CA}\CA\cong 
(\CA/J)\otimes^{\bf L}_{\CA}\CA.
\]
By \ref{track Lambda J 1}\eqref{track Lambda J 1 2} and \ref{mut mut functor general} it follows that $S[-2]\cong (\CA/J)\otimes^{\bf L}_{\CA}\RHom_{\AB}(I_{\AB},\CA)$, so the result follows from the fact that
\begin{align*}
(\CA/J)\otimes^{\bf L}_{\CA}\RHom_{\AB}(I_{\AB},\CA)
& \cong 
\RHom_{\AB}(I_{\AB},(\CA/J)\otimes^{\bf L}_{\CA}\CA) \\
& \cong
\RHom_{\AB}(I_{\AB},S),
\end{align*}
where the first canonical isomorphism is  standard (see e.g.\  \cite[2.10(2)]{IR}).
\end{proof}

The equivalence $\RHom_{\widehat{\Lambda}}(\widehat{I}_{\con},-)$ in \ref{mut mut functor general} above is the complete local version of our  noncommutative twist functor.  The following is a corollary of the results in this section, and will be used later.

\begin{cor}\label{track Lambda J} With the setup as above,
\begin{enumerate}
\item $\RHom_{\widehat{\Lambda}}(\widehat{I}_{\con},\FCA)\cong\FCA[-2]$.
\item $\RHom_{\widehat{\Lambda}}(\widehat{I}_{\con},\widehat{T})\cong \widehat{T}[-2]$.
\end{enumerate}
\end{cor} 
\begin{proof}
(1) We have $\RHom_{\widehat{\Lambda}}(\widehat{I}_{\con},\FCA)
\stackrel{{\scriptsize\mbox{\ref{mut mut functor general}}}}{\cong}
\Phi'\circ\Phi'(\FCA)
\stackrel{{\scriptsize\mbox{\ref{track Lambda J 1}}}}{\cong}
\FCA[-2]$.\\
(2) $\RHom_{\widehat{\Lambda}}(\widehat{I}_{\con},\widehat{T})
\stackrel{{\scriptsize\mbox{\ref{mut mut functor general}}}}{\cong}
\Phi'\circ\Phi'(\widehat{T})
=
\Phi'\circ\Phi'(\mathbb{F}S)
\stackrel{{\scriptsize\mbox{\ref{mut com diagram}}}}{\cong}
\mathbb{F}\circ\Phi\circ\Phi(S)
\stackrel{{\scriptsize\mbox{\ref{track simple S}}}}{\cong}
\mathbb{F}S[-2]
=\widehat{T}[-2]$.
\end{proof}

\section{Zariski Local Twists}\label{alg NC twists section}

We keep our running setup of a flopping contraction as in \S\ref{strategy}, with an irreducible rational curve contracting to a point $\m$. Since $R$ is Gorenstein, $R$ is a canonical $R$-module, but in this non-local setting canonical modules are not unique.  Throughout we set $\omega_R:=g^!\mathbb{C}[-3]$, where $g\colon \Spec R\to \Spec\mathbb{C}$ is the structure morphism.  This may or may not be isomorphic to $R$.

Throughout, as in \S\ref{strategy},  $\Lambda:=\End_R(R\oplus N)$ and we already know that $\Lambda\in\CM R$.  As before $I_{\con}:=[R]$ and $\Lambda_{\con}:=\Lambda/ I_{\con}$.  Since the category $\CM R$ is no longer Krull--Schmidt, we must be careful.

\subsection{Local tilting}\label{local tilting} In this subsection, we show that the $\Lambda$-module $I_{\con}$ is still a tilting module, with endomorphism ring $\Lambda$. First, we need the following lemma.

\begin{lemma}\label{bimod structures match} Put $I:=I_{\con}$, then
\begin{enumerate}
\item\label{bimod structures match 1} Viewing $I$ as a right $\Lambda$-module, we have $\Lambda\cong \End_\Lambda(I)$ under the map which sends $\lambda\in\Lambda$ to the map $(\lambda\cdot)\colon I\to I$ given by $i\mapsto \lambda i$.
\item\label{bimod structures match 2} Under the isomorphism in \t{(1)}, the bimodule ${}_{\End_\Lambda(I)}I_\Lambda$ coincides with the natural bimodule structure ${}_{\Lambda}I_\Lambda$. 
\end{enumerate}
\end{lemma}
\begin{proof}
(1) First recall that $\Lambda\cong\End_\Lambda(\Lambda_\Lambda)$ via the ring homomorphism which sends $\lambda\in\Lambda$ to $(\lambda\cdot)\colon\Lambda\to\Lambda$ sending $x\mapsto \lambda x$, left multiplication by $\lambda$. Now consider the short exact sequence
\begin{eqnarray}
0\to I\to\Lambda\to \Lambda_{\con}\to 0\label{can bimod}
\end{eqnarray}
of $\Lambda$-bimodules.  Applying $\Hom_\Lambda(-,\Lambda)$ gives an exact sequence
\begin{eqnarray}
0\to \Hom_\Lambda(\Lambda_{\con},\Lambda)\to\Hom_\Lambda(\Lambda,\Lambda)\to \Hom_\Lambda(I,\Lambda)\to \Ext^1_{\Lambda}(\Lambda_{\con},\Lambda).\label{can bimod cor 1}
\end{eqnarray}
Since $\Lambda\in\CM R$,  it follows that $\Lambda$ is singular $d$-CY by \cite[2.22(2)]{IW4}, hence 
\[
\dim_R\Lambda_{\con}=d-\inf\{i\geq 0\mid \Ext^i_{\Lambda}(\Lambda_{\con},\Lambda)\neq 0\}
\] 
by \cite[3.4(5)(ii)]{IR}.  But $R$ is normal, which implies that $\dim_R\Lambda_{\con}\leq d-2$ (see e.g.\ \cite[6.19]{IW4}), thus combining we see that $\Hom_{\Lambda}(\Lambda_{\con},\Lambda)\cong\Ext^1_{\Lambda}(\Lambda_{\con},\Lambda)=0$.

On the other hand, applying $\Hom_\Lambda(I,-)$ to \eqref{can bimod} gives
\begin{eqnarray}
0\to \Hom_\Lambda(I,I)\to \Hom_\Lambda(I,\Lambda)\to \Hom_\Lambda(I,\Lambda_{\con})=0.\label{can bimod cor 2}
\end{eqnarray}
Combining \eqref{can bimod cor 1} and \eqref{can bimod cor 2} shows that $\End_\Lambda(\Lambda)\cong \End_\Lambda(I)$.  Chasing through the sequences, this isomorphism is given by \opt{10pt}{$(\lambda\cdot)\in\End_\Lambda(\Lambda)\mapsto (\lambda\cdot)\in\End_\Lambda(I),$}\opt{12pt}{\[(\lambda\cdot)\in\End_\Lambda(\Lambda)\mapsto (\lambda\cdot)\in\End_\Lambda(I),\]} so the isomorphism is in fact a ring isomorphism.\\
(2) $I$ is naturally a left $\End_\Lambda(I)$-module via $f\cdot i:=f(i)$.  By the isomorphism in \eqref{bimod structures match 1}, this is just left multiplication by $\lambda$.
\end{proof}

We can now extend \ref{mut mut functor general}.

\begin{thm}\label{I tilts not complete local}
$I_{\con}$ is a tilting $\Lambda$-module with $\End_\Lambda(I_{\con})\cong\Lambda$.  Furthermore
\begin{enumerate}
\item\label{I tilts not complete local 1} $\pd_\Lambda I_{\con}=2$, so  $\pd_\Lambda \Lambda_{\con}=3$.
\item\label{I tilts not complete local 2} $\pd_{\Lambda^{\op}} \Lambda_{\con}=3$.
\end{enumerate}
\end{thm}
\begin{proof}
The fact that $\End_\Lambda(I)\cong\Lambda$ is \ref{bimod structures match}.  Let $\n\in\Max R$, then the short exact sequence $0\to I_{\con}\to\Lambda\to\Lambda_{\con}\to 0$ shows that $(I_{\con})_\n\cong\Lambda_\n$ if $\n\notin\Supp\Lambda_{\con}=\{\m\}$.  Further, 
\[
\Ext^i_{\Lambda}(I_{\con},I_{\con})_\m\otimes_{R_\m}\widehat{R}_\m\cong\Ext^i_{\widehat{\Lambda}}(\widehat{I}_{\con},\widehat{I}_{\con}) \stackrel{\scriptsize\mbox{\ref{mut mut functor general}}}{=} 0
\]
for all $i>0$.  Thus, in either case, the completion of $\Ext^i_{\Lambda}(I_{\con},I_{\con})$ at each maximal ideal is zero, hence $\Ext^i_{\Lambda}(I_{\con},I_{\con})=0$ for all $i>0$.  Similarly
\[
\pd_{\Lambda}I_{\con}=\sup\{\pd_{\widehat{\Lambda}}\widehat{\Lambda}_{\con}\mid {\n\in\Supp\Lambda_{\con}=\{\m\}}\}-1
\]
which is $2$ by \ref{proj res thm}\eqref{proj res thm 3}.  For generation, suppose that $Y\in\D(\Mod\Lambda)$ with \opt{10pt}{$\RHom_{\Lambda}(I_{\con},Y)=0$.}\opt{12pt}{\[\RHom_{\Lambda}(I_{\con},Y)=0.\]}  Then
\[
0=\RHom_{\Lambda}(I_{\con},Y)_\n\otimes\widehat{R}
\stackrel{\scriptsize\mbox{\ref{useful finite length stuff}}}{\cong}
\RHom_{\widehat{\Lambda}}(\widehat{I}_{\con}, \widehat{Y})
\]  
and so since $\widehat{I}_{\con}$ is tilting by \ref{mut mut functor general}, $\widehat{Y}=0$ for all $\n\in\Max R$. Hence $Y=0$, proving generation.  Part \eqref{I tilts not complete local 2} is shown similarly.
\end{proof}

\subsection{Tracking the contraction algebra} In this subsection we track the objects $\FCA$ and $T$ across the derived equivalence induced by $I_{\con}$ (extending \ref{track Lambda J}), and also under the action of a certain Serre functor.   

\begin{prop}\label{track contraction2} Under the setup above,
\begin{enumerate}
\item\label{track contraction2 1} $\RHom_\Lambda(I_{\con},\FCA)\cong \FCA[-2]$ as $\Lambda$-modules.
\item\label{track contraction2 2} $\RHom_\Lambda(I_{\con},T)\cong T[-2]$ as $\Lambda$-modules.
\end{enumerate}
\end{prop}
\begin{proof}
(1) We have 
\begin{align*}
\RHom_\Lambda(I_{\con},\FCA)_\n\otimes\widehat{R}
&\stackrel{\phantom{\scriptsize\mbox{\ref{track Lambda J}}}}{\cong} \RHom_{\widehat{\Lambda}}(\widehat{I}_{\con},\widehat{\FCA})\\
&\stackrel{\scriptsize\mbox{\ref{track Lambda J}}}{\cong}
\left\{ \begin{array}{cc} 0 & \mbox{if }\n\notin\Supp\FCA=\{\m\}\phantom{.}\\ \FCA[-2] &\mbox{if }\n\in\Supp\FCA=\{\m\}.\end{array}\right.
\end{align*}
It follows that $C:=\RHom(I_{\con},\FCA)$ is a stalk complex in homological degree two, and further that the degree two piece is supported only on $\m$.  Thus
\[
C\cong H^{2}(C)[-2]\stackrel{\scriptsize\mbox{\ref{useful finite length stuff}}}{\cong}\widehat{H^2(C)}[-2]
\cong\FCA[-2]
\]
as required.\\
(2) is similar, using $\RHom_{\widehat{\Lambda}}(\widehat{I}_{\con},\widehat{T})\cong \widehat{T}[-2]$ by \ref{track Lambda J}.
\end{proof}

In our possibly singular setting, we need the following RHom version of Serre functors.

\begin{defin}\label{Serre def for alg}
We say that a functor $\mathbb{S}\colon\Kb(\proj\Lambda)\to\Kb(\proj\Lambda)$ is a Serre functor relative to $\omega_R$ if there are functorial isomorphisms 
\begin{eqnarray}
\RHom_R(\RHom_\Lambda(a,b),\omega_R)\cong \RHom_\Lambda(b,\mathbb{S}(a))\label{Serre alg funct}
\end{eqnarray}
in $\D(\Mod R)$ for all $a\in\Kb(\proj\Lambda)$ and all $b\in\Db(\mod\Lambda)$.  
\end{defin}

Since under our assumptions in this section $R$ is Gorenstein, the module $\omega_R$ is a dualizing module of $R$, i.e.\
\[
\RHom_R(-,\omega_R)\colon\D(\mod R)\to\D(\mod R) 
\]
gives a duality.  Since $\Hom_R(-,\omega_R)\colon \mod\Lambda\to\mod\Lambda^{\op}$, this induces a duality
\[
\RHom_R(-,\omega_R)\colon \D^{\pm}(\mod\Lambda)\to\D^\mp(\mod\Lambda),
\]
and we define the functor $\mathbb{S}_\Lambda$ to be the composition
\[
\D^-(\mod\Lambda)\xrightarrow{\RHom_\Lambda(-,\Lambda)}\D^+(\mod\Lambda^{\op})\xrightarrow{\RHom_R(-,\omega_R)}\D^-(\mod\Lambda).
\]
As in the proof of \cite[3.5]{IR}, functorially $\mathbb{S}_\Lambda(-)\cong \RHom_R(\Lambda,\omega_R)\otimes^{\bf L}_{\Lambda}-$.  The following is known by \cite[7.2.14]{Ginz} and \cite[\S3]{IR}. We include the proof for completeness.

\begin{prop}\label{Serre alg side}
$\mathbb{S}_\Lambda$ is a Serre functor relative to $\omega_R$.
\end{prop}
\begin{proof}
First, since $\Lambda=\End_R(M)\in\CM R$, it follows that
\[
\id_\Lambda\Lambda=\dim R=\id_{\Lambda^{\op}}\Lambda
\]
by combining \cite[2.22]{IW4} and \cite[3.1(6), 3.4(1)]{IR}.  Given this, the fact that $\mathbb{S}_\Lambda$ maps $\Kb(\proj\Lambda)$ to $\Kb(\proj\Lambda)$ is then \cite[7.2.14(i)]{Ginz}, and the fact that \eqref{Serre alg funct} holds is \cite[7.2.8]{Ginz} (see also \cite[\S3]{IR}).
\end{proof}

The object $\FCA$ does not change under the action of this Serre functor.

\begin{prop}\label{Serre on contraction alg local}
$\mathbb{S}_{\Lambda}(\FCA)\cong\FCA$.
\end{prop}
\begin{proof}
We have 
\begin{align*}
\mathbb{S}_{\Lambda}(\FCA)_\n&\stackrel{\phantom{\scriptsize\mbox{\ref{useful finite length stuff}}}}{=}\RHom_R(\RHom_\Lambda(\FCA,\Lambda),\omega_R)_\n\\
&\stackrel{\scriptsize\mbox{\ref{useful finite length stuff}}}{\cong}\RHom_{R_\n}(\RHom_{\Lambda_\n}({\FCA}_\n,\Lambda_\n),(\omega_{R})_\n)
\end{align*}
which as above is isomorphic to $\RHom_{R_{\n}}(\Lambda_\n,(\omega_R)_\n)\otimes^{\bf L}_{\Lambda_\n}{\FCA}_\n$.  Now since $R$ is Gorenstein $(\omega_R)_\n\cong R_\n$, and further $\RHom_{R_{\n}}(\Lambda_\n,R_\n)\cong\Hom_{R_{\n}}(\Lambda_\n,R_\n)$ since $\Lambda\in\CM R$.  But since $\Lambda_\n$ is a symmetric $R_\n$-algebra \cite[2.4(3)]{IR}, $\Hom_{R_{\n}}(\Lambda_\n,R_\n)\cong\Lambda_{\n}$ as $\Lambda_\n$-bimodules.  Thus
\[
\mathbb{S}_{\Lambda}(\FCA)_\n\cong \Hom_{R_{\n}}(\Lambda_\n,R_\n)\otimes^{\bf L}_{\Lambda_\n} {\FCA}_\n\cong\left\{ \begin{array}{cc} 0 & \mbox{if }\n\notin\Supp\FCA=\{\m\}\phantom{.}\\ \FCA &\mbox{if }\n\in\Supp\FCA=\{\m\}.\end{array}\right.
\]
It follows that $\mathbb{S}_{\Lambda}(\FCA)$ is a stalk complex  in homological degree zero, and that degree zero piece has finite length, supported only on $\m$.  Thus
\[
\mathbb{S}_{\Lambda}(\FCA)\cong H^0(\mathbb{S}_{\Lambda}(\FCA))\stackrel{\scriptsize\mbox{\ref{useful finite length stuff}}}{\cong}H^0(\mathbb{S}_{\Lambda}(\FCA))_\m
\cong \FCA.\qedhere
\]
\end{proof}

\subsection{Algebraic noncommutative twist functors}
We can now give our definition of algebraic noncommutative twist functors.  The ring homomorphism $\Lambda\to\Lambda_{\con}$ gives rise to the standard extension--restriction of scalars adjunction.  To ease notation, we temporarily write $M:={}_{\Lambda}{(\Lambda_{\con})}_{\Lambda_{\con}}$ and $N:={}_{\Lambda_{\con}}{(\Lambda_{\con})}_{\Lambda}$. The adjunctions then read
\begin{eqnarray}
\begin{array}{c}
\begin{tikzpicture}[looseness=1,bend angle=15]
\node (a1) at (0,0) {$\Db(\mod \Lambda_{\con})$};
\node (a2) at (8,0) {$\Db(\mod \Lambda)$};
\draw[->] (a1) -- node[gap] {$\scriptstyle F_{\con}:=\RHom_{\Lambda_{\con}}(M,-)=-\otimes^{\bf L}_{\Lambda_{\con}}N$} (a2);
\draw[->,bend right] (a2) to node[above] {$\scriptstyle F_{\con}^{\LA}:=-\otimes^{\bf L}_{\Lambda}M$} (a1);
\draw[->,bend left] (a2) to node[below] {$\scriptstyle F_{\con}^{\RA}:=\RHom_{\Lambda}(N,-)$} (a1);
\end{tikzpicture}
\end{array}\label{one sided adjunction}
\end{eqnarray}
where we make use of the following:
\begin{lemma}
$F_{\con}$, $F_{\con}^{\RA}$ and $F_{\con}^{\LA}$ all preserve bounded complexes of finitely generated modules.
\end{lemma}
\begin{proof}
$F_{\con}$ preserves boundedness since restriction of scalars is exact on the abelian level.  We remark that if $Y\in\Mod\Lambda$ then $F_{\con}^{\RA}(Y)$ and $F_{\con}^{\LA}(Y)$ are both bounded since $\Lambda_{\con}$ has finite projective, hence finite flat, dimension as both a left and right $\Lambda$-module by \ref{I tilts not complete local}.  The fact that $F_{\con}^{\RA}$ and $F_{\con}^{\LA}$ both preserve boundedness then follows by induction on the length of the bounded complex.  The preservation of finite generation  is obvious.
\end{proof}

Now $\Lambda_{\con}\cong\widehat{\Lambda}_{\con}$ as algebras by \ref{contraction splits}, and there is a morita equivalence 
\begin{eqnarray}
\begin{array}{c}
\begin{tikzpicture}[xscale=1]
\node (d1) at (3,0) {$\mod \CA$};
\node (e1) at (7.5,0) {${}_{}\mod\widehat{\Lambda}_{\con}$};
\draw[->,transform canvas={yshift=+0.4ex}] (d1) to  node[above] {$\scriptstyle -\otimes_{\CA}\FCA $} (e1);
\draw[<-,transform canvas={yshift=-0.4ex}] (d1) to node [below]  {$\scriptstyle \Hom_{\widehat{\Lambda}_{\con}}(\FCA,-) $} (e1);
\end{tikzpicture}
\end{array}\label{morita for CA}
\end{eqnarray}
induced from \ref{morita lemma}.

\begin{lemma}
Composing \eqref{one sided adjunction} with \eqref{morita for CA} gives adjunctions
\begin{eqnarray}
\begin{array}{c}
\begin{tikzpicture}[looseness=1,bend angle=15]
\node (a1) at (0,0) {$\Db(\mod \CA)$};
\node (a2) at (7,0) {$\Db(\mod \Lambda)$};
\draw[->] (a1) -- node[gap] {$\scriptstyle F:=-\otimes^{\bf L}_{\CA}\FCA$} (a2);
\draw[->,bend right] (a2) to node[above] {$\scriptstyle F^{\LA}$} (a1);
\draw[->,bend left] (a2) to node[below] {$\scriptstyle F^{\RA}:=\RHom_{\Lambda}(\FCA,-)$} (a1);
\end{tikzpicture}
\end{array}\label{one sided adjunction morita}
\end{eqnarray}
\end{lemma}
\begin{proof}
This follows since $\FCA\otimes_{\Lambda_{\con}}\Lambda_{\con}\cong \FCA$ as $\CA$--$\Lambda$ bimodules.
\end{proof}

Now since $M\otimes_{\Lambda_{\con}}^{\bf L}N\cong M\otimes_{\Lambda_{\con}}N\cong {}_{\Lambda}(\Lambda_{\con})_{\Lambda}$ as $\Lambda$-bimodules, we have functorial isomorphisms 
\[
F\circ F^{\RA}\cong F^{\phantom{R}}_{\con}\circ F_{\con}^{\RA}\cong \RHom_\Lambda(M\otimes_{\Lambda_{\con}}^{\bf L}N,-)\cong \RHom_\Lambda({}_{\Lambda}(\Lambda_{\con})_{\Lambda},-)
\]
and 
\[
F\circ F^{\LA}\cong F^{\phantom{R}}_{\con}\circ F_{\con}^{\LA}\cong -\otimes^{\bf L}_{\Lambda}M\otimes^{\bf L}_{\Lambda_{\con}}N\cong -\otimes^{\bf L}_{\Lambda}{}_{\Lambda}(\Lambda_{\con})_{\Lambda}.
\]

\begin{defin}\label{def twist alg local}
Using the bimodule structure ${}_{\Lambda}I_\Lambda$ we define
\begin{enumerate}
\item The twist functor $T_{\con}\colon\Db(\mod\Lambda)\to\Db(\mod\Lambda)$ by $T_{\con}=\RHom_\Lambda(I_{\con},-)$.
\item The dual twist functor $T^*_{\con}\colon\Db(\mod\Lambda)\to\Db(\mod\Lambda)$ by $T_{\con}^*=-\otimes^{\bf L}_\Lambda I_{\con}$. 
\end{enumerate}
\end{defin}
The twist is an equivalence, and its inverse is the dual twist, by \ref{I tilts not complete local}.  The terminology `twist' is justified by the following lemma.
\begin{lemma}\label{funct triangles}
We have functorial triangles as follows 
\begin{enumerate}
\item $F\circ F^{\RA} \to \Id_{\Lambda} \to T_{\con} \to$
\item $T^*_{\con} \to \Id_{\Lambda} \to F\circ F^{\LA} \to$
\end{enumerate}
\end{lemma}
\begin{proof}
(1) The short exact sequence \eqref{can bimod} of $\Lambda$-bimodules gives rise to a triangle
\begin{eqnarray}
I_{\con}\to \Lambda\to \Lambda_{\con}\to \label{bimodtri1}
\end{eqnarray}
in the derived category of $\Lambda$-bimodules.   For $X\in\Db(\mod\Lambda)$, applying $\RHom_\Lambda(-,X)$ to  (\ref{bimodtri1}) gives a triangle
\[
\RHom_\Lambda(\Lambda_{\con},X)\to\RHom_\Lambda(\Lambda,X)\to \RHom_\Lambda(I_{\con},X)\to
\]
in $\Db(\mod\Lambda)$, which is simply
\[
F\circ F^{\RA}(X)\to \Id_{\Lambda}(X)\to T_{\con}(X)\to.
\]
(2) This follows by applying $X\otimes^{\bf L}_{\Lambda}-$ to  (\ref{bimodtri1}). 
\end{proof}

\subsection{Geometric local noncommutative twists}
\label{geom noncomm twist}

Again we keep the setup of \S\ref{strategy}. In this subsection we define the geometric noncommutative twist purely in terms of the geometry of $U$, and use it as the local model for our global twist functor later.

To fix notation, as in \eqref{derived equivalence} we write $\cV:=\cO_U\oplus\cN$ and thus we have a derived equivalence 
\begin{eqnarray}
\begin{array}{c}
\begin{tikzpicture}[looseness=1,bend angle=15]
\node (a1) at (0,0) {$\Db(\mod \Lambda)$};
\node (a2) at (6,0) {$\Db(\coh U)$};
\draw[->] (a1) -- node[gap] {$\scriptstyle H:=-\otimes^{\bf L}_{\Lambda}\cV$} (a2);
\draw[->,bend right] (a2) to node[above] {$\scriptstyle H^{\LA}:=\RHom_{U}(\cV,-)$} (a1);
\draw[->,bend left] (a2) to node[below] {$\scriptstyle H^{\RA}:=\RHom_{U}(\cV,-)$} (a1);
\end{tikzpicture}
\end{array}\label{db equiv diagram}
\end{eqnarray}
given by a tilting bundle $\cV$ and $\Lambda\cong\End_Y(\cV)$.  We know by \ref{highers vanish 2} that $\cEU$ (in degree zero) corresponds to $\FCA$ across the derived equivalence.

\begin{lemma}\label{adjoints Db compose lemma}
With notation and setup as in \S\ref{strategy}, composing \eqref{one sided adjunction morita} with \eqref{db equiv diagram} gives
\begin{eqnarray}\label{sph functor and adjoints defn}
\begin{array}{c}
\begin{tikzpicture}[looseness=1,bend angle=15]
\node (a1) at (0,0) {$\Db(\mod \CA)$};
\node (a2) at (8,0) {$\Db(\coh U)$};
\draw[->] (a1) -- node[gap] {$\scriptstyle G_U:=-\otimes^{\bf L}_{\CA}\cEU$} (a2);
\draw[->,bend right] (a2) to node[above] {$\scriptstyle G_U^{\LA}:=\RHom_{\CA}(\RHom_{U}(-,\cEU),\CA)$} (a1);
\draw[->,bend left] (a2) to node[below] {$\scriptstyle G_U^{\RA}:=\RHom_{U}(\cEU,-)$} (a1);
\end{tikzpicture}
\end{array}\label{EU intrinsic}
\end{eqnarray}
\end{lemma}
\begin{proof}
The composition $H\circ F$ is clear since $\FCA\otimes_{\Lambda}^{\bf L}\cV\cong\cEU$ by \ref{highers vanish 2}.  Also by \ref{highers vanish 2} $\RHom_U(\cV, \cEU)\cong \FCA$, so the composition $F^{\RA}\circ H^{\RA}$  is functorially isomorphic to 
\begin{align*}
\RHom_\Lambda(\FCA,\RHom_U(\cV,-))&\cong \RHom_\Lambda(\RHom_U(\cV,\cEU),\RHom_U(\cV,-))\\ 
&\cong \RHom_U(\cEU,-).
\end{align*}
Lastly, 
\begin{align*}
\RHom_{\CA}(F^{\LA}\circ H^{\LA}(-),\CA)&\cong
\RHom_{U}(-,H\circ F(\CA))\\
&\cong 
\RHom_{U}(-,\CA\otimes_{\CA}^{\bf L}\cEU)\\
&\cong
\RHom_{U}(-,\cEU).
\end{align*}
But since $\CA$ is self-injective $\RHom_{\CA}(-,\CA)=\Hom_{\CA}(-,\CA)$ is a duality, so dualizing both sides gives
\begin{align*}
F^{\LA}\circ H^{\LA}(-)&\cong
\RHom_{\CA}(\RHom_{U}(-,\cEU), \CA),
\end{align*}
as required.
\end{proof}

\begin{remark}
We remark that \eqref{EU intrinsic} is intrinsic to the geometry of $U$, and does not depend on the choice of $\Lambda$.
\end{remark}

\begin{defin}\phantomsection\label{local twist geom defin}
\begin{enumerate}
\item\label{local twist geom defin 1} We define $T_{\cEU}$ as the composition 
\[
\Db(\coh U)
\xrightarrow{\RHom_U(\cV,-)}
\Db(\mod \Lambda)
\xrightarrow{T_{\con}}
\Db(\mod \Lambda)
\xrightarrow{-\otimes^{\bf L}_{\Lambda}\cV}
\Db(\coh U)
\]
and call it the {\em geometric noncommutative twist functor}.
\item\label{local twist geom defin 2} We obtain an inverse twist $T^*_{\cEU}$ by composing similarly with $T^*_{\con}.$
\end{enumerate}
\end{defin}
Being a composition of equivalences, the geometric noncommutative twist is also an equivalence.  We remark that the definition of the geometric noncommutative twist functor relies on the tilting equivalence, but below in \ref{is FM} we show that it can be formulated as a Fourier--Mukai transform, by making use of the following setup.

Consider the commutative diagram
\begin{eqnarray}
\begin{array}{c}
\begin{tikzpicture}
\node (a1) at (0,1.5) {$U\times U$};
\node (a2) at (2,1.5) {$U$};
\node (b1) at (0,0) {$U$};
\node (b2) at (2,0) {$\pt$};
\draw[->] (a1) -- node[above] {$\scriptstyle p_1$} (a2);
\draw[->] (a1) -- node[left] {$\scriptstyle p_2$} (b1);
\draw[->] (a2) -- node[right] {$\scriptstyle \pi_U$} (b2);
\draw[->] (b1) -- node[above] {$\scriptstyle \pi_U$} (b2);
\draw[->] (a1) -- node[gap] {$\scriptstyle \pi_{U\times U}$} (b2);
\end{tikzpicture}
\end{array}\label{BC 1}
\end{eqnarray}
and set $\cV^{\vee}\boxtimes \cV=p_1^*\cV^{\vee}\otimes_{\cO_{U\times U}}^{\bf L}p_2^*\cV$.  Then there is an induced derived equivalence 
\begin{eqnarray}
\begin{array}{c}
\begin{tikzpicture}[xscale=1]
\node (d1) at (3,0) {$\Db(\coh U\times U)$};
\node (e1) at (9,0) {$\Db(\mod \Lambda\otimes_{\mathbb{C}}\Lambda^{\op})$};
\draw[->,transform canvas={yshift=+0.4ex}] (d1) to  node[above] {$\scriptstyle \RHom_{U\times U}(\cV^{\vee}\boxtimes \cV,-)$} (e1);
\draw[<-,transform canvas={yshift=-0.4ex}] (d1) to node [below]  {$\scriptstyle -\otimes^{\bf L}_{\Lambda^{\rm e}}(\cV^{\vee}\boxtimes \cV)$} (e1);
\end{tikzpicture}
\end{array}\label{bimodule adj geom alg}
\end{eqnarray}
as in \cite{BH}, where we denote the enveloping algebra by $\Lambda^{\rm e}:=\Lambda\otimes_{\mathbb{C}}\Lambda^{\op}$. 

\begin{notation}\label{notn O-Gamma modules} If $Y$ is a $\K$-scheme and $\Gamma$ is a $\K$-algebra, we define \[\cO_Y^\Gamma:=\cO_Y\otimes_{\underline{\mathbb{C}}}\underline{\Gamma},\] where $\underline{\mathbb{C}}$ and $\underline{\Gamma}$ are the constant sheaves associated to $\mathbb{C}$ and $\Gamma$.\end{notation}

To be slightly more pedantic, and armed with the notation above, we think of $\cV^{\vee}\boxtimes \cV$ as a left $\cO_{U\times U}^{\Lambda^{\rm e}}$-module, and the lower functor in \eqref{bimodule adj geom alg} is the derived version of the composition
\[
\mod \Lambda^{\rm e}\xrightarrow{\pi_{U\times U}^*(-)}\mod \cO_{U\times U}^{\Lambda^{\rm e}}
\xrightarrow{-\otimes_{\cO_{U\times U}^{\Lambda^{\rm e}}}(\cV^{\vee}\boxtimes \cV)}\mod \cO_{U\times U}.
\]

Now, we may view the tilting equivalences in \ref{local twist geom defin}\eqref{local twist geom defin 1} as Fourier--Mukai functors between $\Db(\coh U)$ and $\Db(\mod \Lambda)$, given by $\cV^{\vee}$ and $\cV$ respectively, considered as objects in $\Db(\mod \cO_U^{\Lambda^{(\op)}})$. The lower functor in \eqref{bimodule adj geom alg} should then be seen as taking a $\Lambda$-bimodule, and composing it with these Fourier--Mukai functors as in \ref{local twist geom defin}\eqref{local twist geom defin 1}. Hence, taking $\cW'$ to be the object in $\Db(\coh U\times U)$ corresponding to the $\Lambda$-bimodule $\RHom_\Lambda(I_{\con},\Lambda)$, we will find in \ref{is FM}\eqref{is FM 1} below that our twist $T_{\cEU}$ is itself a Fourier--Mukai functor with kernel $\cW'$. This motivates the following definition and lemma.

\begin{defin}\phantomsection\label{defin twist kernels}
\begin{enumerate}
\item\label{defin twist kernels 1} We define a {\em geometric twist kernel} $\cW'$ as follows
\[
\cW':=\pi_{U\times U}^*(\RHom_\Lambda(I_{\con},\Lambda))\otimes^{\bf L}_{\cO^{\Lambda^{\rm e}}_{U\times U}}(p_1^*\cV^{\vee}\otimes_{\cO_{U\times U}}^{\bf L}p_2^*\cV).
\]
\item\label{defin twist kernels 2} We obtain an inverse twist kernel $\cW$ as follows
\begin{eqnarray}
\cW:=\pi_{U\times U}^* I_{\con} \otimes^{\bf L}_{\cO^{\Lambda^{\rm e}}_{U\times U}}(p_1^*\cV^{\vee}\otimes_{\cO_{U\times U}}^{\bf L}p_2^*\cV).\label{WJ' expression}
\end{eqnarray}
\end{enumerate}
\end{defin}

The following lemma describes the twist and inverse twist.

\begin{lemma}\label{is FM}
There are functorial isomorphisms
\begin{enumerate}
\item\label{is FM 1} $T_{\cEU} \cong\FM{\cW'}:={\Rp_2}_*\big(p_1^*(-)\otimes^{\bf L}_{\cO_{U\times U}}\cW'\big),$
\item\label{is FM 2} $T^*_{\cEU} \cong\FM{\cW}.$
\end{enumerate}
\end{lemma}
\begin{proof}
Put $I:=I_{\con}$. We prove \eqref{is FM 2} first. We have
\begin{align*}
T^*_{\cEU}&=\pi_U^*(\RHom_{\cO_U}(\cV,-)\otimes^{\bf L}_{\Lambda}I)
\otimes^{\bf L}_{\cO_U^\Lambda}\cV\tag{\ref{local twist geom defin}\eqref{local twist geom defin 2}}\\
&\cong\pi_U^*(\RGamma(-\otimes^{\bf L}_{\cO_U}\cV^{\vee})\otimes^{\bf L}_{\Lambda}I)
\otimes^{\bf L}_{\cO_U^\Lambda}\cV\tag{definition of $\RHom$}\\
&\cong{\Rp_2}_*p_1^*(-\otimes^{\bf L}_{\cO_U}\cV^{\vee})\otimes^{\bf L}_{\cO_U^\Lambda}\pi_U^*I
\otimes^{\bf L}_{\cO_U^\Lambda}\cV\tag{\mbox{base change for \eqref{BC 1}}}\\
&\cong{\Rp_2}_*\left(p_1^*(-\otimes^{\bf L}_{\cO_U}\cV^{\vee})\otimes^{\bf L}_{\cO_{U\times U}^\Lambda}p_2^*(\pi_U^*I
\otimes^{\bf L}_{\cO_{U}^\Lambda}\cV)\right)\tag{\mbox{projection formula}}\\
&\cong{\Rp_2}_*\left(p_1^*(-)\otimes_{\cO_{U\times U}}^{\bf L}\left(p_1^*\cV^{\vee}\otimes^{\bf L}_{\cO_{U\times U}^\Lambda}p_2^*\pi_U^*I\otimes^{\bf L}_{\cO_{U\times U}^\Lambda}p_2^*\cV\right)\right)\\
&\cong{\Rp_2}_*\left(p_1^*(-)\otimes_{\cO_{U\times U}}^{\bf L}\left(\pi_{U\times U}^*I
\otimes^{\bf L}_{\cO_{U\times U}^{\Lambda^{\rm e}}}(p_1^*\cV^{\vee}\otimes^{\bf L}_{\cO_{U\times U}} p_2^*\cV)\right)\right)\\
&\cong{\Rp_2}_*\big(p_1^*(-)\otimes^{\bf L}_{\cO_{U\times U}}\cW\big).\tag{\mbox{by \eqref{WJ' expression}}}
\end{align*}Since $T_{\con}=\RHom_{\Lambda}(I,-)=(-)\otimes^{\bf L}_{\Lambda}\RHom_{\Lambda}(I,\Lambda)$, a similar calculation using \ref{local twist geom defin}\eqref{local twist geom defin 1} gives \eqref{is FM 1}.
\end{proof}

For our purposes later we must track the object $\cEU$ under a certain Serre functor.  Again, since $U$ is possibly singular, this involves the RHom version. Throughout, if $Y$ is a $\mathbb{C}$-scheme of dimension $d$, although canonical modules are not unique, we write $\cD_Y:=g^! \mathbb{C}$ for the dualizing complex of $Y$, where $g\colon Y\to\Spec \mathbb{C}$ is the structure morphism.  If further $Y$ is normal CM of dimension $d$, we will always write $\omega_Y$ for $\cD_Y[-d]$, and refer to $\omega_Y$ as the {\em geometric canonical}.

\begin{defin}\label{Serre def}
Suppose that $f\colon Y\to\Spec T$ is a morphism where $T$ is a normal CM $\mathbb{C}$-algebra with canonical module $\omega_T$. Typically
\begin{itemize}
\item $T=\mathbb{C}$ (in the case $Y=X$ is projective), or
\item $T=R$ with $\Spec R$ a 3-fold and $f$ is also birational (in the case $Y=U$ of the local model).
\end{itemize}
We say that a functor $\mathbb{S}\colon \Perf(Y)\to\Perf(Y)$ is a \emph{Serre functor relative to $\omega_T$} if there are functorial isomorphisms
\[
\RHom_T(\RHom_Y(\cF,\cG),\omega_T)\cong \RHom_Y(\cG,\mathbb{S}(\cF))
\]
in $\D(\Mod T)$ for all $\cF\in\Perf(Y)$, $\cG\in\Db(\coh Y)$. 
\end{defin}

\begin{lemma}\label{Serre for X}
Suppose that $f\colon Y\to\Spec T$ is a projective morphism, where $Y$ and $\Spec T$ are both Gorenstein varieties.  Then 
\[
\mathbb{S}_Y:=-\otimes_Y \omega_Y[\dim Y-\dim T]\colon \Perf(Y)\to\Perf(Y)
\]
is a Serre functor relative to $\omega_T$.
\end{lemma}
\begin{proof}
First, since $Y$ is Gorenstein, $\omega_Y$ is a line bundle, so tensoring does preserve perfect complexes.  The remainder is just Grothendieck duality: since $f^!\omega_T=\omega_Y[\dim Y-\dim T]$ we have
\opt{10pt}{\[
\RHom_Y(\cG,\mathbb{S}_Y(\cF))\cong\RHom_Y(\RsHom_Y(\cF,\cG),f^!\omega_T)\cong\RHom_T(\RHom_Y(\cF,\cG),\omega_T) 
\]}
\opt{12pt}{\begin{align*}
\RHom_Y(\cG,\mathbb{S}_Y(\cF))&\cong\RHom_Y(\RsHom_Y(\cF,\cG),f^!\omega_T)\\&\cong\RHom_T(\RHom_Y(\cF,\cG),\omega_T) 
\end{align*}}
in $\D(\Mod T)$ for all $\cF\in\Perf(Y)$, $\cG\in\Db(\coh Y)$. 
\end{proof}

We now revert back to the assumptions and setup of \S\ref{strategy}.  The following are geometric versions of \ref{Serre on contraction alg local} and \ref{track contraction2}.

\begin{lemma}\label{cE local tensor}
$\mathbb{S}_U(\cEU)\cong\cEU$, i.e.\ $\cEU\otimes\omega_U\cong\cEU$.
\end{lemma}
\begin{proof}
If we temporarily denote the derived equivalence by $\Psi:=\RHom_U(\cV,-)$ with inverse $\Phi:=-\otimes^{\bf L}_{\Lambda}\cV$, then  on $\Lambda$,  $\Psi\circ \mathbb{S}_U\circ\Phi$ is a Serre functor relative to $\omega_R$, so since Serre functors are unique with respect to a fixed canonical,  $\mathbb{S}_\Lambda\cong \Psi\circ \mathbb{S}_U\circ\Phi$. Hence
\[
\mathbb{S}_U(\cEU)
\stackrel{\scriptsize\mbox{\ref{highers vanish 2}}}{\cong}
\mathbb{S}_U(\Phi \FCA)\cong\Phi\Psi\mathbb{S}_U(\Phi \FCA)\cong\Phi\mathbb{S}_\Lambda(\FCA)
\stackrel{\scriptsize\mbox{\ref{Serre on contraction alg local}}}{\cong}\Phi\FCA
\stackrel{\scriptsize\mbox{\ref{highers vanish 2}}}{\cong}
\cEU.\qedhere
\]
\end{proof}

\begin{prop}\label{track contraction2FM} With notation as above,
\begin{enumerate}
\item\label{track contraction2FM 1} $T_{\cE_U}(\cE_U)\cong \cE_U[-2]$ and $T_{\cE_U}(E)\cong E[-2]$.
\item\label{track contraction2FM 2} $T^*_{\cE_U}(\cE_U)\cong \cE_U[2]$ and $T^*_{\cE_U}(E)\cong E[2]$.
\end{enumerate}
\end{prop}
\begin{proof}
(1) By definition $T_{\cE_U}$ is the composition
\[
\Db(\coh U)
\xrightarrow{\RHom_U(\cV,-)}
\Db(\mod \Lambda)
\xrightarrow{T_{\con}}
\Db(\mod \Lambda)
\xrightarrow{-\otimes^{\bf L}_{\Lambda}\cV}
\Db(\coh U).
\]
We know that the first functor takes $\cE_U$ to $\FCA$ by \ref{highers vanish 2}\eqref{highers vanish 2 1}, the second functor takes $\FCA$ to $\FCA[-2]$ by \ref{track contraction2}\eqref{track contraction2 1}, and the third functor takes $\FCA[-2]$ to $\cE_U[-2]$, again by \ref{highers vanish 2}\eqref{highers vanish 2 1}.  Tracking $E$ is similar.\\
(2) This follows immediately from \eqref{track contraction2FM 1}, since as we have already observed,  $T^*_{\cE_U}$ is the inverse of $T_{\cE_U}$.
\end{proof}

\section{Noncommutative Twist Functors: Global Case}\label{global section}

In this section we globalize the previous sections to obtain noncommutative twists of projective varieties.  We keep the assumptions and setup of \S\ref{strategy}.

\subsection{Global $\Ext$ vanishing}\label{global flops}

Using \ref{highers vanish 1}, we can deduce Ext vanishing on $X$ by reducing the problem to Ext vanishing on $U$, which we have already solved in  \ref{Ext cor}.  Also, we remark that the following shows that although $i_*E$ need not be perfect, its noncommutative thickening $\cE$ automatically is.

\begin{thm}\label{Gsph conditions ok}  We keep the assumptions as above, in particular $X$ is a projective normal $3$-fold $X$ with at worst Gorenstein terminal singularities.  Then $\cE$ is a perfect complex and further
\[
\Ext_X^t(\cE,i_*E) =\left\{ \begin{array}{cl} \mathbb{C}&\mbox{if }t=0,3\\
0&\mbox{else}.\\  \end{array} \right.
\]
\end{thm}
\begin{proof}  
Now the first assertion in the statement is local, so it suffices to check that $\cE_x$ has finite projective dimension for all closed points $x\in X$.  Restricting $\cE$ to $U$, by \ref{highers vanish 1} it is clear that $i^*\cE=\cEU$, which across the derived equivalence \eqref{derived equivalence} corresponds to $\FCA$ by \ref{highers vanish 2}.  Since $\FCA$ is supported only on the maximal ideal $\m\in\Max R$ by \ref{cont is fd for flop}, it follows that $\pd_\Lambda\FCA=\pd_{\widehat{\Lambda}}\FCA=\pd_{\AB}\CA$ which by \ref{proj res thm} is finite.  Hence back across the equivalence $i^*\cE$ is a perfect $\cO_U$-module.  Thus $\cE_u$ has finite projective dimension for all $u\in U$.  Since $\cE_x=0$ for all $x\notin U$, this implies that $\cE$ is perfect.

For the second assertion, we have
\[
 \Ext^t_X(\cE,i_*E)
 \stackrel{\t{\scriptsize{\eqref{inclusion line}}}}{\cong}
\Ext^t_U(\cEU,E)\cong \Ext^t_{\Lambda}(\FCA,T)
\]
and so the result follows from \ref{Ext cor}.
\end{proof}

The following result will be needed later.  Again \ref{highers vanish 1} reduces the proof to the local model, allowing us to use \ref{cE local tensor}.

\begin{prop}\label{EX tensor global}
$\cE\otimes_X\omega_X\cong\cE$.
\end{prop}
\begin{proof}
Since $X$ is Gorenstein, $\omega_X$ is a line bundle and so 
\opt{10pt}{\[
\cE\otimes_X\omega_X
\stackrel{\t{\scriptsize{\ref{highers vanish 1}}}}{\cong}
\Ri_*\cEU\otimes_X^{\bf L}\omega_X
\cong 
\Ri_*(\cEU\otimes_U^{\bf L} i^*\omega_X)
\cong
\Ri_*(\cEU\otimes_U^{\bf L} \omega_U)
\stackrel{\t{\scriptsize{\ref{cE local tensor}}}}{\cong}
\Ri_*(\cEU)
\stackrel{\t{\scriptsize{\ref{highers vanish 1}}}}{\cong}
\cE.
\]}
\opt{12pt}{\begin{align*}
\cE\otimes_X\omega_X
\stackrel{\t{\scriptsize{\ref{highers vanish 1}}}}{\cong}
\Ri_*\cEU\otimes_X^{\bf L}\omega_X
&\cong 
\Ri_*(\cEU\otimes_U^{\bf L} i^*\omega_X)\\
&\stackrel{\t{\scriptsize{\ref{shriek=pull}}}}{\cong}
\Ri_*(\cEU\otimes_U^{\bf L} \omega_U)\\
&\stackrel{\t{\scriptsize{\ref{cE local tensor}}}}{\cong}
\Ri_*(\cEU)\\
&\stackrel{\t{\scriptsize{\ref{highers vanish 1}}}}{\cong}
\cE.\qedhere
\end{align*}}\end{proof}

\subsection{Global inverse twist}\label{section global inverse twist}
We now extend the definition of the local inverse twist $T^*_{\cE_U}$ on the open neighbourhood $U$ (from \ref{local twist geom defin}) to a global inverse twist $T^*_{\cE}$ on $X$.  It turns out to be technically much easier to extend $T^*_{\cE_U}$ rather than to extend $T_{\cE_U}$, so we do this first.  We return to  the problem of defining the global noncommutative twist in \S\ref{newsubsection}.   

Throughout, we keep the setting and assumptions of \S\ref{strategy}.  We begin with a slight refinement of the results of \S\ref{geom noncomm twist}. Consider the triangle of $\Lambda$-bimodules \eqref{bimodtri1} 
\begin{align*}
I_{\con}\to & \Lambda\to \Lambda_{\con}\to ,
\end{align*}
then applying the functor  $- \otimes_{\Lambda^e}^{\bf L} (\cV^\vee \boxtimes \cV)$ we obtain an exact triangle
\begin{eqnarray}
\cW_U \to \cO_{\Delta,U}  \xrightarrow{\phi_U} \cQ_U\to\label{localFMkernel}
\end{eqnarray} 
of Fourier--Mukai kernels on $U\times U$.

\begin{defin}\label{global twist defin}
The \emph{inverse twist} $T^*_{\cE}$ is defined to be the Fourier--Mukai functor $\FM{\cW_X}$ given by the kernel $\cW_X := \operatorname{Cone}(\phi_X)[-1]$, where $\phi_X$ is the composition
\begin{equation}\label{inverse twist cone morphism}
\cO_{\Delta,X} \xrightarrow{\eta_\Delta} \RDerived(i \times i)_* \cO_{\Delta,U} \xrightarrow{\RDerived(i \times i)_* \phi_U} \RDerived(i \times i)_* \cQ_U=:\cQ_X.
\end{equation}
Here $\eta_\Delta$ is defined to be
\begin{align}
\label{unit for U to X inclusion} 
\RDerived\Delta_{X*}\cO_X\xrightarrow{\RDerived\Delta_{X*}\eta}\RDerived\Delta_{X*} \RDerived i_*i^*\cO_X\cong\RDerived(i\times i)_* \RDerived\Delta_{U*}\cO_U,
\end{align}
with $\eta\colon \cO_X\to\RDerived i_*i^*\cO_X$ the unit morphism. 
\end{defin}

Since $\RDerived (i\times i)_*$ does not in general preserve coherence, we must work hard to establish that $\cW_X, \cQ_X\in\Db(\coh X\times X)$.  To do this, we first establish the triangle representation for $T^*_{\cE}$ in \eqref{t formula for T*} below, and thence show that the Fourier--Mukai functors for $\cW_X$ and $\cQ_X$ preserve $\Db(\coh X)$.

Consider the adjunctions in \eqref{sph functor and adjoints defn}, which were used to describe the local inverse twist $T^*_{\cE_U}$. These adjunctions are induced from adjunctions on the unbounded level, and composing them with \eqref{inclusion line} yields the diagram of adjoints
\begin{eqnarray}\label{orig sph functor and adjoints defn global}
\begin{array}{c}
\begin{tikzpicture}[looseness=1,bend angle=15]
\node (a1) at (0,0) {$\D(\Mod \CA)$};
\node (a2) at (8,0) {$\D(\Qcoh X)$};
\draw[->] (a1) -- node[gap] {$\scriptstyle G_X:=\Ri_*(-\otimes^{\bf L}_{\CA}\cEU)$} (a2);
\draw[->,bend right] (a2) to node[above] {$\scriptstyle G_X^{\LA}:=G_U^{\LA}\circ i^*$} (a1);
\draw[->,bend left] (a2) to node[below] {$\scriptstyle G_X^{\RA}:=G_U^{\RA}\circ i^!$} (a1);
\end{tikzpicture}
\end{array}
\end{eqnarray}
By the projection formula we have $G_X\cong-\otimes^{\bf L}_{\CA}\Ri_*\cEU$, so by \ref{highers vanish 1} and adjunction, the above yields
\begin{eqnarray}
\begin{array}{c}
\begin{tikzpicture}[looseness=1,bend angle=15]
\node (a1) at (0,0) {$\D(\Mod \CA)$};
\node (a2) at (7,0) {$\D(\Qcoh X)$};
\draw[->] (a1) -- node[gap] {$\scriptstyle G_X\cong-\otimes^{\bf L}_{\CA}\cE$} (a2);
\draw[->,bend right] (a2) to node[above] {$\scriptstyle G_X^{\LA}$} (a1);
\draw[->,bend left] (a2) to node[below] {$\scriptstyle G_X^{\RA}\cong\RHom_X(\cE,-)$} (a1);
\end{tikzpicture}
\end{array}\label{thelastlabel!}
\end{eqnarray}
We remark that the above diagram is intrinsic to the geometry of $X$.  The following lemma establishes, amongst other things, that the inverse twist preserves the bounded derived category.

\begin{lemma}\label{funct triangles global} Under the setup above,
\begin{enumerate}
\item\label{funct triangles global 1} $\FM{\cQ_X}\cong G_X^{\phantom{L}}\circ G_X^{\LA}$, and so for all $x\in\D(\Qcoh X)$, there is a triangle
\begin{align}
T^*_{\cE}(x)\to x\to G^{\phantom{L}}_X\circ G_X^{\LA}(x)\to.\label{t formula for T*}
\end{align}
\item\label{funct triangles global 2} Both $\FM{\cQ_X}$ and the inverse twist $T^*_{\cE}$ preserve $\Db(\coh X)$.
\end{enumerate}
\end{lemma}
\begin{proof} 
(1) From the definition \ref{global twist defin}, there is a functorial triangle
\begin{align*}
T^*_{\cE} \to \Id_{X} \to \FM{\cQ_X} \to.
\end{align*}
Now since $\cQ_X=\RDerived(i\times i)_*\cQ_U$ by definition \eqref{inverse twist cone morphism}, necessarily
\[
\FM{\cQ_X} \cong \Ri_* \circ \FM{\cQ_U} \circ i^*
\]
by \cite[5.12]{HuybrechtsFM}.  By the argument of \ref{is FM}, we have that $\FM{\cQ_U}\cong G_U^{\phantom{L}}\circ G_U^{\LA}$, hence
\begin{equation}\label{equation action part of twist}
\FM{\cQ_X}\cong \Ri_* \circ \FM{\cQ_U} \circ i^*\cong \Ri_* \circ \, G^{\phantom{L}}_U\circ G_U^{\LA} \circ i^*= G^{\phantom{L}}_X\circ G_X^{\LA}
\end{equation}
and so the statement follows. \\
(2) Since $\CA$ is a finite dimensional algebra, its derived category $\Db(\mod \CA)$ is generated by its simple module $S$.  Note that $G_X$ sends $S$ to $i_*E\in\Db(\coh X)$, thus it follows that $G_X$ sends $\Db(\mod \CA)$ to $\Db(\coh X)$.   Since $G_X^{\LA}$ is the composition $G_U^{\LA}\circ i^*$, and by \eqref{EU intrinsic} $G_U^{\LA}$ takes $\Db(\coh X)$ to $\Db(\mod \CA)$, it follows that $G_X^{\LA}$ sends $\Db(\coh X)$ to $\Db(\mod \CA)$.   Combining, we see that $G_X^{\phantom{L}}\circ G_X^{\LA}$ takes $\Db(\coh X)$ to $\Db(\coh X)$, so the claim that the inverse twist preserves $\Db(\coh X)$ follows immediately from the triangle in \eqref{funct triangles global 1}, using the two-out-of-three property.
\end{proof}

\subsection{Existence of adjoints}\label{existence of adjoints}

The existence of left and right adjoints of our inverse twist follows directly from the following known theorem \cite{HMS07, HMS09}.  Recall that for projections $p_1,p_2\colon Y\times Y\to Y$, an object $\cG\in\D(\Qcoh Y\times Y)$ is said to have \emph{finite projective dimension} with respect to $p_i$, if $\RsHom_{\cO_{Y\times Y}}(\cG, p_i^!\cF)\in\Db(\Qcoh Y\times Y)$ for all $\cF\in\Db(\Qcoh Y)$.

\begin{thm}\label{existence of adjoints thm} Suppose that $Y$ is a projective variety with only Gorenstein singularities, and $\cG\in\Db(\Qcoh Y\times Y)$.  If $\FM{\cG}:=\RDerived p_{2 *}(\cG\otimes^{\bf L}_{\cO_{Y\times Y}}p_1^{*}(-))$ preserves $\Db(\coh Y)$, then
\begin{enumerate}
\item\label{existence of adjoints thm 1} $\cG\in\Db(\coh Y\times Y)$.
\item\label{existence of adjoints thm 2} $\cG$ has finite projective dimension with respect to $p_1$.
\end{enumerate}
If furthermore $\RDerived p_{1 *}(\cG\otimes^{\bf L}_{\cO_{Y\times Y}}p_2^{*}(-))$ also preserves $\Db(\coh Y)$, then
\begin{enumerate}
\setcounter{enumi}{2}
\item\label{existence of adjoints thm 3} $\cG$ has finite projective dimension with respect to both $p_1$ and $p_2$, and $\FM{\cG}$ admits left and right adjoints which also preserve $\Db(\coh Y)$.  The right adjoint of $\FM{\cG}$ is 
\[
\FM{\cG^\vee\otimes p_1^*\omega_Y[\dim Y]}.
\]
\end{enumerate}
\end{thm}
\begin{proof}
(1) Take a very ample bundle $\cO(1)$ on $Y$, then $p_1^*\cO(1)$ is very ample with respect to the morphism $p_2$.  Since by assumption $\RDerived p_{2 *}(\cG\otimes^{\bf L}_{\cO_{Y\times Y}}p_1^{*}\cO(r))$ has coherent cohomology for all $r$, it follows that $\cG\in\Db(\coh Y\times Y)$ by \cite[remark below 2.5]{HMS09}.\\
(2) By \eqref{existence of adjoints thm 1} $\cG\in\Db(\coh Y\times Y)$, so since by assumption its Fourier--Mukai functor $\FM{\cG}$ takes $\Db(\coh Y)$ to $\Db(\coh Y)$, by \cite[2.7]{HMS09} $\cG$ must have finite homological dimension with respect to $p_1$ (for the definition of this, see \cite[1.3]{HMS07}).  Hence by \cite[1.6]{HMS07} $\cG$ has finite projective dimension with respect to $p_1$.\\
(3) With the additional assumption, applying \eqref{existence of adjoints thm 2} with the roles of $p_1$ and $p_2$ exchanged we see that $\cG$ also has finite projective dimension with respect to $p_2$.  Thus by \cite[1.17]{HMS07} $\FM{\cG}$ now has both left and right adjoints, which also preserve $\Db(\coh Y)$, and the right adjoint is of the form above.
\end{proof}

The following is a consequence of \ref{funct triangles global} and \ref{existence of adjoints thm}.
\begin{cor}\label{our NC twist has adjoints}
We keep the assumptions as in \S\ref{strategy}, and consider the triangle 
\[
\cW_X\to\cO_{\Delta,X}\to\cQ_X\to\]
from \ref{global twist defin}.  Then
\begin{enumerate}
\item\label{our NC twist has adjoints 1} $\RDerived p_{2 *}(\cQ_X\otimes^{\bf L}_{\cO_{X\times X}}p_1^{*}(-))$ and $\RDerived p_{1 *}(\cQ_X\otimes^{\bf L}_{\cO_{X\times X}}p_2^{*}(-))$ both preserve $\Db(\coh X)$.
\item\label{our NC twist has adjoints 2} $\cW_X$, $\cO_{\Delta,X}$ and $\cQ_X$ all belong to $\Db(\coh X\times X)$ and have finite projective dimension over both factors.  Hence the corresponding Fourier--Mukai functors have right and left adjoints, which preserve $\Db(\coh X)$.  
\end{enumerate}
\end{cor}
\begin{proof}
(1) We have $\FM{\cQ_X}=\RDerived p_{2 *}(\cQ_X\otimes^{\bf L}_{\cO_{X\times X}}p_1^{*}(-))$, which we established preserves $\Db(\coh X)$ in \ref{funct triangles global}\eqref{funct triangles global 2}.   For the other, we note that by \cite[5.12]{HuybrechtsFM} 
\begin{eqnarray}
\RDerived p_{1 *}(\cQ_X\otimes^{\bf L}_{\cO_{X\times X}}p_2^{*}(-))\cong \RDerived i_*\circ \RDerived p_{1 *}(\cQ_U\otimes^{\bf L}_{\cO_{U\times U}}p_2^{*}(-))\circ i^*\label{cheeky way}
\end{eqnarray}
where by abuse of notation $p_i$ also denote the projections  $U\times U\to U$.  But again using the argument in \ref{is FM} it follows immediately that $\RDerived p_{1 *}(\cQ_U\otimes^{\bf L}_{\cO_{U\times U}}p_2^{*}(-))$ is the composition
\[
\Db(\coh U)\xrightarrow{\RHom_U(\cV^\vee,-)}\Db(\mod\Lambda^{\op})\xrightarrow{\Lambda_{\con}\otimes_\Lambda^{\bf L}-}\Db(\mod\Lambda^{\op})\xrightarrow{\cV^\vee\otimes_{\Lambda}^{\bf L}-}\Db(\coh U)
\]
so every object in the image is a bounded complex of coherent sheaves filtered by $E$.  Since $\RDerived i_* E=i_*E$ by \eqref{inclusion line} and is certainly bounded coherent, it follows that \eqref{cheeky way} preserves $\Db(\coh X)$.\\
(2) The fact that $\cW_X, \cQ_X\in\Db(\coh X\times X)$ is a consequence of \ref{funct triangles global}\eqref{funct triangles global 2} and \ref{existence of adjoints thm}\eqref{existence of adjoints thm 1}.  Thus by \eqref{our NC twist has adjoints 1} and \ref{existence of adjoints thm}\eqref{existence of adjoints thm 3}, $\cQ_X$ has finite projective dimension over both factors.  Since $\cO_{\Delta,X}$ also has this property, by two-out-of-three so does $\cW_X$.  The remaining statements now follow, again using \ref{existence of adjoints thm}\eqref{existence of adjoints thm 3}.
\end{proof}

\begin{remark}
If one could establish that the inverse twist $T^*_{\cE}$ has both left and right adjoints that preserve $\Db(\coh X)$ under the weaker assumption that $X$ is only quasi-projective, then the proof that $T^*_{\cE}$ is an autoequivalence in the next subsection goes through in that level of generality.
\end{remark}

\subsection{Proof of equivalence for global flops}\label{global proof subsection} In this subsection, keeping the assumptions and setting of \S\ref{strategy}, we give the proof that the inverse twist $T^*_{\cE}$ is a derived equivalence in the situation of a flopping contraction.  As above, our strategy is to reduce as much as possible to the local case and apply the results of \S\ref{alg NC twists section}.  With this in mind, the following is useful.

\begin{prop}\label{diagram commutes U to X}  The following diagram is naturally commutative:\\
\[
\begin{array}{c}
\opt{10pt}{\begin{tikzpicture}}
\opt{12pt}{\begin{tikzpicture}[scale=1.2]}
\node (a1) at (0,0) {$\D(\Qcoh U)$};
\node (a2) at (2.5,0) {$\D(\Qcoh X)$};
\node (b1) at (0,-1.5) {$\D(\Qcoh U)$};
\node (b2) at (2.5,-1.5) {$\D(\Qcoh X)$};
\draw[->] (a1) to node[above] {$\scriptstyle \Ri_*$} (a2);
\draw[->] (b1) to node[above] {$\scriptstyle \Ri_*$} (b2);
\draw[->] (a1) to node[left] {$\scriptstyle T^*_{\cEU}$} (b1);
\draw[->] (a2) to node[right] {$\scriptstyle T^*_{\cE}$} (b2);
\end{tikzpicture}
\end{array}
\]

\end{prop}
\begin{proof} Translating the question into Fourier--Mukai kernels, by \cite[5.12]{HuybrechtsFM} we require
\begin{equation}\label{equation.local_global_twist_kernels}
(i \times 1)^* \cW_X \cong \RDerived(1 \times i)_* \cW_U.
\end{equation}
We evaluate the left-hand side of \eqref{equation.local_global_twist_kernels}. By definition (\ref{global twist defin}), $\cW_X = \operatorname{Cone}(\phi_X)[-1]$, where $\phi_X$ factors through the morphism
\[ 
\cO_{\Delta,X} \xrightarrow{\eta_\Delta}  \RDerived(i \times i)_* \cO_{\Delta,U} .
\]
By definition \eqref{unit for U to X inclusion}, the morphism $\eta_\Delta$  is given (up to isomorphism) as a pushforward $\RDerived \Delta_{X*} \eta$ of a unit morphism $\eta$  associated to the adjunction $ i^* \dashv \RDerived i_*$. A triangular identity for this adjunction yields that $i^* \eta$ is an isomorphism, and so using flat base change around the following Cartesian square
\[
\begin{array}{c}
\begin{tikzpicture}
\node (a1) at (0,0) {$U$};
\node (a2) at (2.5,0) {$U\times X$};
\node (b1) at (0,-1.5) {$X$};
\node (b2) at (2.5,-1.5) {$X\times X$};
\draw[->] (a1) to node[above] {$\scriptstyle 1\times i$} (a2);
\draw[->] (b1) to node[above] {$\scriptstyle  \Delta_X$} (b2);
\draw[->] (a1) to node[left] {$\scriptstyle  i$} (b1);
\draw[->] (a2) to node[right] {$\scriptstyle  i\times 1$} (b2);
\end{tikzpicture}
\end{array}
\]
we find that $(i\times 1)^* \eta_\Delta$ is also an isomorphism.

It then follows, using the definition of $\phi_X$ \eqref{inverse twist cone morphism}, that 
\begin{align*}
(i \times 1)^* \cW_X & \cong (i \times 1)^* \operatorname{Cone}( \RDerived(i \times i)_* \phi_U )[-1] \\
& \cong (i \times 1)^* \RDerived(i \times i)_*  \operatorname{Cone}( \phi_U )[-1] \\
& \cong (i \times 1)^* \RDerived(i \times 1)_* \RDerived(1 \times i)_*  \operatorname{Cone}( \phi_U )[-1].
\end{align*}
Now $i \times 1$ is an open embedding and thence $(i \times 1)^* \RDerived(i \times 1)_*=\Id$, and $\cW_U \cong \operatorname{Cone}( \phi_U )[-1]$ by \eqref{localFMkernel}, so we have shown \eqref{equation.local_global_twist_kernels}, and the statement is proved. 
\end{proof}

We now build up our next main result in \ref{global flops via local methods} via a series of lemmas.  Recall that $\Omega\subseteq\Db(\coh X)$ is called a {\em spanning class} if
\begin{enumerate}
\item $\RHom_X(a,c)=0$ for all $c\in\Omega$ implies that $a=0$.
\item $\RHom_X(c,a)=0$ for all $c\in\Omega$ implies that $a=0$.
\end{enumerate}

To obtain a spanning class, usually one uses a Serre functor, which for us is $\mathbb{S}_X$ from \ref{Serre for X}.  However in our setting, since $E$ need not be perfect, obtaining a spanning class is a little delicate since the usual candidate $E\cup E^\perp$ will not do.  The solution is to use $\cE$, which we already know is perfect.

\begin{lemma}\label{Serre gives spanning class}
With the assumptions as in \S\ref{strategy}, define
\begin{align*}
\cE^\perp&:=\{ a\in\Db(\coh X)\mid \RHom(\cE,a)=0 \},\\
{}^{\perp}\cE&:=\{ a\in\Db(\coh X)\mid \RHom(a,\cE)=0 \}.
\end{align*}
Then $\cE^\perp={}^\perp\cE$, and $\Omega:=\cE\cup\cE^\perp$ is a spanning class in $\Db(\coh X)$.
\end{lemma} 
\begin{proof}
For the first statement, note that 
\[
\RHom_X(a,\cE)=0\stackrel{\scriptsize\mbox{\ref{EX tensor global}}}{\iff}\RHom_X(a,\cE\otimes\omega_X)=0 
\stackrel{\scriptsize\mbox{\ref{Serre for X}}}{\iff}
 \RHom_X(a,\mathbb{S}\cE)=0.
\]
On the other hand
\[
\RHom_X(\cE,a)=0\iff\RHom_{T}(\RHom_X(\cE,a),\omega_T)=0
\stackrel{\scriptsize\mbox{\ref{Serre def}}}{\iff}
 \RHom_X(a,\mathbb{S}\cE)=0
\]
where $T=\mathbb{C}$.  Combining gives $\cE^\perp={}^\perp\cE$.

Now we check that  $\Omega:=\cE\cup\cE^\perp$ is a spanning class.  If $\RHom_X(a,c)=0$ for all $c\in\Omega$, then $a\in{}^{\perp}\cE$ and so by the above $a\in\cE^{\perp}$. Thus taking $c:=a$ gives $\RHom_X(a,a)=0$ and hence $a=0$.  Also, if $\RHom_X(c,a)=0$ for all $c\in\Omega$, in particular  $\RHom_X(\cE,a)=0$ and so $a\in\cE^\perp$.  As above, this gives $a=0$.
\end{proof}

The following is expected from the standard theory of Seidel--Thomas twists.  We note here that in our setting the proof follows immediately from the local model, bypassing subtleties with the octahedral axiom in the `usual' proof \cite[p124]{Toda}.

\begin{lemma}\label{goodbye octahedral}
$T_{\cE}^*(\cE)\cong\cE[2]$, $T_{\cE}^*(i_*E)\cong i_*E[2]$ and functorially $\Id|^{\phantom *}_{\cE^{\perp}}\cong T^*_{\cE}|^{\phantom *}_{\cE^{\perp}}$. 
\end{lemma}
\begin{proof}
The first and second assertions follow from our local mutation calculation: 
\[
T_{\cE}^*(\cE)\stackrel{\mbox{\scriptsize\eqref{inclusion line}}}{\cong} 
T_{\cE}^*(\Ri_*\cEU)\stackrel{\mbox{\scriptsize\ref{diagram commutes U to X}}}{\cong} 
\Ri_*T^*_{\cEU}(\cEU)\stackrel{\mbox{\scriptsize\ref{track contraction2FM}}}{\cong} 
\Ri_* \cEU[2]\stackrel{\mbox{\scriptsize\eqref{inclusion line}}}{\cong}  
\cE[2]. 
\]
\[
T^*_{\cE}(i_*E)\stackrel{\mbox{\scriptsize\eqref{inclusion line}}}{\cong} 
T^*_{\cE}(\Ri_*E)\stackrel{\mbox{\scriptsize\ref{diagram commutes U to X}}}{\cong} 
\Ri_*T^*_{\cEU}(E)\stackrel{\mbox{\scriptsize\ref{track contraction2FM}}}{\cong} 
\Ri_* E[2]\stackrel{\mbox{\scriptsize\eqref{inclusion line}}}{\cong}  
i_*E[2]. 
\]

For the last assertion, since any morphism between kernels induces a natural transformation between the resulting Fourier--Mukai functors, the morphism $\cW_X\to\cO_{{\Delta},X}$ in \ref{global twist defin} induces a natural transformation $\beta\colon T^*_{\cE}\to\Id$.  Since $\FM{\cQ_X}=\RDerived i_*\circ G_U^{\phantom{L}}\circ G_U^{\LA}\circ i^*$ by \eqref{equation action part of twist}, and we have
\[
G_U^{\LA}(i^*a)
\stackrel{\scriptsize\mbox{\ref{adjoints Db compose lemma}}}{\cong}\RHom_{\CA}(\RHom_{U}(i^*a,\cEU), \CA)
\stackrel{\scriptsize\mbox{\eqref{inclusion line}}}{\cong}\RHom_{\CA}(\RHom_{X}(a,\cE), \CA)
\]
which is zero for all  $a\in{}^{\perp}\cE$, we see that $\beta_a\colon T^*_{\cE}(a)\to a$ is an isomorphism for all $a\in{}^{\perp}\cE$.  Thus functorially $\Id|^{\phantom *}_{{}^{\perp}\cE}\cong T^*_{\cE}|^{\phantom *}_{{}^{\perp}\cE}$.  Since $\cE^\perp={}^\perp\cE$ by 
\ref{Serre gives spanning class}, the result follows.
\end{proof}

We are almost ready to prove our main result.  The key point below is that we check that the functor is full and faithful using the spanning class from \ref{Serre gives spanning class}, where the hard case (case 1 in the proof of \ref{global flops via local methods} below) is already taken care of by the known derived equivalence on the local model (\S\ref{alg NC twists section}).  Once we know our functor is fully faithful, the problem then is that the usual Serre functor trick to establish equivalence does not work in our possibly singular setting. However we are able to bypass this with the following.

\begin{lemma}\label{equiv trick}
Let $\cC$ be a triangulated category, and $F\colon \cC\to\cC$ an exact fully faithful functor with right adjoint $F^{\RA}$.  Suppose that there exists an object $c\in\cC$ such that $F(c)\cong c[i]$ for some $i$, and further $F(x)\cong x$ for all $x\in c^\perp$.  Then $F$ is an equivalence. 
\end{lemma}
\begin{proof}
Since $F$ is fully faithful, by \cite[1.24]{HuybrechtsFM} we have
\begin{eqnarray}
\Id\xrightarrow{\sim} F^{\RA}\circ F.\label{nat isos}
\end{eqnarray}
Now we know that $F(a)\cong a$ for all $a\in c^\perp$, hence it follows from \eqref{nat isos} that $F^{\RA}(a)\cong a$ for all $a\in c^\perp$.

Now by \cite[1.50]{HuybrechtsFM}, $F$ is an equivalence provided that $F^{\RA}(x)=0$ implies that $x=0$.   Hence suppose that $F^{\RA}(x)=0$, then trivially \opt{10pt}{$\Hom_{\cC}(c,F^{\RA}(x)[j])=0$}\opt{12pt}{\[\Hom_{\cC}(c,F^{\RA}(x)[j])=0\]} for all $j\in\mathbb{Z}$.  By adjunction $\Hom_{\cC}(F(c),x[j])=0$ for all $j\in\mathbb{Z}$, so since $F(c)\cong c[i]$, we have $\Hom_{\cC}(c,x[j])=0$ for all $j\in\mathbb{Z}$, i.e.\  $x\in c^{\perp}$.  But then $F^{\RA}(x)=x$, so $x=0$, and we are done.
\end{proof}

\begin{thm}\label{global flops via local methods}
With the assumptions as in \S\ref{strategy}, the inverse  twist 
\[
T^*_\cE\colon\Db(\coh X)\to\Db(\coh X)\]
is an equivalence.
\end{thm}
\begin{proof}
We know that $T^*_{\cE}\colon\Db(\coh X)\to\Db(\coh X)$ by \ref{funct triangles global}, and has left and right adjoints that preserve $\Db(\coh X)$ by \ref{our NC twist has adjoints}.  Combining \ref{goodbye octahedral} with \ref{equiv trick}, it suffices to show that $T^*_{\cE}$ is fully faithful. 

We know from \ref{Serre gives spanning class} that $\Omega=\cE\cup\cE^\perp$ is a spanning class, so by  \cite[1.49]{HuybrechtsFM} we just need to check that 
\begin{eqnarray}
T^*_{\cE}\colon\Hom_{\D(X)}(\cG_1,\cG_2[i])\to\Hom_{\D(X)}(T^*_{\cE}(\cG_1),T^*_{\cE}(\cG_2)[i])\label{bij map}
\end{eqnarray}
is a bijection for all $\cG_1,\cG_2\in\Omega$ and all $i\in\mathbb{Z}$.  This is taken care of in the following four cases.\\
\emph{Case 1:} $\cG_1=\cG_2=\cE$.  Since $\Ri_* \cEU=\cE$, by \ref{diagram commutes U to X} we have a commutative diagram
 \[
\begin{array}{c}
\opt{10pt}{\begin{tikzpicture}}
\opt{12pt}{\begin{tikzpicture}[scale=1.2]}
\node (a1) at (0,0) {$\Hom_{\D(U)}(\cEU,\cEU[i])$};
\node (a2) at (5.5,0) {$\Hom_{\D(X)}(\cE,\cE[i])$};
\node (b1) at (0,-1.5) {$\Hom_{\D(U)}(T^*_{\cEU}(\cEU),T^*_{\cEU}(\cEU)[i])$};
\node (b2) at (5.5,-1.5) {$\Hom_{\D(X)}(T^*_{\cE}(\cE),T^*_{\cE}(\cE)[i])$};
\draw[->] (a1) to node[above] {$\scriptstyle \Ri_*$} node[below] {$\scriptstyle \sim$} (a2);
\draw[->] (b1) to node[above] {$\scriptstyle \Ri_*$}  node[below] {$\scriptstyle \sim$} (b2);
\draw[->] (a1) to node[left] {$\scriptstyle T^*_{\cEU}$} (b1);
\draw[->] (a2) to node[right] {$\scriptstyle T^*_{\cE}$} (b2);
\end{tikzpicture}
\end{array}
\]
We already know that the left-hand map is a bijection since $T^*_{\cEU}$ is an equivalence, hence the right-hand map is also a bijection.\\
\emph{Case 2:} $\cG_1=\cE$, $\cG_2\in\cE^{\perp}$.  The left-hand side of \eqref{bij map} is $H^i\RHom_X(\cE,\cG_2)$, which is zero.  By \ref{goodbye octahedral}, the right-hand side of \eqref{bij map} is $\Hom(\cE[2],\cG_2[i])=H^{i-2}\RHom_X(\cE,\cG_2)$, which also equals zero.\\
\emph{Case 3:} $\cG_1\in\cE^{\perp}$, $\cG_2=\cE$.  Since $\cE^\perp={}^\perp\cE$ by \ref{Serre gives spanning class},  we obtain $\RHom_X(\cG_1,\cE)=0$.  As in Case 2, this implies that both sides of \eqref{bij map} are zero.\\
\emph{Case 4:} $\cG_1,\cG_2\in\cE^{\perp}$. Since functorially $\Id|^{\phantom *}_{\cE^{\perp}}\cong T^*_{\cE}|^{\phantom *}_{\cE^{\perp}}$ by \ref{goodbye octahedral},  \eqref{bij map} is bijective for all $\cG_1,\cG_2\in\cE^{\perp}$.
\end{proof}

\subsection{The global noncommutative twist}\label{newsubsection} In this subsection we define the noncommutative twist functor, and show it sits in the expected functorial triangle.  We consider the triangle of Fourier--Mukai kernels
\begin{eqnarray}
\cW_X \to \cO_{\Delta,X}  \xrightarrow{\phi_X} \cQ_X\to\label{7.3 triangle}
\end{eqnarray} 
from \ref{global twist defin}, which by \ref{funct triangles global} give the Fourier--Mukai functors $T^*_{\cE}$, $\Id$ and $G_X^{\phantom{L}}\circ G_X^{\LA}$ respectively, which all preserve $\Db(\coh X)$.    Set $D:=(-)^\vee\otimes p_1^*\omega_X[\dim X]$.  By \ref{existence of adjoints thm}\eqref{existence of adjoints thm 3} and \ref{our NC twist has adjoints}\eqref{our NC twist has adjoints 2}, applying $D$ to \eqref{7.3 triangle} induces the right adjoints, which also preserve $\Db(\coh X)$.

\begin{defin}\label{NCTwist finally}
We define the \emph{noncommutative twist} $T_{\cE}\colon\Db(\coh X)\to \Db(\coh X)$ to be $\FM{D(\cW_X)}$.
\end{defin}

\begin{prop}\label{prop twist functorial triangle}
With the assumptions as in \S\ref{strategy},
\begin{enumerate}
\item\label{prop twist functorial triangle 1} $T_{\cE}$ is the inverse functor to $T^*_{\cE}$. In particular, $T_{\cE}$ is an autoequivalence.
\item\label{prop twist functorial triangle 2} There is a functorial triangle
\[
G_X^{\phantom{R}}\circ G_X^{\RA} \to \Id \to T_{\cE}\to 
\]
on $\Db(\coh X)$, where $G_X^{\phantom{R}}\circ G_X^{\RA} \cong \RHom_X(\cE,-)\otimes_{\CA}^{\bf L}\cE$.
\end{enumerate}
\end{prop}
\begin{proof}
(1) Since $T^*_{\cE}$ is an equivalence, its right adjoint $T_{\cE}$ is its inverse.\\
(2) By \ref{existence of adjoints thm}, the triangle
\begin{eqnarray*}
D(\cQ_X)\to D(\cO_{\Delta,X})\to D(\cW_X)\to
\end{eqnarray*}
gives the right adjoints.  The Fourier--Mukai functor of $D(\cQ_X)$ gives the right adjoint to $G_X^{\phantom{L}}\circ G_X^{\LA}$, namely $G_X^{\phantom{R}}\circ G_X^{\RA}$, which by \eqref{thelastlabel!} is isomorphic to $\RHom_X(\cE,-)\otimes^{\bf L}_{\CA}\cE$.  The Fourier--Mukai functor of $D(\cO_{\Delta,X})$ gives the right adjoint to the identity, which is the identity.  By definition, the Fourier--Mukai functor of $D(\cW_X)$ gives $T_{\cE}$. The result follows.
\end{proof}

\subsection{Relation with flop--flop functors} Finally, in this subsection we show that the noncommutative twist $T_{\cE}$ is inverse to an autoequivalence naturally associated with the flop, namely the flop--flop functor $\flopflop$. We thence obtain an intrinsic description of  $\flopflop$ in terms of the noncommutative deformation theory of the curve, without assuming that the flop exists.

Recall that the category $\Per X$ of perverse sheaves is defined as
\begin{equation}\label{eqn per X}
\Per X:=\left\{ a\in\Db(\coh X)\left| \begin{array}{c}H^i(a)=0\mbox{ if }i\neq 0,-1\\
f_*H^{-1}(a)=0\mbox{, }\Rfi{1}_* H^0(a)=0\\ \Hom(c,H^{-1}(a))=0\mbox{ for all }c\in\cC \end{array}\right. \right\},
\end{equation}
with $\cC:=\{ c\in\coh X\mid \Rf_*c=0\}$.

The following two straightforward lemmas, \ref{twist lemma} and \ref{flopflop lemma}, record the properties of $T_{\cE}$ and $\flopflop$ that we will need. In \ref{psi num equiv lemma}, we deduce that the functor $\TwistVsFlop := T_{\cE} \circ \flopflop$ preserves the category $\Per X$ above, and preserves the numerical equivalence class of skyscraper sheaves. The proof that $T_{\cE}$ is inverse to $\flopflop$ then concludes in \ref{twist flopflop}, following an argument given by Toda.

\begin{lemma}\label{twist lemma}
The twist functor $T_{\cE}$ has the following properties:
\begin{enumerate}
\item\label{twist lemma 1} $\Rf_* \circ \, T_{\cE} \cong \Rf_*$.
\item\label{twist lemma 2} $T_{\cE} (\cO_X) \cong \cO_X$.
\item\label{twist lemma 3} $T_{\cE} (E) \cong E[-2] $.
\end{enumerate}
\end{lemma}
\begin{proof}Property \eqref{twist lemma 1} may be deduced by applying $\Rf_*$ to the triangle representation of $T_{\cE}$ in \ref{prop twist functorial triangle}\eqref{prop twist functorial triangle 2}, using the fact that $\Rf_* \circ \, G_X \cong 0$, which we prove now.  The commutative diagram
\[
\begin{array}{c}
\begin{tikzpicture}
\node (U) at (1,0) {$U$};
\node (X) at (3.5,0) {$X$};
\draw[right hook->] (U) to node[above] {$\scriptstyle i$} (X);
\node (Uc) at (1,-1.5) {$\Spec R$};
\node (Xc) at (3.5,-1.5) {$X_{\con}$};
\draw[right hook->] (Uc) to node[above] {$\scriptstyle i_{\con}$} (Xc);
\draw[->] (X) --  node[right] {$\scriptstyle f$} (Xc);
\draw[->] (U) --  node[right] {$\scriptstyle f$}  (Uc);
\end{tikzpicture}
\end{array}
\]
gives
\begin{align*}
\Rf_* \circ \, G_X & = \Rf_* \Ri_* ( - \otimes_{\CA}^{\bf L} \cEU ) \\
& \cong \Ri_{\con *} \Rf_* ( - \otimes_{\CA}^{\bf L} \cEU )\\
& \cong \Ri_{\con *} ( - \otimes_{\CA}^{\bf L} \Rf_*\cEU ).
\end{align*}
As $\cEU$ is filtered by $E$, it follows that $\Rf_*\cEU \cong0$, and the property follows.

To obtain \eqref{twist lemma 2}, we note that $f^!\cO_{X_\con}=\cO_X$ since $f$ is crepant, so
\[
\RHom_X(\cE,\cO_X)=\RHom_X(\cE,f^!\cO_{X_\con})\cong\RHom_{X_\con}(\Rf_*\cE,\cO_{X_\con})=0,
\]
again using $\RDerived{f_{*}} \cE \cong 0$.  Hence  $T_{\cE} (\cO_X) \cong \cO_X$ by \ref{prop twist functorial triangle}\eqref{prop twist functorial triangle 2}.

The last property \eqref{twist lemma 3} follows from \ref{goodbye octahedral}.
\end{proof}

\begin{lemma}\label{flopflop lemma}
The flop--flop functor $\flopflop$ has the following properties:
\begin{enumerate}
\item\label{flopflop lemma 1} $\Rf_* \circ \, \flopflop \cong \Rf_*$.
\item\label{flopflop lemma 2} $\flopflop (\cO_X) \cong \cO_X$.
\item\label{flopflop lemma 3} $\flopflop (E) \cong E[2] $.
\end{enumerate}
\end{lemma}
\begin{proof}
We denote the flop by $X^+$. Property \eqref{flopflop lemma 1} follows immediately since the individual flop functors commute with the pushdown by \cite[(4.4)]{Bridgeland}.  Further, by \cite[(4.8)]{Bridgeland} noting the correction to the sign in \cite[(27)]{TodaWidth}, the flop functor induces an equivalence
\[
\Per X\to\mPer X^+.
\]
By morita theory, under this equivalence projectives must go to projectives, and simples must go to simples.  By Van den Bergh's description of these objects \cite[3.2.7, 3.5.7, 3.5.8]{VdB1d}, and the fact that the flop functor commutes with pushdown, it follows that the flop functor must take $\cO_X$ to $\cO_{X^+}$, and $E$ to $E[1]$.  Applying the flop functor again,  properties \eqref{flopflop lemma 2} and \eqref{flopflop lemma 3} follow.
\end{proof}

The following observations are used in the proof of \ref{twist flopflop}, where we will show that the functor $\TwistVsFlop$ defined below is in fact isomorphic to the identity.

\begin{prop}\label{psi num equiv lemma}
The functor $\TwistVsFlop := T_{\cE} \circ \flopflop$ is an autoequivalence of $\Db(\coh X)$, preserving the category $\Per X$ of perverse sheaves given in \eqref{eqn per X}. It also preserves the numerical equivalence class of skyscraper sheaves, that is for points $p\in X$ we have \[\chi(\cL,\cO_p)=\chi(\cL,\TwistVsFlop(\cO_p))\] for all locally free sheaves $\cL$ on $X$, where $\chi$ is the Euler characteristic. For definitions, see for instance \cite[7.61]{FMNahm}.
\end{prop}
\begin{proof}
Note that $\TwistVsFlop$ is an equivalence by construction, using \ref{prop twist functorial triangle}\eqref{prop twist functorial triangle 1}, and also that combining \ref{twist lemma} and \ref{flopflop lemma} gives
\begin{enumerate}
\item\label{psi num equiv lemma proof 1} $\Rf_* \circ \, \TwistVsFlop \cong \Rf_*$, 
\item\label{psi num equiv lemma proof 2} $\TwistVsFlop (E) \cong E$.
\end{enumerate}
We first argue that $\TwistVsFlop$ preserves the category $\Per X$ of perverse sheaves. As explained in \cite[Lemma 3.2]{Bridgeland}, this category is obtained as the intersection, for perversity $p=0$, of the subcategories
\begin{align*}
{}^p \cA^{\leq 0} & := \{ a\in\Db(\coh X)\mid \Rf_*a \in \Db(\coh X_{\con})^{\leq 0}, \: a \in {}^{\perp}(\cC^{>p}) \}, \\
{}^p \cA^{\geq 0} & := \{ a\in\Db(\coh X)\mid \Rf_*a \in \Db(\coh X_{\con})^{\geq 0}, \: a \in (\cC^{<p}){}^{\perp} \},
\end{align*}
where $\cC^{>p}:=\{ c\in\Db(\coh X)\mid \Rf_*c=0\mbox{, }H^i(c)=0\mbox{ for }i\leq p\}$ and $\cC^{<p}$ is defined similarly.  Now  the property \eqref{psi num equiv lemma proof 1} implies that $\TwistVsFlop^{-1}$ preserves $\{ c\in\Db(\coh X)\mid \Rf_*c=0\}$, and using property \eqref{psi num equiv lemma proof 2} it is easy to see that  $\TwistVsFlop^{-1}$ preserves the induced standard $t$-structure on $\{ c\in\Db(\coh X)\mid \Rf_*c=0\}$, and hence preserves $\cC^{>p}$ and $\cC^{<p}$.  From this, it follows that $\TwistVsFlop^{-1}$ preserves ${}^p \cA^{\leq 0}$ and ${}^p \cA^{\geq 0}$, and hence $\Per X$, as claimed.

Next we argue that $\TwistVsFlop$ preserves the numerical equivalence class of skyscraper sheaves. Taking a sheaf $\cO_p$ for a point $p\in X$, we consider two cases. If $p \not\in C$, then $f$ is an isomorphism at $p$. Property \eqref{psi num equiv lemma proof 1} above then gives that $\TwistVsFlop(\cO_p) \cong \cO_p$, and so there is nothing to prove in this case. If $p \in C$, then $\cO_p$ is no longer simple in $\Per X$ but it is filtered by the simples $\{\omega_C[1]\mbox{, }E\}$.  Thus we have a series of  exact triangles 
\begin{equation}
\begin{array}{c}
\begin{tikzpicture}
\node (b0) at (0,0) {$b_0$};
\node (a1) at (2,0) {$a_1$};
\node (a2) at (4,0) {$a_2$};
\node (dots) at (6,0) {$\hdots$};
\node (a3) at (8,0) {$a_{n-1}$};
\node (a4) at (10,0) {$\cO_p$};
\node (b1) at (1,-1) {$b_1$};
\node (b2) at (3,-1) {$b_2$};
\node (b4) at (9,-1) {$b_n$};
\draw[->] (b0) -- (a1);
\draw[->] (a1) -- (b1);
\draw[->,densely dotted] (b1) -- (b0);
\draw[->] (a1) -- (a2);
\draw[->] (a2) -- (b2);
\draw[->,densely dotted] (b2) -- (a1);
\draw[->] (a2) -- (dots);
\draw[->] (dots) -- (a3);
\draw[->] (a3) -- (a4);
\draw[->] (a4) -- (b4);
\draw[->,densely dotted] (b4) -- (a3);
\end{tikzpicture}
\end{array}\label{filt 1}
\end{equation}
where all $b_i\in \{\omega_C[1]\mbox{, }E\}$.  But $\TwistVsFlop$ is an equivalence preserving $\Per X$, so it must permute the simples.  Since we already know that $\TwistVsFlop$ fixes $E$ and $\cO_x$ for $x\in X\backslash C$, it follows that $\TwistVsFlop$ also fixes $\omega_C[1]$ and hence fixes all the simples.

Because of this, we also have a series of exact triangles
\begin{equation}
\begin{array}{c}
\begin{tikzpicture}
\node (b0) at (0,0) {$b_0$};
\node (a1) at (2,0) {$\TwistVsFlop a_1$};
\node (a2) at (4,0) {$\TwistVsFlop a_2$};
\node (dots) at (6,0) {$\hdots$};
\node (a3) at (8,0) {$ \TwistVsFlop a_{n-1}$};
\node (a4) at (10,0) {$\TwistVsFlop \cO_p$};
\node (b1) at (1,-1) {$b_1$};
\node (b2) at (3,-1) {$b_2$};
\node (b4) at (9,-1) {$b_n$};
\draw[->] (b0) -- (a1);
\draw[->] (a1) -- (b1);
\draw[->,densely dotted] (b1) -- (b0);
\draw[->] (a1) -- (a2);
\draw[->] (a2) -- (b2);
\draw[->,densely dotted] (b2) -- (a1);
\draw[->] (a2) -- (dots);
\draw[->] (dots) -- (a3);
\draw[->] (a3) -- (a4);
\draw[->] (a4) -- (b4);
\draw[->,densely dotted] (b4) -- (a3);
\end{tikzpicture}
\end{array}\label{filt 2}
\end{equation}
Using additivity of $\chi(\cL,-)$ on triangles, the left-hand triangles in \eqref{filt 1} and \eqref{filt 2} imply that $\chi(\cL,a_1)=\chi(\cL,\TwistVsFlop a_1)$.  Inducting to the right on \eqref{filt 1} and \eqref{filt 2} gives $\chi(\cL,\cO_p)=\chi(\cL,\TwistVsFlop \cO_p)$, as required.
\end{proof}

The following is the main result of this subsection.

\begin{prop}\label{twist flopflop}
$T_{\cE}$ is an inverse of $\flopflop$, the Bridgeland--Chen flop--flop functor of \S\ref{sect hom alg der auto}.
\end{prop}
\begin{proof}
We use the argument indicated in \cite[end of proof of 3.1]{Toda}, but write the proof in full to show that there is nothing particular about commutative deformations or smoothness used in {\em loc.\ cit.} that affects the proof. As above, consider the autoequivalence 
\begin{equation*}
\TwistVsFlop := T_{\cE} \circ \flopflop.
\end{equation*} 
Using \ref{twist lemma} and \ref{flopflop lemma} we know that $\TwistVsFlop(\cO_X)\cong\cO_X$, and by \ref{psi num equiv lemma} $\TwistVsFlop$ preserves the category $\Per X$ of perverse sheaves, and preserves the numerical equivalence class of skyscraper sheaves.

We now follow the argument in \cite[end of proof of 6.1]{TodaRes} to deduce that in fact $\TwistVsFlop\cong\Id$. We first show that for points $p\in X$, $\TwistVsFlop(\cO_p)$ is a skyscraper sheaf. Most of the work here is to show that $\TwistVsFlop(\cO_p)$ is a sheaf, as the result will then follow from the above remark on numerical equivalence classes. We take the short exact sequence of sheaves
\[ 
0 \to \cI_p \to \cO_X \to \cO_p \to 0, 
\]
and observe that it is also a short exact sequence in $\Per X$ (for this, we need only check that $\Rfi{1}_* = 0$ for each sheaf). Since $\TwistVsFlop$ preserves $\Per X$, and $\TwistVsFlop(\cO_X)\cong\cO_X$,
\[ 
0 \to \TwistVsFlop(\cI_p) \to \cO_X \to \TwistVsFlop(\cO_p) \to 0 \]
is also an exact sequence in $\Per X$.  The perversity gives that $H^1 \TwistVsFlop(\cI_p) = 0$, and so the associated long exact sequence yields a surjection 
\begin{equation}\label{eqn crucial surjection}
\cO_X \to H^0\TwistVsFlop(\cO_p).
\end{equation}
We claim that the target of this surjection must be non-trivial.  Otherwise, the perversity implies that $\TwistVsFlop(\cO_p)=H^{-1}\TwistVsFlop(\cO_p)[1]$, which then implies that for $\cL$ suitably anti-ample
\[
1= \chi(\cL,\cO_p) \stackrel{\mbox{\scriptsize\ref{psi num equiv lemma}}}{=} \chi(\cL,\TwistVsFlop(\cO_p))=-\chi(\cL,H^{-1}\TwistVsFlop(\cO_p))\leq 0,
\]
which is a contradiction.  Hence the map \eqref{eqn crucial surjection} above is non-trivial. Furthermore, it is adjoint to a map $\cO_{X_{\con}} \to f_* H^0\TwistVsFlop(\cO_p)$, which must consequently also be non-trivial, and so we find that \begin{equation}\label{eqn skyscraper image non-vanishing}f_* H^0\TwistVsFlop(\cO_p) \neq 0.\end{equation}

Now the fibres of $f$ have dimension at most $1$, so $\Rfi{>1}_*=0$ and therefore the Grothendieck spectral sequence
\[E^{p,q}_2 = \Rfi{p}_*(H^q \TwistVsFlop (-)) \to \RDerivedi{p+q}{(f_* \circ \TwistVsFlop)}(-) \cong \Rfi{p+q}_*(-)\]
collapses.  Hence the exact sequence of low-degree terms applied to $\cO_p$ gives a short exact sequence
\[ 
0 \to \Rfi{1}_* H^{-1} \TwistVsFlop(\cO_p) \to f_*(\cO_p) \to f_* H^{0} \TwistVsFlop(\cO_p) \to 0.
\]
The middle term is just $\cO_{f(p)}$. This surjects onto the last term, which we know by \eqref{eqn skyscraper image non-vanishing} to be non-zero.  We deduce that $\Rfi{1}_* H^{-1} \TwistVsFlop(\cO_p) = 0.$ Recall that $\TwistVsFlop(\cO_p) \in \Per X$, and hence $f_* H^{-1} \TwistVsFlop(\cO_p) = 0$ also, so we conclude that \[\Rf_* H^{-1} \TwistVsFlop(\cO_p) = 0,\] so that $H^{-1} \TwistVsFlop(\cO_p) \in \cC$, the category used in the definition of $\Per X$. Using once more the fact that $\TwistVsFlop(\cO_p) \in \Per X$, we see from the definition that \[\Hom_X( H^{-1} \TwistVsFlop(\cO_p), H^{-1} \TwistVsFlop(\cO_p) ) = 0,\] and thence that $H^{-1} \TwistVsFlop(\cO_p) =0,$ and that $\TwistVsFlop(\cO_p)$ is a sheaf. This sheaf is numerically equal to a skyscraper sheaf by \ref{psi num equiv lemma}, and thence by \cite[7.62(2)]{FMNahm} is itself a skyscraper sheaf.

Finally, we therefore have that $\TwistVsFlop(\cO_p) \cong \cO_{\phi(p)}$ for all $p$ and some function $\phi$. We deduce that $\phi$ is an automorphism because $\TwistVsFlop$ is an equivalence. Therefore \[\TwistVsFlop \cong \cL \otimes \phi_*(-)\] for some line bundle $\cL$, which must be trivial as $\TwistVsFlop(\cO_X) \cong \cO_X$. We must have $f \circ \phi = f$ because $\TwistVsFlop$ commutes with the contraction, and so $\phi = \Id$. This completes the proof.
\end{proof}

\appendix
\newpage
\section{List of Notations}
\label{list of notations}
\opt{10pt}{\begin{longtable}{p{2.7cm} p{9cm} p{0.5cm}}}
\opt{12pt}{\begin{longtable}{p{3cm} p{8.5cm} p{0.5cm}}}
$\eta$, $\varepsilon$ & unit and counit morphisms & \\
$(-)^{\ab}$ & abelianization & \\
$(-)^{\rm e}$ & enveloping algebra & \\
$\pd$ & projective dimension & \\
$\Mod A$ ($\mod A$) & (finitely generated) right $A$-modules \\
$\add M$ & summands of finite sums of module $M$ \\
$\CM R$ & category of maximal Cohen--Macaulay\opt{12pt}{\\&} $R$-modules \\

$\cart$ & category of local commutative artinian\opt{12pt}{\\&} $\K$-algebras \\

$\D$ ($\Db$) & derived category (bounded) \\ 
$\D^-$, $\D^+$ & derived category bounded above, and below \\ 
$\FM{\cW}$ & Fourier--Mukai functor with kernel $\cW$ \\
$F^{\RA}$, $F^{\LA}$ & right and left adjoints, for a functor $F$\\ 

$\flop$, $\flopback$ & flop functors & \S\ref{sect hom alg der auto} \\
$\flopflop$ & flop--flop functor & \S\ref{sect hom alg der auto}\\
$\alg_1$ ($\calg_1$) & category of $1$-pointed (commutative) \opt{10pt}{$\K$-algebras} & \ref{defin.n-pointed_non-comm_K-alg} \\
\opt{12pt}{& $\K$-algebras \\}
$\art_1$ ($\cart_1$) & category of $1$-pointed artinian (commutative) $\K$-algebras & \ref{defin.n-pointed_non-comm_K-alg} \\
$\Gamma$ & $1$-pointed artinian $\K$-algebra & \ref{defin.n-pointed_non-comm_K-alg} \\
$a,b\in\cA$ & objects of abelian category $\cA$ & \S\ref{NCdef section} \\
$\PairsCat{\cA}{\Gamma}$ & $\Gamma$-objects in $\cA$ & \ref{rem pairs cat}\\
$\Def^{\cA}_{a}$ ($\cDef^{\cA}_a$) & noncommutative (commutative) deformation functor & \ref{NC deformation definition} \\
$C\subset X$ & curve $C$ in a $3$-fold $X$ & \S\ref{setup section 2.1} \\
$C^{\redu}$ & reduced subscheme associated to $C$ ($\cong \mathbb{P}^1$) & \S\ref{setup section 2.1} \\
$E$ & sheaf $\cO_{\mathbb{P}^1}(-1)$ supported on $C^{\redu}$ & \S\ref{setup section 2.1} \\
$p \in X_{\con}$ & singular point $p$ in contracted $3$-fold $X_{\con}$ & \S\ref{setup section 2.1} \\
$f$ & contraction morphism & \S\ref{setup section 2.1} \\
$U_{\con}$ & affine open neighbourhood of singular point $p$ & \S\ref{setup section 2.1} \\
$R$ & $\mathbb{C}$-algebra, with $U_{\con} \cong \Spec R$ & \S\ref{setup section 2.1} \\
$U$ & $:=f^{-1}(U_{\con})$, open neighbourhood of curve $C$ & \S\ref{setup section 2.1} \\
$i$ & open embedding of $U$ in $X$ & \S\ref{setup section 2.1} \\
$\cV$ & $:=\cO_U\oplus\cN$, tilting bundle on $U$ & \S\ref{sect geom to alg} \\
$N$ & $:=f_*\cN$ & \S\ref{sect geom to alg} \\
$T$ & $:=\Hom_U(\cV,E)$ & \S\ref{sect geom to alg} \\
$\Per U$ & perverse sheaves on $U$ & \S\ref{sect geom to alg} \\
$\Lambda$ & endomorphism algebra $\End_U(\cV)$ & \S\ref{sect geom to alg} \\
$\Lambda_{\con}$ & contraction algebra associated to $\Lambda$ & \ref{contraction def lambda} \\
$I_{\con}$ & ideal associated to $\Lambda_{\con}$ & \ref{contraction def lambda} \\
$[R]$ & ideal of morphisms of $\Lambda$ which factor \opt{10pt}{through $\add R$} & \ref{contraction def lambda} \\
\opt{12pt}{& through $\add R$ \\}
$\widehat{U}$ & formal fibre of $f$ over singular point $p$ & \S\ref{complete local geometry summary} \\ 
$\widehat{R}$, $\mathfrak{R}$ & completion of $R$ at $p$ & \S\ref{complete local geometry summary} \\ 
$\widehat{\Lambda}$ & completion of $\Lambda$ & \S\ref{complete local geometry summary} \\
$\cO_{\widehat{U}}\oplus \cN_1$ & tilting bundle on $\widehat{U}$ & \S\ref{complete local geometry summary} \\ 
$N_1$& $:=f_*\cN_1$ & \S\ref{complete local geometry summary} \\
$\mathbb{F}$ & morita equivalence & \S\ref{complete local geometry summary} \\
$\AB$ & basic algebra, morita equivalent to $\widehat{\Lambda}$ & \ref{def basic algebra} \\
$\CA$ & contraction algebra & \ref{def basic algebra} \\
$\cE$, $\cEU$ & noncommutative universal sheaf on $X$, \opt{10pt}{and on $U$} & \S\ref{sect chasing rep couple} \\
\opt{12pt}{& and on $U$ \\}
$\cF$, $\cFU$ & commutative universal sheaf on $X$, and on $U$ & \S\ref{sect chasing rep couple} \\
$\DA$ & $:=\End_X(\cE)$, noncommutative deformation \opt{10pt}{algebra} & \ref{contraction def} \\
\opt{12pt}{& algebra \\}
$\width(C)$, $\cwidth(C)$ & width, and commutative width, of $C$ & \ref{width def} \\
$\upmu$, $\upnu$& right and left mutation & \ref{setup2} \\
$\Phi$ & mutation functor & \S\ref{sect on the mutation functor} \\
$\mathbb{S}$ & Serre functor & \ref{Serre def for alg} \\
$T_{\con}$, $T^*_{\con}$ & algebraic twist, and inverse & \ref{def twist alg local} \\
$T_{\cEU}$, $T^*_{\cEU}$ & noncommutative twist on $U$, and inverse & \ref{local twist geom defin} \\
$\cO_Y^\Gamma$ & $:=\cO_Y\otimes_{\underline{\mathbb{C}}}\underline{\Gamma}$ & \ref{notn O-Gamma modules} \\
$\cW_U$, $\cQ_U$ & FM kernel for $T^*_{\cEU}$, and related kernel & \S\ref{section global inverse twist} \\
$T^*_{\cE}$ & inverse twist on $X$ & \ref{global twist defin} \\
$\cW_X$, $\cQ_X$ & FM kernel for $T^*_{\cE}$, and related kernel & \ref{global twist defin} \\
$T_{\cE}$ & noncommutative twist on $X$ & \ref{NCTwist finally}
\end{longtable}

\end{document}